\documentclass[11pt]{article}
\usepackage{authblk}
\usepackage[toc,page]{appendix}
\usepackage[top=2cm, bottom=2cm, left=2cm, right=2cm]{geometry}

\usepackage{color}
\usepackage{helvet}         
\usepackage{courier}        
\usepackage{type1cm}        

\usepackage{framed} 
\usepackage{tikz}
\usepackage{makeidx}         
\usepackage{graphicx}        
\usepackage{multicol}        
\usepackage[bottom]{footmisc}

\usepackage{amsmath}
\usepackage{amssymb}
\usepackage{bbold}
\usepackage{amsthm}
\usepackage{subcaption}
\usepackage{sidecap}
\usepackage{floatrow}
\usepackage{pdflscape}
\usepackage{comment}
\usepackage[font=small]{caption}
\usepackage{enumitem}
\usepackage{esint}

\usepackage[hidelinks]{hyperref}

\usepackage{scalerel}
\usepackage{shuffle}
\usepackage{mathrsfs}
\usepackage{mathtools}

\mathtoolsset{showonlyrefs,showmanualtags}

\newtheorem{theorem}{Theorem}
\newtheorem{corollary}[theorem]{Corollary}
\newtheorem{lemma}[theorem]{Lemma}

\newtheorem{proposition}[theorem]{Proposition}

\theoremstyle{definition}

\theoremstyle{remark}

\newtheorem{remark}[theorem]{\bf Remark}

\numberwithin{theorem}{section}
\numberwithin{figure}{section}
\numberwithin{equation}{section}

\begin{document}
\title{Classification of commutation relation for multi-chordal SLE}
\date{}
\bigskip{}
\author[1]{Mo Chen\thanks{chenmothuprobab@gmail.com. ORCID:0009-0000-6574-4308}}
\author[1]{Chongzhi Huang\thanks{huangchzh2001prob@gmail.com. ORCID:0009-0002-7991-6267}}
\author[1]{Hao Wu\thanks{hao.wu.proba@gmail.com. ORCID:0000-0003-4265-7417}}
\affil[1]{Tsinghua University, China}


%
%


\newcommand{\Prad}[1]{\mathsf{P}_{#1\mathrm{\textnormal{-}rad}}}
\newcommand{\Erad}[1]{\mathsf{E}_{#1\mathrm{\textnormal{-}rad}}}
\newcommand{\LZrad}[1]{\mathcal{Z}_{#1\mathrm{\textnormal{-}rad}}}
\newcommand{\LSrad}[1]{\mathsf{S}_{#1\mathrm{\textnormal{-}rad}}}
\newcommand{\LFrad}[1]{\mathcal{F}_{#1\mathrm{\textnormal{-}rad}}}

\newcommand{\Pchord}{\mathsf{P}}
\newcommand{\Pchordtwo}{\Pchord_{\includegraphics[scale=0.15]{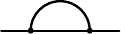}}}
\newcommand{\Echord}{\mathsf{E}}
\newcommand{\Echordtwo}{\mathsf{E}_{\includegraphics[scale=0.15]{figures/link-0}}}
\newcommand{\LSchord}[1]{\mathsf{S}_{#1\mathrm{\textnormal{-}chord}}}
\newcommand{\LSchordhat}[1]{\widehat{\mathsf{S}}_{#1\mathrm{\textnormal{-}chord}}}
\newcommand{\LFchord}[1]{\mathcal{F}_{#1\mathrm{\textnormal{-}chord}}^{(\lambda;\sigma)}}
\newcommand{\LGchord}[1]{\mathcal{G}_{#1\mathrm{\textnormal{-}chord}}^{(\kappa;\lambda;\sigma)}}
\newcommand{\cstchord}[1]{\mathrm{C}_{#1\mathrm{\textnormal{-}chord}}^{(\kappa;\lambda;\sigma)}}
\newcommand{\LGchordGFF}[1]{\mathcal{G}_{#1\mathrm{\textnormal{-}chord}}^{(4;\lambda;\sigma)}}
\newcommand{\cstchordGFF}[1]{\mathrm{C}_{#1\mathrm{\textnormal{-}chord}}^{(4;\lambda;\sigma)}}
\newcommand{\integrand}[1]{\Phi_{#1\mathrm{\textnormal{-}chord}}^{(\kappa;\lambda;\mathfrak{r})}}
\newcommand{\integrandzero}[1]{\Phi_{#1\mathrm{\textnormal{-}chord}}^{(\kappa;0;\mathfrak{r})}}

\newcommand{\blm}{\mathfrak{m}}
\newcommand{\LZmix}{\mathcal{Z}_{\bs{s};\alpha}^{(p)}}
\newcommand{\LZmixtilde}{\widetilde{\mathcal{Z}}_{\bs{s};\alpha}^{(p)}}
\newcommand{\Cmix}{C_{\bs{s};\alpha}^{(p)}}
\newcommand{\LZmixp}{\mathcal{Z}_{\bs{s};\alpha}^{(p; \mu_p)}}
\newcommand{\Cmixp}{C_{\bs{s};\alpha}^{(p; \mu_p)}}
\newcommand{\LZmixpminus}{\mathcal{Z}_{\bs{s};\alpha}^{(p; -\mu_p)}}
\newcommand{\Cmixpminus}{C_{\bs{s};\alpha}^{(p; -\mu_p)}}
\newcommand{\LZmixhat}{\mathcal{Z}_{\bs{s};\hat{\alpha}}^{(p)}}
\newcommand{\LZalphaCR}{\mathcal{Z}_{\alpha{\textnormal{-}}\mathrm{cr}}^{(\ell;\mu)}}
\newcommand{\CalphaCR}{C_{\alpha{\textnormal{-}}\mathrm{cr}}^{(\ell;\mu)}}
\newcommand{\LZalphaCRhat}{\mathcal{Z}_{\hat{\alpha}{\textnormal{-}}\mathrm{cr}}^{(\ell;\mu)}}
\newcommand{\LZalphahcap}{\mathcal{Z}_{\alpha{\textnormal{-}}\hcap}^{(\lambda)}}
\newcommand{\LZalphahcaphat}{\mathcal{Z}_{\hat{\alpha}{\textnormal{-}}\hcap}^{(\lambda)}}

\global\long\def\U{\mathbb{U}}
\global\long\def\T{\mathbb{T}}
\global\long\def\HH{\mathbb{H}}
\global\long\def\R{\mathbb{R}}
\global\long\def\C{\mathbb{C}}
\global\long\def\N{\mathbb{N}}
\global\long\def\Z{\mathbb{Z}}
\global\long\def\E{\mathbb{E}}
\global\long\def\PP{\mathbb{P}}
\global\long\def\QQ{\mathbb{Q}}
\global\long\def\A{\mathbb{A}}
\global\long\def\one{\mathbb{1}}

\global\long\def\CR{\mathrm{CR}}
\global\long\def\ST{\mathrm{ST}}
\global\long\def\SF{\mathrm{SF}}
\global\long\def\cov{\mathrm{cov}}
\global\long\def\dist{\mathrm{dist}}
\global\long\def\SLE{\mathrm{SLE}}
\global\long\def\hSLE{\mathrm{hSLE}}
\global\long\def\CLE{\mathrm{CLE}}
\global\long\def\GFF{\mathrm{GFF}}
\global\long\def\inte{\mathrm{int}}
\global\long\def\ext{\mathrm{ext}}
\global\long\def\inrad{\mathrm{inrad}}
\global\long\def\outrad{\mathrm{outrad}}
\global\long\def\dimH{\mathrm{dim}}
\global\long\def\capa{\mathrm{cap}}
\global\long\def\diam{\mathrm{diam}}
\global\long\def\sign{\mathrm{sgn}}
\global\long\def\cat{\mathrm{Cat}}
\global\long\def\cst{\mathrm{C}}
\global\long\def\ck{\mathrm{C}_{\kappa}}
\global\long\def\free{\mathrm{free}}
\global\long\def\hF{{}_2\mathrm{F}_1}
\global\long\def\simple{\mathrm{simple}}
\global\long\def\even{\mathrm{even}}
\global\long\def\odd{\mathrm{odd}}
\global\long\def\st{\mathrm{ST}}
\global\long\def\usf{\mathrm{USF}}
\global\long\def\Leb{\mathrm{Leb}}
\global\long\def\LP{\mathrm{LP}}
\global\long\def\I{\mathrm{I}}
\global\long\def\K{\mathrm{K}}
\global\long\def\J{\mathrm{J}}
\global\long\def\Y{\mathrm{Y}}
\global\long\def\Cat{\mathrm{Cat}}

\global\long\def\II{\mathrm{II}}
\global\long\def\hcap{\mathrm{hcap}}
\global\long\def\const{\mathrm{C}}
\global\long\def\Poisson{\mathrm{H}}
\global\long\def\Green{\mathrm{G}}
\global\long\def\Rat{\mathrm{Rat}}

\global\long\def\LA{\mathcal{A}}
\global\long\def\LB{\mathcal{B}}
\global\long\def\LC{\mathcal{C}}
\global\long\def\LD{\mathcal{D}}
\global\long\def\LF{\mathcal{F}}
\global\long\def\LK{\mathcal{K}}
\global\long\def\LE{\mathcal{E}}
\global\long\def\LG{\mathcal{G}}
\global\long\def\LI{\mathcal{I}}
\global\long\def\LJ{\mathcal{J}}
\global\long\def\LL{\mathcal{L}}
\global\long\def\LM{\mathcal{M}}
\global\long\def\LN{\mathcal{N}}
\global\long\def\OO{\mathcal{O}}
\global\long\def\LQ{\mathcal{Q}}
\global\long\def\LR{\mathcal{R}}
\global\long\def\LT{\mathcal{T}}
\global\long\def\LS{\mathcal{S}}
\global\long\def\LU{\mathcal{U}}
\global\long\def\LV{\mathcal{V}}
\global\long\def\LW{\mathcal{W}}
\global\long\def\LX{\mathcal{X}}
\global\long\def\LY{\mathcal{Y}}
\global\long\def\PartF{\mathcal{Z}}
\global\long\def\LH{\mathcal{H}}
\global\long\def\LJ{\mathcal{J}}

\global\long\def\blm{\mathfrak{m}}

\global\long\def\LZ{\mathcal{Z}}
\global\long\def\LZrp{\mathcal{Z}_{\alpha; \bs{s}}^{(p)}}
\global\long\def\LJrp{\mathcal{J}_{\alpha; \bs{s}}^{(p)}}
\global\long\def\chamberrp{\chamber_{\alpha; \bs{s}}^{(p)}}
\global\long\def\LErp{\mathcal{E}_{\alpha; \bs{s}}^{(p)}}
\global\long\def\Greenrp{G_{\alpha; \bs{s}}^{(p)}}
\global\long\def\Prp{P_{\alpha; \bs{s}}^{(p)}}
\global\long\def\norcst{\mathrm{C}_{\kappa}^{(\mathfrak{r})}}
\global\long\def\LZalphar{\mathcal{Z}_{\alpha}^{(\mathfrak{r})}}

\newcommand{\LZtwo}{\mathcal{Z}_{\includegraphics[scale=0.15]{figures/link-0}}}
\newcommand{\LEtwo}{\mathcal{E}_{\includegraphics[scale=0.15]{figures/link-0}}}
\newcommand{\LItwo}{\mathcal{I}_{\includegraphics[scale=0.15]{figures/link-0}}}
\newcommand{\LUtwo}{\mathcal{U}_{\includegraphics[scale=0.15]{figures/link-0}}}
\newcommand{\LZtwor}{\LZtwo^{(\mathfrak{r})}}
\newcommand{\LHtwo}{\mathcal{H}_{\includegraphics[scale=0.15]{figures/link-0}}}
\newcommand{\LFtwo}{\mathcal{F}_{\includegraphics[scale=0.15]{figures/link-0}}}

\newcommand{\LZfoura}{\mathcal{Z}_{\includegraphics[scale=0.15]{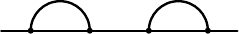}}}
\newcommand{\LZfourb}{\mathcal{Z}_{\includegraphics[scale=0.15]{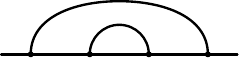}}}
\newcommand{\LHfoura}{\mathcal{H}_{\includegraphics[scale=0.15]{figures/link-1}}}
\newcommand{\LHfourb}{\mathcal{H}_{\includegraphics[scale=0.15]{figures/link-2}}}
\newcommand{\LFfoura}{\mathcal{F}_{\includegraphics[scale=0.15]{figures/link-1}}}
\newcommand{\LFfourb}{\mathcal{F}_{\includegraphics[scale=0.15]{figures/link-2}}}
\newcommand{\LFfouraRenorm}{\widehat{\mathcal{F}}_{\includegraphics[scale=0.15]{figures/link-1}}}
\newcommand{\LFfourbRenorm}{\widehat{\mathcal{F}}_{\includegraphics[scale=0.15]{figures/link-2}}}

\newcommand{\chamberfusion}[1]{\chamber_{\includegraphics[scale=0.5]{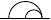}_{#1}}}
\newcommand{\Pfusion}[1]{\mathsf{P}_{\includegraphics[scale=0.5]{figures/link4fusion}_{#1}}}
\newcommand{\Efusion}[1]{\mathsf{E}_{\includegraphics[scale=0.5]{figures/link4fusion}_{#1}}}
\newcommand{\LZfusion}[1]{\mathcal{Z}_{\includegraphics[scale=0.5]{figures/link4fusion}_{#1}}}
\newcommand{\LZfusionInv}[1]{\mathcal{Z}_{\rotatebox{180}{\scalebox{1}[-1]{\includegraphics[scale=0.5]{figures/link4fusion}}}_{#1}}}
\newcommand{\Prainbow}[1]{\mathsf{P}_{\includegraphics[scale=0.15]{figures/link-2}_{#1}}}
\newcommand{\Erainbow}[1]{\mathsf{E}_{\includegraphics[scale=0.15]{figures/link-2}_{#1}}}
\newcommand{\LZrainbow}[1]{\mathcal{Z}_{\includegraphics[scale=0.15]{figures/link-2}_{#1}}}

\global\long\def\coulomb{\LH}
\global\long\def\auxcoulomb{\hat{\coulomb}}
\global\long\def\coulombGas{\LF}
\global\long\def\coulombnew{\LK}
\global\long\def\coulombLine{\LG}
\global\long\def\kfunc{p}

\global\long\def\eps{\epsilon}
\global\long\def\ov{\overline}
\global\long\def\QQrp{\QQ_{\alpha; \bs{s}}^{(p)}}

\global\long\def\bn{\mathbf{n}}
\global\long\def\MR{MR}
\global\long\def\cond{\,|\,}
\global\long\def\la{\langle}
\global\long\def\ra{\rangle}
\global\long\def\tree{\Upsilon}
\global\long\def\prob{\mathbb{P}}
\global\long\def\hm{\mathrm{Hm}}
%

\global\long\def\Im{\operatorname{Im}}
\global\long\def\Re{\operatorname{Re}}

\global\long\def\ud{\mathrm{d}}
\global\long\def\pder#1{\frac{\partial}{\partial#1}}
\global\long\def\pdder#1{\frac{\partial^{2}}{\partial#1^{2}}}
\global\long\def\pddder#1{\frac{\partial^{3}}{\partial#1^{3}}}
\global\long\def\der#1{\frac{\ud}{\ud#1}}

\global\long\def\bZnn{\mathbb{Z}_{\geq 0}}
\global\long\def\bZpos{\mathbb{Z}_{> 0}}
\global\long\def\bZneg{\mathbb{Z}_{< 0}}

\global\long\def\Vfunc{\LG}
\global\long\def\gfunc{g^{(\rr)}}
\global\long\def\hfunc{h^{(\rr)}}

\global\long\def\SimplexInt{\rho}
\global\long\def\CubeInt{\widetilde{\rho}}

\global\long\def\ii{\mathfrak{i}}
\global\long\def\ee{\mathrm{e}}
\global\long\def\rr{\mathfrak{r}}
\global\long\def\chamber{\mathfrak{X}}
\global\long\def\Wchamber{\mathfrak{W}}

\global\long\def\SimplexIntKappa8{\SimplexInt}

\global\long\def\nested{\boldsymbol{\underline{\Cap}}}
\global\long\def\unnested{\boldsymbol{\underline{\cap\cap}}}
\global\long\def\unnested{\boldsymbol{\underline{\cap\cap}}}

\global\long\def\acycle{\vartheta}
\global\long\def\bcycle{\tilde{\acycle}}

\global\long\def\metric{\mathrm{dist}}

\global\long\def\adj#1{\mathrm{adj}(#1)}

\global\long\def\bs{\boldsymbol}

\global\long\def\edge#1#2{\langle #1,#2 \rangle}
\global\long\def\graph{G}

\newcommand{\conn}{\varsigma}
\newcommand{\realacycle}{\smash{\mathring{\acycle}}}
\newcommand{\realpt}{\smash{\mathring{x}}}
\newcommand{\corrind}{\LC}
\newcommand{\bssymb}{\pi}
\newcommand{\PRCM}{\mu}
\newcommand{\coeff}{p}
\newcommand{\MainConst}{C}

\global\long\def\removeLink{/}

\global\long\def\domainofdef{\mathfrak{U}}
\global\long\def\Test_space{C_c^\infty}
\global\long\def\Distr_space{(\Test_space)^*}

\global\long\def\bs{\boldsymbol}
\global\long\def\cst{\mathrm{C}}

\newcommand{\red}{\textcolor{red}}
\newcommand{\blue}{\textcolor{blue}}
\newcommand{\green}{\textcolor{green}}
\newcommand{\magenta}{\textcolor{magenta}}

\newcommand{\coulombGasH}{\mathcal{H}}
\newcommand{\secondbeta}{\intloop}

\newcommand{\cev}[1]{\reflectbox{\ensuremath{\vec{\reflectbox{\ensuremath{#1}}}}}}

\global\long\def\anticonf{\zeta}
\global\long\def\intloop{\varrho}
\global\long\def\Gloop{\smash{\mathring{\intloop}}}

\global\long\def\SLEmeasure{\mathrm{P}}
\global\long\def\SLEmeasureEx{\mathrm{E}}

\global\long\def\fugacity{\nu}
\global\long\def\meanderMat{\mathcal{M}}
\global\long\def\LM{\mathcal{M}}
\global\long\def\meanderMatrix{\meanderMat_{\fugacity}}
\global\long\def\meanderMatrixPrime{\meanderMat_{\fugacity(\kappa')}}
\global\long\def\meanderRenorm{\widehat{\mathcal{M}}}

\global\long\def\PartFRenorm{\widehat{\PartF}}
\global\long\def\coulombGasRenorm{\widehat{\coulombGas}}

\global\long\def\hexa{\scalebox{1.3}{\hexagon}}

\global\long\def\np{p}

\global\long\def\FKdual{\mathcal{L}}

\global\long\def\fixedindex{\flat}
\maketitle
\vspace{-1cm}
\begin{center}
\begin{minipage}{0.95\textwidth}
\begin{abstract}
We classify the linear solution spaces of the chordal Belavin--Polyakov--Zamolodchikov (BPZ) equations, which arise in Dub\'edat's commutation relations for multiple Schramm--Loewner evolutions (SLE). Without imposing growth conditions, we determine the subspaces selected by successive conformal Ward identities and their exact dimensions. For every $\kappa>0$ and BPZ spectral parameter $\lambda\in\mathbb{R}$, we determine the spectrum and Jordan structure of the translation operator. When $\lambda=0$ and $\kappa\in(0,8)$, we further determine all admissible scaling exponents and the dimensions of the corresponding translation-invariant solution spaces. A triangular change of derivative coordinates yields a first-order system of rational Knizhnik--Zamolodchikov type, allowing us to represent the Ward operators explicitly and reduce the classification to finite-dimensional linear algebra.

For $2N$ boundary points and $\kappa\in(0,8)$, translation-invariant BPZ solutions with $\lambda=0$ and the homogeneity required by M\"obius covariance automatically satisfy the third Ward identity and the usual power-law bound. This recovers, under weaker assumptions, the Catalan dimension formula previously established by Flores and Kleban. At $\kappa=8$, the third Ward identity imposes an additional constraint when $N\ge2$, but the full M\"obius-covariant solution space still has Catalan dimension, establishing completeness of the partition functions constructed from uniform spanning trees. Finally, at $\kappa=4$ and $\lambda>0$, we construct an explicit basis of positive BPZ solutions and identify its elements as partition functions for level lines of a Gaussian free field with suitable harmonic means.
\end{abstract}

\end{minipage}
\end{center}

\tableofcontents
\section{Introduction}
Conformal invariance and the domain Markov property are the guiding principles behind the description of planar critical interfaces by Schramm–Loewner evolution (SLE)~\cite{SchrammFirstSLE}. 
When boundary conditions produce several interfaces, one must additionally account for their joint law: exploring the curves in different orders should give consistent conditional distributions. This principle underlies the multiple SLE framework developed in connection with conformal field theory and statistical mechanics~\cite{BauerBernardKytolaMultipleSLE}, and its probabilistic formulation through Dub\'edat’s commutation relations~\cite{DubedatCommutationSLE}. It leads to a basic classification problem: which local growth laws satisfy these relations, and which partition functions encode them?

Dub\'edat’s analysis expresses the driving drifts as logarithmic derivatives of a partition function and translates the commutation conditions into a system of Belavin–Polyakov–Zamolodchikov (BPZ) equations~\cite{DubedatCommutationSLE}, connecting the probabilistic formulation with the null-vector equations of conformal field theory~\cite{BPZInfiniteConformalSymmetry2DQuantum,BPZInfiniteConformalSymmetryCritical2D,BauerBernardCFTSLE}. In the chordal formulation with a distinguished point at infinity, covariance under conformal maps fixing that point imposes translation invariance and homogeneity, without fixing the homogeneity degree in advance. The more restrictive M\"obius-covariant problem with $2N$ boundary insertions of weight $h_{1,2}=(6-\kappa)/(2\kappa)$ involves three conformal Ward identities, associated with translation, scaling and inversion. For $\kappa\in(0,8)$, Flores and Kleban proved that the solution space of the BPZ equations and these three identities, subject to a power-law bound, has dimension equal to the $N$-th Catalan number~\cite{FloresKlebanPDE1,FloresKlebanPDE2,FloresKlebanPDE3}. 
Related classification problems have also been studied in radial geometry: Krusell, Wang, and Wu treated two locally commuting curves under an interchangeability assumption~\cite{KrusellWangWuCommutationRelation}, while Huang and Wu determined the solution spaces of the radial BPZ and rotation Ward equations for arbitrary numbers of marked points~\cite{HuangWuCommutationRadial}.

The chordal M\"obius-covariant space studied by Flores and Kleban has a distinguished basis of pure partition functions, indexed by planar link patterns and selected by boundary fusion asymptotics. Their construction connects Coulomb gas integrals~\cite{DubedatEulerIntegralsCommutingSLEs,
FloresKlebanPDE3,PeltolaBasisPDE,FengLiuPeltolaWu2024,zhang2025multiplechordalslekappaquantum}, quantum-group methods~\cite{KytolaPeltolaPurePartitionFunctions,KytolaPeltolaConformalCovBoundaryCorrelation}, and probabilistic constructions of multiple SLE~\cite{KozdronLawlerMultipleSLEs,PeltolaWuGlobalMultipleSLEs,WuHyperSLE}. The corresponding global curve laws have been studied through their resampling characterization~\cite{BeffaraPeltolaWuUniqueness,ZhanExistenceUniquenessMultipleSLE} and, more recently, conformal welding~\cite{AngHoldenSunYu2023}. The fusion of boundary marked points leads to partition functions satisfying higher-order BPZ equations~\cite{DubedatSLEVirasoroFusion,LiuWuCrossingProbaMGFF,KarrilaLafayPeltolaRoussillonUSTBranches,KarrilaPeltolaSchougGFFdegeneratelevelline}.

With the usual growth condition and normalization, a unique positive family of pure partition functions is now available for every $\kappa\in(0,8)$, see references in~\cite{FengLiuPeltolaWu2024}. These functions encode connectivity information: suitably weighted ratios of pure partition functions are predicted to describe connection probabilities in critical loop-$O(n)$ model, random-cluster models and GFF level lines~\cite{FloresSimmonsKlebanZiffCrossingProba,PeltolaWuGlobalMultipleSLEs,PeltolaWuCrossingProbaIsing,FengPeltolaWuConnectionProbaFKIsing,FengLiuPeltolaWu2024}. At the endpoint $\kappa=8$, Liu, Peltola, and Wu constructed positive, linearly independent partition functions associated with uniform spanning tree connectivities, with logarithmic boundary fusion asymptotics~\cite{LiuPeltolaWuUST}. Whether these functions span the full M\"obius-covariant BPZ solution space was left open.

In this article, we classify the successive Ward constraints on the full chordal BPZ solution space, rather than imposing M\"obius covariance and boundary growth at the beginning.
In particular, we give answers to several questions about BPZ equations asked in~\cite{FloresKlebanPDE1,PeltolaCFTSLE, LiuPeltolaWuUST,zhang2025multiplechordalslekappaquantum}. 

\paragraph*{Chordal BPZ system.}
Fix $n\ge 2$ and $\kappa>0$ and $\lambda\in\R$. We consider functions $\LZ$ defined on the space
\[\LX_n=\{(x_1, \ldots, x_n)\in\R^n: x_1<\cdots<x_n\}\]
that satisfy the chordal BPZ system:
\begin{equation}\label{eqn::BPZ_chordal}
\frac{\kappa}{2}\partial_j^2\LZ+\sum_{\ell\neq j}\left(\frac{2}{x_{\ell}-x_j}\partial_{\ell}\LZ-\frac{(6-\kappa)/\kappa}{(x_{\ell}-x_j)^2}\LZ\right)=\frac{\lambda}{\kappa}\LZ,\qquad \text{for all }1\le j\le n.
\end{equation}
Define 
\begin{equation}
\LSchord{n}^{(\lambda)}:=\{F\in C^2(\LX_n\to\R): F\text{ satisfies chordal BPZ system~\eqref{eqn::BPZ_chordal}}\}. 
\end{equation}

We recall Dub\'edat's commutation relation in Appendix~\ref{appendix::Dubedat}. 
We present two versions of commutation relation there: a weak version and the original version in~\cite{DubedatCommutationSLE}. 
The weak version corresponds to positive solutions of the chordal BPZ system~\eqref{eqn::BPZ_chordal} and the first Ward identity, for some constants $\lambda,\nu\in\R$:
\begin{equation}\label{eqn::Ward}
\sum_{j=1}^n\partial_j\LZ=\frac{\nu}{\kappa}\LZ.
\end{equation}
The original version imposes $\lambda=\nu=0$ and the second Ward identity: for $\lambda=\nu=0$ and $\xi\in\R$, we consider positive solutions $\LZ$ that also satisfy the second Ward identity:
\begin{align}\label{eqn::Ward_2nd}
\sum_{j=1}^n x_j\partial_j\LZ=\frac{\xi}{\kappa}\LZ.
\end{align}
We summarize the results below.
\begin{itemize}
\item We first decompose the solution space $\LSchord{n}^{(\lambda)}$ according to different levels in the first Ward identity~\eqref{eqn::Ward}.
In the generic case $\lambda\neq 0$, Theorem~\ref{thm::BPZ_decomposition} determines the admissible values of $\nu$ in~\eqref{eqn::Ward} and the corresponding dimensions, giving a direct sum decomposition when $\lambda>0$ and ruling out positive solutions when $\lambda<0$.
In the degenerate case $\lambda=0$, the only admissible value of $\nu$ is $\nu=0$, and Theorem~\ref{thm::BPZ_decomposition_degenerate} gives the dimensions of the subspaces determined by the generalized first Ward identity~\eqref{eqn::Ward_generalized}. Together, these results classify the linear solution spaces for the BPZ system and the first Ward identity. Their positive elements are precisely the partition functions for the weak commutation relation in Appendix~\ref{appendix::Dubedat}.
\item In the degenerate case $\lambda=\nu=0$, we further decompose the space according to different levels in the second Ward identity~\eqref{eqn::Ward_2nd}. 
Theorem~\ref{thm::xi_p_decomposition} identifies the admissible levels $\xi$ and the dimensions of the corresponding solution spaces for $\kappa\in(0,8)$. For even $n=2N$ at the level $\xi=\xi_0$, Corollary~\ref{cor::BPZ_Catalan} recovers the Catalan dimension of the space of M\"obius-covariant solutions studied in~\cite{FloresKlebanPDE1,FloresKlebanPDE2,FloresKlebanPDE3} and strengthens this result: interestingly, the BPZ system and the first two Ward identities already imply both the third Ward identity~\eqref{eqn::Ward_3rd} and the power-law bound when $\kappa\in(0,8)$; whereas, this is not the case when $\kappa=8$.

\item At $\kappa=8$ and $n=2N$ even, the scaling exponents $\xi_0$ and $\xi_2$ of the second Ward identity~\eqref{eqn::Ward_2nd} coincide. Theorem~\ref{thm::UST_decomposition} determines the dimension of the corresponding solution space and shows that imposing the third Ward identity selects a subspace of dimension $\Cat_N$, without any growth assumption. This extends the Catalan dimension result to the endpoint $\kappa=8$. For $N\ge2$, there are solutions with the required translation and scaling covariance that fail the third Ward identity. Thus full M\"obius covariance is an additional constraint at this endpoint.

\item When $\kappa=4$, we provide explicit solutions to the chordal BPZ system~\eqref{eqn::BPZ_chordal} in Proposition~\ref{prop::LFchordsigma}. When $\lambda>0$, they are $2^n$ linearly independent positive solutions. Moreover, they are the partition functions for level lines of Gaussian free field with properly chosen mean value, see Proposition~\ref{prop::GFF_chord}.  
\end{itemize}

Our approach is very different from that in~\cite{FloresKlebanPDE1,FloresKlebanPDE2,FloresKlebanPDE3}.
Our analysis extends the first-order formulation developed in~\cite{HuangWuCommutationRadial} from radial setting to chordal setting.
The radial setting has only one conformal Ward identity; whereas, the chordal setting has three. For this reason, the chordal setting exhibits a much richer hierarchy structure coming from conformal Ward constraints.
To analyze the further hierarchy structure, a key difference from the radial is that, in the chordal setting,  we figure out a triangular combinations of the derivative coordinates.
The resulting first-order system is of rational Knizhnik--Zamolodchikov type, see~\cite{KnizhnikZamolodchikov1984,FelderMarkovTarasovVarchenkoKZ,ZabrodinZotov2017}. It has explicit coefficient matrices, and this allows us to represent the Ward operators explicitly and study the successive constraints using finite-dimensional linear algebra.
We identify the solution spaces selected by the Ward identities with eigenspaces or kernels of finite-dimensional matrices, with the Ward parameters appearing as eigenvalues. Successively imposing these identities gives a hierarchy of subspaces. At $\lambda=0$, an $\mathfrak{sl}_2$ representation built from the translation matrix helps us analyze the structure of these subspaces and of the remaining Ward operators.
 The formal semiclassical limit of the same chordal BPZ system is recently studied in~\cite{ChenWuSemiclassicalBPZ}, which classifies the global real-valued solutions of the resulting nonlinear Hamilton--Jacobi equations modulo additive constants.

\subsection{First Ward identity}
Fix $n\ge 2$ and $\kappa>0$. For $\lambda\in\R$ and $\nu\in\R$, we define 
\begin{align}
	\LSchord{n}^{(\lambda;\nu)}:=\left\{F\in C^2(\LX_n\to \R):
	\begin{array}{l}
		F \text{ satisfies chordal BPZ system~\eqref{eqn::BPZ_chordal}}\\
		\text{and the first Ward identity~\eqref{eqn::Ward}}
	\end{array}\right\}.
\end{align}
We will decompose the solution space $\LSchord{n}^{(\lambda)}$ according to different levels of $\nu$ in the first Ward identity~\eqref{eqn::Ward}. 
The decomposition is different for the generic case $\lambda\neq 0$ and the degenerate case $\lambda=0$. 

\begin{theorem}\label{thm::BPZ_decomposition}
Fix $n\ge 2$ and $\kappa>0$. For $\lambda\in\R$, we have $\dim\LSchord{n}^{(\lambda)}=2^n$. 
\begin{itemize}
\item When $\lambda>0$, we have
\begin{align}\label{eqn::lambda_nu_dimension}
\dim\LSchord{n}^{(\lambda;\nu)}=
\begin{cases}
\displaystyle\binom{n}{\mathfrak{r}}, \quad &\text{if }\nu=(n-2\mathfrak{r})\sqrt{2\lambda}, \quad \mathfrak{r}\in\{0,1,\ldots, n\}; \\
0, &\text{otherwise}.
\end{cases}
\end{align}
Moreover, we have the decomposition
\begin{align}\label{eqn::lambda_nu_decomposition}
\LSchord{n}^{(\lambda)}=\bigoplus_{\mathfrak{r}=0}^n
			\LSchord{n}^{\left(\lambda; \nu_\mathfrak{r} \right)}, \qquad \text{where }\nu_{\mathfrak{r}}:=(n-2\mathfrak{r})\sqrt{2\lambda}.
\end{align}
\item When $\lambda<0$, we have
\begin{equation}\label{eqn::lambda_negative_nu_dimension}
\dim\LSchord{n}^{(\lambda;\nu)}=
\begin{cases}
\displaystyle\binom{n}{n/2}, &\text{if $n$ is even and $\nu=0$},\\
0, &\text{otherwise}.
\end{cases}
\end{equation}
Moreover, any solution in $\LSchord{n}^{(\lambda)}$ has at least one zero in $\LX_n$.
\end{itemize}
\end{theorem}

\begin{figure}[ht!]
	\begin{center}
	\includegraphics[width=\textwidth]{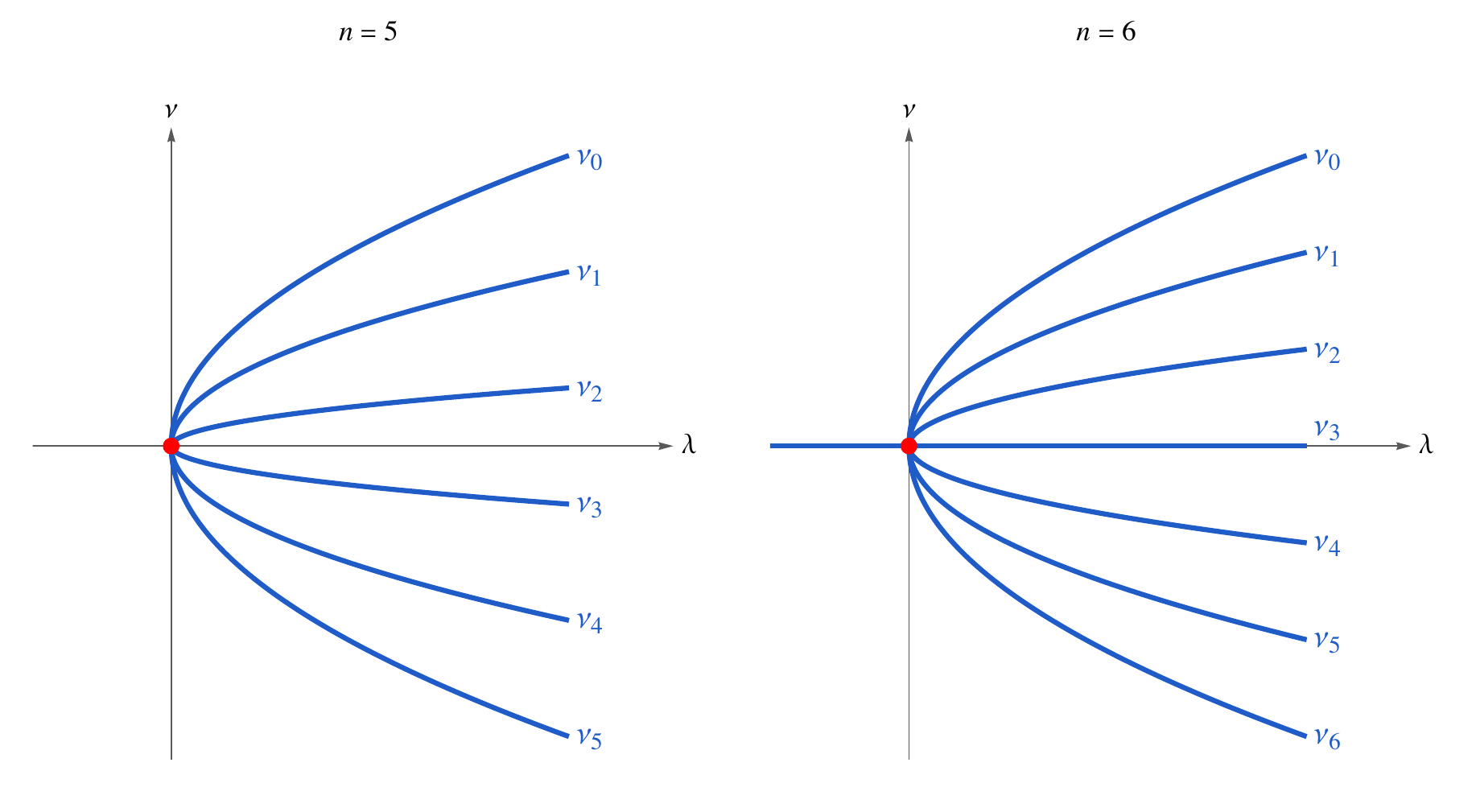}
	\end{center}
	\caption{Illustration for the admissible pairs $(\lambda,\nu)$ in Theorem~\ref{thm::BPZ_decomposition} for $n=5$ and $n=6$. The blue curves correspond to the generic case $(\lambda>0, \nu_{\mathfrak{r}}=(n-2\mathfrak{r})\sqrt{2\lambda})$ and $(\lambda<0, \nu=0)$. The red dots correspond to the degenerate case $(\lambda=0, \nu=0)$.
	Note that the line $\nu=0$ is admissible for even $n$ and is not admissible for odd $n$ except at the degenerate case.}
	\label{fig::lambdanu}
\end{figure}

We prove Theorem~\ref{thm::BPZ_decomposition} in Section~\ref{subsec::Ward_1st}. The triangular change of coordinates (in Section~\ref{subsec::renormalization}) turns the translation operator into an explicit constant matrix. We construct its eigenvectors directly and read off the admissible values of $\nu$ and their multiplicities. This gives a direct algebraic calculation of the spectrum. For $\lambda<0$, an energy estimate rules out positive solutions, so this parameter range cannot occur for the weak commutation relation.
For $\lambda>0$, Proposition~\ref{prop::CGI} constructs bases of all the first-Ward eigenspaces from Coulomb gas integrals. The Coulomb gas construction originates in~\cite{DotsenkoFateevCoulombGas} and was applied to commuting SLEs in~\cite{DubedatEulerIntegralsCommutingSLEs}.
The exponential modification of Coulomb gas integral also occurs in integral representations of irregular conformal blocks, see~\cite{HaghighatLiuReshetikhinIrregularBlocks}.
In the following two special cases, the solutions are explicit.
\begin{itemize}
\item When $\nu=\nu_0=n\sqrt{2\lambda}$, no screening variables are needed, and there is only one-dimensional solution in $\LSchord{n}^{(\lambda; \nu_0)}$ which is given by 
\begin{align}\label{eqn::halfwatermelon_pf_drift}
\LZ_{\shuffle_n}^{(\lambda)}(\bs{x})=\prod_{1\le i<\ell\le n}(x_{\ell}-x_i)^{\frac{2}{\kappa}}\times\exp\left(\frac{\sqrt{2\lambda}}{\kappa}\sum_{j=1}^n x_j\right), \qquad\text{for }\bs{x}=(x_1, \ldots, x_n)\in\LX_n. 
\end{align}
This is a generalization of the partition function for half-watermelon $\SLE_{\kappa}$, see Section~\ref{subsec::halfwatermelon_SLE}. 
\item When $\kappa=4$, the residue calculation in Proposition~\ref{prop::CGI_kappa4} evaluates the integrals explicitly and gives the $2^n$ positive basis functions in Proposition~\ref{prop::LFchordsigma}.
\end{itemize}


\subsection{Degenerate case}
In the degenerate case $\lambda=0$, the only admissible $\nu$ in the first Ward identity~\eqref{eqn::Ward} is $\nu=0$ and $\dim\LSchord{n}^{(0;0)}=\binom{n}{\lfloor n/2\rfloor}$. The solutions in $\LSchord{n}^{(0)}\setminus \LSchord{n}^{(0;0)}$ can be classified by generalized first Ward identity.
For $q\in\{1,2, \ldots, n+1\}$, we consider solutions $\LZ\in\LSchord{n}^{(0)}$ that also satisfy the generalized first Ward identity
\begin{align}\label{eqn::Ward_generalized}
\LSchord{n}^{\langle q\rangle}:=\left\{F\in \LSchord{n}^{(0)}: \big(\sum_{j=1}^n\partial_j\big)^q F=0\right\}.
\end{align}
Note that $\LSchord{n}^{\langle 1\rangle}=\LSchord{n}^{(0;0)}$. 

\begin{theorem}\label{thm::BPZ_decomposition_degenerate}
Fix $\kappa>0$ and $\lambda=0$. For $n\ge 2$, we denote 
\begin{align}\label{eqn::parityn_def}
\mathsf{P}_n=\{p\in\{0,1,\ldots, n\}: n-p\text{ is even}\}. 
\end{align}
For $p\in\mathsf{P}_n$, define 
\begin{align}\label{eqn::bnp_def}
\mathsf{B}_{n,p}=\binom{n}{(n-p)/2}-\binom{n}{(n-p)/2-1}, 
\end{align}
where $\binom{n}{-1}:=0$. 
For $q\in\{1,2, \ldots, n+1\}$, we have
\begin{align}
\dim\LSchord{n}^{\langle q\rangle}=\sum_{p\in\mathsf{P}_n}\min(q,p+1)\mathsf{B}_{n,p}.
\end{align}
In particular,
\[\dim\LSchord{n}^{\langle 1\rangle}=\binom{n}{\lfloor n/2\rfloor},\qquad \dim\LSchord{n}^{\langle n+1\rangle}=2^n.\]
\end{theorem}

We prove Theorem~\ref{thm::BPZ_decomposition_degenerate} in Section~\ref{subsec::BPZ_decomposition_degenerate}. 
At $\lambda=0$, the translation matrix becomes nilpotent, and the same change of coordinates identifies $\LSchord{n}^{\langle q\rangle}$ with the kernel of its $q$-th power. Its Jordan blocks correspond to the irreducible components of an $\mathfrak{sl}_2$ representation. For the relation between its highest-weight spaces and planar link patterns, see \cite{FloresPeltolaSchurWeyl}. 
We next restrict to the translation-invariant level $\LSchord{n}^{\langle1\rangle}=\LSchord{n}^{(0;0)}$, of dimension $\binom{n}{\lfloor n/2\rfloor}$, and decompose it according to the second Ward identity.

\subsection{Second Ward identity}
Fix $n\ge 2$ and $\kappa>0$. For $\lambda=\nu=0$ and $\xi\in\R$, we define 
\begin{align}
\LSchord{n}^{(0;0;\xi)}:=\left\{F\in C^2(\LX_n\to \R):
	\begin{array}{l}
		F \text{ satisfies chordal BPZ system~\eqref{eqn::BPZ_chordal} with }\lambda=0\\
		\text{and the first Ward identity~\eqref{eqn::Ward}  with }\nu=0\\
		\text{and the second Ward identity~\eqref{eqn::Ward_2nd}}
	\end{array}\right\}.
\end{align}

\begin{theorem}\label{thm::xi_p_decomposition}
Fix $\kappa\in (0,8)$ and $\lambda=\nu=0$. For $n\ge 2$, recall that $\mathsf{P}_n=\{p\in\{0,1,\ldots, n\}: n-p\text{ is even}\}$ is defined in~\eqref{eqn::parityn_def} and $\mathsf{B}_{n,p}$ is defined in~\eqref{eqn::bnp_def}. Define 
\begin{align}\label{eqn::xi_p_def}
\xi_p=p(p+2-\tfrac{\kappa}{2})-n(3-\tfrac{\kappa}{2}). 
\end{align}
We have
\begin{align}\label{eqn::solution_xi_dimension}
\dim\LSchord{n}^{(0;0;\xi)}=
\begin{cases}
\mathsf{B}_{n,p}, &\text{if }\xi=\xi_p, \quad p\in\mathsf{P}_n;\\
0, &\text{otherwise}.
\end{cases}
\end{align}
 Moreover, we have the decomposition
 \begin{align}\label{eqn::solution_xi_decomposition}
 \LSchord{n}^{(0;0)}=\bigoplus_{p\in\mathsf{P}_n}\LSchord{n}^{(0;0;\xi_p)}. 
 \end{align}
\end{theorem}

We prove Theorem~\ref{thm::xi_p_decomposition} in Section~\ref{subsec::Ward_2nd}.
In the renormalized coordinates, translation-invariant solutions correspond to highest-weight vectors of the same $\mathfrak{sl}_2$ representation. The Casimir and degree operators determine the scalar diagonal blocks of the dilation matrix, while a weighted lowering operator connects adjacent highest-weight spaces. This block upper triangular form gives both its characteristic polynomial and an annihilating polynomial. For $\kappa\in(0,8)$, the levels $\xi_p$ are distinct, so the matrix is diagonalizable and the dimensions follow from the characteristic polynomial. This further decomposition uses the Euclidean dilation symmetry of the chordal configuration space. Thus Theorem~\ref{thm::xi_p_decomposition} classifies the linear solution spaces underlying the original version of Dub\'edat's commutation relation, and settles the dimension and full-space classification problem left open in~\cite[Section~1.2]{zhang2025multiplechordalslekappaquantum}, for $\kappa\in(0,8)$. Their positive elements give the corresponding partition functions, as explained in Appendix~\ref{appendix::Dubedat}. In Appendix~\ref{sec::solutions_lambdazero}, we give a probabilistic construction of the solutions for $\LSchord{n}^{(0;0;\xi_p)}$ when $\kappa\in (0,4]$. These solutions are all positive and form a basis for the space $\LSchord{n}^{(0;0;\xi_p)}$. 

For irrational $\kappa\in(0,8)$, Peltola constructs $\mathsf B_{n,p}$ linearly independent solutions at the level $\xi=\xi_p$, indexed by planar link patterns with $p$ defects, using the spin chain--Coulomb gas correspondence of~\cite{KytolaPeltolaConformalCovBoundaryCorrelation}; see~\cite[Theorem~5.3 and Proposition~6.3]{PeltolaBasisPDE}. Theorem~\ref{thm::xi_p_decomposition} shows that these functions span the complexification of $\LSchord{n}^{(0;0;\xi_p)}$, without imposing growth conditions or prescribed fusion asymptotics. Our dimension formula also covers rational $\kappa\in(0,8)$.

The levels $\xi_p$ in Theorem~\ref{thm::xi_p_decomposition} have a natural interpretation in terms of the first row of the Kac table:
\begin{align*}
\frac{\xi_p}{\kappa}=h_{1,p+1}-nh_{1,2},\qquad\text{where }h_{1,s}=\frac{(\kappa-4s)^2-(\kappa-4)^2}{16\kappa}.
\end{align*}
In the boundary CFT interpretation, each finite marked point carries the level-two degenerate weight $h_{1,2}=(6-\kappa)/(2\kappa)$, while $p$ unpaired curves ending at infinity correspond to a boundary field of weight $h_{1,p+1}$ at their common endpoint.
Under a dilation $z\mapsto Rz$, the derivative at infinity, computed in the local coordinate $1/z$, is $R^{-1}$; conformal covariance therefore gives the homogeneity degree $h_{1,p+1}-nh_{1,2}$.

\begin{figure}[ht!]
	\begin{center}
	\includegraphics[width=0.48\textwidth]{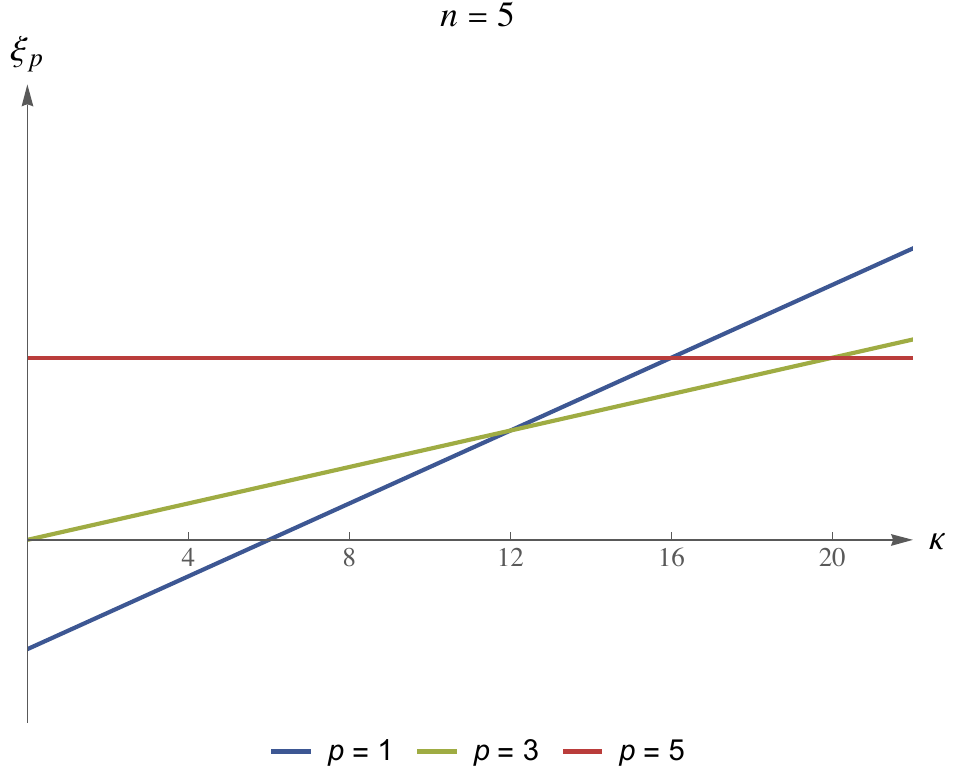}\hfill
    \includegraphics[width=0.48\textwidth]{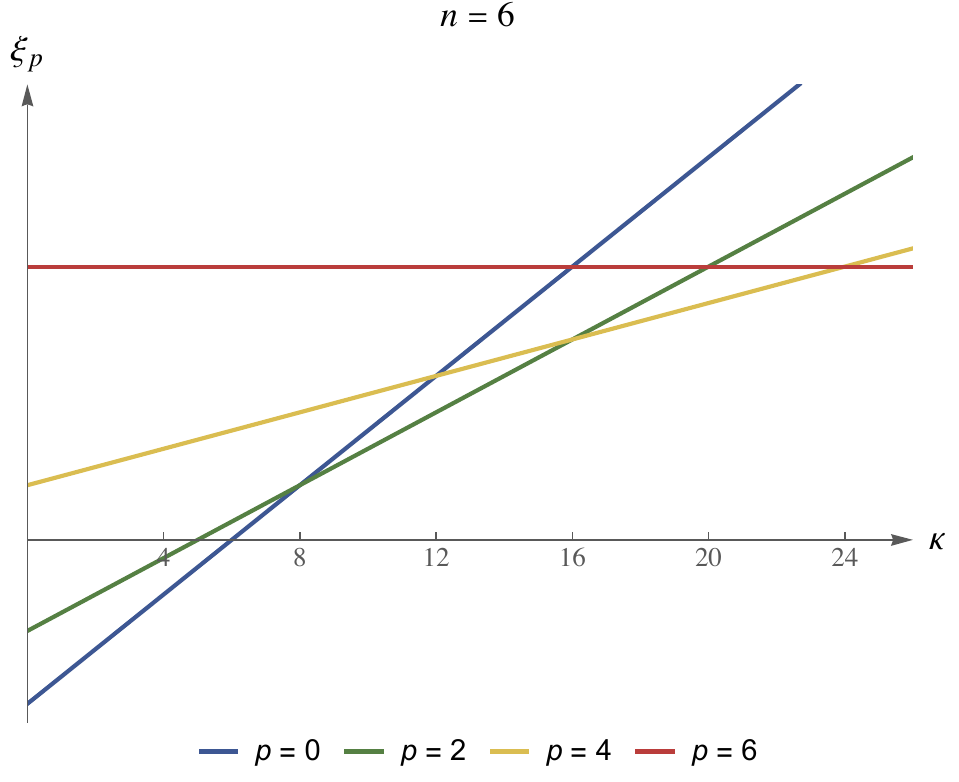}
	\end{center}
	\caption{Illustration for the admissible $\xi_p$ in Theorem~\ref{thm::xi_p_decomposition} for $n=5$ and $n=6$. In the left panel, the three lines are $\xi_1, \xi_3, \xi_5$. The first intersection happens at $\kappa=12$. In the right panel, the four lines are $\xi_0, \xi_2, \xi_4, \xi_6$. The first intersection happens at $\kappa=8$.}
	\label{fig::xi_p}
\end{figure}

We emphasize that the conclusions in Theorems~\ref{thm::BPZ_decomposition} and~\ref{thm::BPZ_decomposition_degenerate} hold for all $\kappa>0$, whereas the range for $\kappa$ in Theorem~\ref{thm::xi_p_decomposition} is $\kappa\in (0,8)$. In fact, the results in Theorem~\ref{thm::xi_p_decomposition} hold for every $\kappa>0$ satisfying
\begin{equation}
	\kappa\notin \{2(p_1+p_2+2):p_1, p_2\in\mathsf{P}_n,\ p_1<p_2\}.
\end{equation}
The excluded parameters are precisely those for which two of the exponents $\xi_p$ coincide, see Figure~\ref{fig::xi_p}. Since $p_1+p_2\ge2$, they are all at least $8$.
At these parameters, the results can fail, see Remark~\ref{rem::xi_p_fail}. 
Theorems~\ref{thm::BPZ_decomposition}, \ref{thm::BPZ_decomposition_degenerate}, and~\ref{thm::xi_p_decomposition} also hold for $\kappa<0$: the same arguments apply, and the exponents $\xi_p$ are distinct throughout this range. We state the results for positive $\kappa$ because of their SLE interpretation.

\subsection{Third Ward identity}
\label{subsec::intro_3rdWard}
For $n=2N$ even and $\kappa\in(0,8)$, the choice $p=0$ in Theorem~\ref{thm::xi_p_decomposition} corresponds to the homogeneity degree $\xi_0/\kappa=-nh_{1,2}$.
The following corollary recovers the Catalan dimension obtained in~\cite[Theorem~8]{FloresKlebanPDE3}.

\begin{corollary}\label{cor::BPZ_Catalan}
Fix $n=2N$ even and $\kappa\in(0,8)$ and $\lambda=\nu=0$. 
Recall that $\xi_0=-n(3-\tfrac{\kappa}{2})$ from~\eqref{eqn::xi_p_def} and 
\begin{align}\label{eqn::space_Catalan}
\LSchord{n}^{(0;0;\xi_0)}=\left\{F\in C^2(\LX_n\to \R):
	\begin{array}{l}
		F \text{ satisfies chordal BPZ system~\eqref{eqn::BPZ_chordal} with }\lambda=0\\
		\text{and the first Ward identity~\eqref{eqn::Ward}  with }\nu=0\\
		\text{and the second Ward identity~\eqref{eqn::Ward_2nd} with }\xi=\xi_0
	\end{array}\right\}.
\end{align}
	We have 
\begin{align}\label{eqn::solution_Catalan_dimension}
\dim \LSchord{n}^{(0;0;\xi_0)}=\frac{1}{N+1}\binom{2N}{N}=:\Cat_N.
\end{align}
Furthermore, every $\LZ\in\LSchord{n}^{(0;0;\xi_0)}$ satisfies the third Ward identity
\begin{equation}\label{eqn::Ward_3rd}
    \sum_{j=1}^n x_j^2\partial_j\LZ
    =-\frac{6-\kappa}{\kappa}\left(\sum_{j=1}^n x_j\right)\LZ.
\end{equation}
\end{corollary}

Let us explain the relation between Corollary~\ref{cor::BPZ_Catalan} and~\cite{FloresKlebanPDE1,FloresKlebanPDE2,FloresKlebanPDE3}. 
Define
\begin{align*}
\LSchord{n}^{\mathrm{FK}}=
\left\{F\in \LSchord{n}^{(0;0;\xi_0)}:
	\begin{array}{l}
		F \text{ satisfies the third Ward identity~\eqref{eqn::Ward_3rd}}\\
		\text{and a power-law bound}
	\end{array}
	\right\}.
\end{align*}
It is proved in~\cite[Theorem~8]{FloresKlebanPDE3} that, for $n=2N$ even and $\kappa\in(0,8)$, 
\begin{align}\label{eqn::FK15}
\dim \LSchord{n}^{\mathrm{FK}}=\Cat_N.
\end{align}
Comparing with~\eqref{eqn::solution_Catalan_dimension}, the third Ward identity~\eqref{eqn::Ward_3rd} and the power-law bound are consequences of the three requirements in~\eqref{eqn::space_Catalan}. 
Flores and Kleban obtain the upper bound in~\eqref{eqn::FK15} by analyzing collision limits and constructing an injective map from the solution space to a space of dimension $\Cat_N$~\cite{FloresKlebanPDE1,FloresKlebanPDE2}. Such analysis uses the M\"obius covariance encoded by all three Ward identities, while the power-law bound supplies the initial control for the collision estimates. Our method is very different.
Our dimension calculation uses the equivalent first-order system and finite-dimensional linear algebra, without imposing the third Ward identity or a power-law bound. The known pure partition functions form a basis of $\LSchord{n}^{(0;0;\xi_0)}$. Every element of this space therefore inherits their third Ward identity and power-law bounds. This confirms Peltola's suggestion in~\cite[Section~3.B]{PeltolaCFTSLE} that the power-law bound follows from the other defining properties of pure partition functions. 
Our conclusion applies to the full solution space, without assuming positivity or prescribed fusion asymptotics, and also makes the third Ward identity automatic.
\medbreak
At $\kappa=8$, the collision $\xi_0=\xi_2=n$ changes the conclusion of Corollary~\ref{cor::BPZ_Catalan}: for $N\ge2$, the dimension of $\LSchord{n}^{(0;0;\xi_0)}$ exceeds the Catalan number, and the third Ward identity~\eqref{eqn::Ward_3rd} no longer follows from the three requirements in~\eqref{eqn::space_Catalan}. The next theorem determines the dimensions both before and after imposing this additional identity.

\begin{theorem}\label{thm::UST_decomposition}
Fix $n=2N$ even and $\kappa=8$ and $\lambda=\nu=0$. 
In this case $\xi_0=n$. 
Define 
\begin{align}\label{eqn::solutionspace_UST}
\LSchord{n}^{\mathrm{UST}}=\left\{F\in \LSchord{n}^{(0;0;\xi_0)}:
	\begin{array}{l}
		F \text{ satisfies the third Ward identity~\eqref{eqn::Ward_3rd}}
	\end{array}\right\}.
\end{align}
We have 
\begin{align}
\dim\LSchord{n}^{(0;0;\xi_0)}=\frac{3N}{N+2}\Cat_N, \qquad \dim \LSchord{n}^{\mathrm{UST}}=\Cat_N.
\end{align}
\end{theorem}

We prove Theorem~\ref{thm::UST_decomposition} in Section~\ref{subsec::Ward_3rd}. At the collision $\xi_0=\xi_2$, eigenvalue multiplicities alone do not determine the eigenspace dimension. The third Ward operator maps this eigenspace into a different dilation eigenspace, so its matrix is generally rectangular. For $N\ge2$, projection onto suitable highest-weight components reduces the third Ward identity to the kernel of an explicit polynomial matrix. This confirms the extension to $\kappa=8$ anticipated in~\cite{FloresKlebanPDE1} and answers the dimension aspect of the question raised in~\cite[Question~3.13]{PeltolaCFTSLE}, without assuming a power-law bound.

In the uniform spanning tree setting, Liu, Peltola, and Wu~\cite{LiuPeltolaWuUST} construct $\Cat_N$ positive, linearly independent solutions in $\LSchord{n}^{\mathrm{UST}}$. 
In~\cite{LiuPeltolaWuUST}, the authors do not know the dimension of the solution space. Our Theorem~\ref{thm::UST_decomposition} determines this dimension and proves that the constructed solutions form a basis of $\LSchord{n}^{\mathrm{UST}}$. They are associated with operators in a logarithmic CFT with central charge $\mathfrak{c}=-2$.
The solutions constructed in~\cite{LiuPeltolaWuUST} encode the connection probabilities for uniform spanning tree in polygons. From Theorem~\ref{thm::UST_decomposition}, we see that the space $\LSchord{n}^{(0;0;\xi_0)}$ is bigger than $\LSchord{n}^{\mathrm{UST}}$ in general. It is an interesting question to give probabilistic interpretation for the solutions in $\LSchord{n}^{(0;0;\xi_0)}\setminus\LSchord{n}^{\mathrm{UST}}$.

\subsection{GFF level lines}
In this section, we present explicit solutions for $\LSchord{n}^{(\lambda)}$ when $\kappa=4$ and $\lambda>0$. 
These solutions are partition functions for level lines of Gaussian free field with properly chosen mean value. 
\begin{proposition}\label{prop::LFchordsigma}
Fix $n\ge 2$ and $\kappa=4$ and $\lambda>0$. Fix $\sigma=(\sigma_1, \ldots, \sigma_n)\in\{\pm1\}^n$. 
Define
\begin{align}\label{eqn::LFchord_sigma_def}
\LFchord{n}(\bs{x})=\prod_{1\le i<\ell\le n}(x_{\ell}-x_i)^{\frac{1}{2}\sigma_i\sigma_{\ell}}\times\exp\left(\frac{\sqrt{2\lambda}}{4}\sum_{j=1}^n\sigma_j x_j\right),\qquad\text{for }\bs{x}=(x_1, \ldots, x_n)\in\LX_n.
\end{align}
The collection $\left\{\LFchord{n}: \sigma\in\{\pm 1\}^n\right\}$ are $2^n$ linearly independent solutions to the chordal BPZ system~\eqref{eqn::BPZ_chordal} with $\kappa=4$ and $\lambda>0$.
\end{proposition}

The partition functions $\LFchord{n}$ with $\lambda=0$ are ``conformal blocks" for GFF analyzed in~\cite[Section~6]{PeltolaWuGlobalMultipleSLEs}. They are linearly dependent when $\lambda=0$; whereas, they become linearly independent when $\lambda>0$.

\begin{proposition}\label{prop::GFF_chord}
Fix $n\ge 2$ and $\kappa=4$ and $\lambda>0$. Fix $\sigma=(\sigma_1, \ldots, \sigma_n)\in\{\pm1\}^n$. 
Suppose $\Phi$ is Dirichlet GFF in $\HH$. For $\bs{x}=(x_1,\ldots,x_n)\in\LX_n$, define $\varphi_{\sigma}(\bs{x};\cdot)$ to be the harmonic function on $\HH$ given by
\begin{align}\label{eqn::sigma_H}
	\varphi_{\sigma}(\bs{x};z):={}&\pi\sum_{i=1}^n\sigma_i-\sum_{i=1}^n\sigma_i\arg(z-x_i), \qquad z\in\HH,
\end{align}
where $\arg(z-x_i)\in(0,\pi)$. The boundary data of $\varphi_{\sigma}(\bs{x}; \cdot)$ is
\begin{align}\label{eqn::data_sigma_chordal}
	0 \text{ on }(-\infty, x_1);\qquad
	\pi\sum_{i=1}^j\sigma_i\text{ on }(x_j, x_{j+1}), \qquad 1\le j\le n-1; \qquad \pi\sum_{i=1}^n\sigma_i\text{ on }(x_n,\infty).
\end{align}
For $1\le j\le n$, the level line of 
\begin{align}\label{eqn::GFF_sigma_mu_H}
	\Phi(z)+\varphi_{\sigma}(\bs{x};z)-\frac{\sqrt{2\lambda}}{2}\Im{z}
\end{align}
starting from $x_j$ with height $\pi\sum_{i=1}^{j-1}\sigma_i + \pi \sigma_j/2$ has the same law as the chordal Loewner chain with the following driving function, up to the swallowing time of $\{x_{j-1}, x_{j+1}\}$:
\begin{align}\label{eqn::LFchord_sigma_SDE}
	\ud W_t=2\ud B_t+4\left(\partial_j\log\LFchord{n}\right)(g_t(x_1), \ldots, g_t(x_{j-1}), W_t, g_t(x_{j+1}), \ldots, g_t(x_{n}))\ud t;
\end{align}
where $W_0=x_j$ and $B_t$ is one-dimensional Brownian motion and $\LFchord{n}$ is defined in~\eqref{eqn::LFchord_sigma_def}. 
\end{proposition}

Proposition~\ref{prop::GFF_chord} involves singular modifications of the GFF. In our previous work~\cite{HuangWuCommutationRadial}, we obtained the radial counterpart of this construction. 
In the radial setup, the distinguished point is the interior marked point $0$, and the additional harmonic term is $\frac{\sqrt{2\lambda}}{2}\log |z| =-\frac{\sqrt{2\lambda}}{2}\Green_{\mathbb U}(z,0)$, where $\Green$ is the Green's function. 
In the chordal setup, the corresponding term is $-\frac{\sqrt{2\lambda}}{2}\Im z$, where, up to normalization, $\Im z$ is the Poisson kernel of $\mathbb H$ associated with the boundary point at infinity. This is closely related to the work of Ivanov, Kang, and Vasil'ev~\cite{IvanovKangVasilevSlitHoloStochasticGFF}, who studied unbounded harmonic modifications of the GFF and showed, in particular, that a term proportional to $\Im z$ couples to chordal SLE with constant drift. In both geometries, the extra drift is therefore encoded by the canonical harmonic kernel associated with the distinguished marked point: a Green's function for the interior point in the radial case, and a Poisson kernel for infinity in the chordal case.

We prove Propositions~\ref{prop::LFchordsigma} and~\ref{prop::GFF_chord} in Section~\ref{sec::GFF}. Proposition~\ref{prop::LFchordsigma} follows by evaluating the Coulomb gas basis at $\kappa=4$ using residues. For Proposition~\ref{prop::GFF_chord}, we use the martingale characterization of the GFF–SLE coupling. Along the Loewner evolution, we track the harmonic mean of the field and show that it is a local martingale; its quadratic covariation exactly compensates for the decrease of the Green's function in the evolving domain.

\paragraph*{Outline.}
We prove Theorems~\ref{thm::BPZ_decomposition}-\ref{thm::UST_decomposition} in Section~\ref{sec::PDE_analysis}. 
We construct solutions for $\LSchord{n}^{(\lambda)}$ with $\lambda>0$ using Coulomb gas integral and prove Propositions~\ref{prop::LFchordsigma} and~\ref{prop::GFF_chord} in Section~\ref{sec::GFF}. 
We summarize Dub\'edat's commutation relation in Appendix~\ref{appendix::Dubedat}.  
We construct solutions for $\LSchord{n}^{(0;0; \xi_p)}$ in Appendix~\ref{sec::solutions_lambdazero}. 

\paragraph{Acknowledgment.}
We thank Eveliina Peltola, Nicolai Reshetikhin and Yilin Wang for helpful discussion. 
H.W. is supported by New Cornerstone Investigator Program 100001127. H.W. is partly affiliated at Yanqi Lake Beijing Institute of Mathematical Sciences and Applications, Beijing, China.

\section{Analysis on the PDE system}
\label{sec::PDE_analysis}
In this section, we complete the proof of Theorems~\ref{thm::BPZ_decomposition}-\ref{thm::UST_decomposition}.
We fix $\kappa>0$ and set $\mathfrak{b}=(6-\kappa)/(2\kappa)$. 
We denote the differential operator in the BPZ system~\eqref{eqn::BPZ_chordal} by 
\begin{equation}\label{eqn::chordal_BPZ_operators}
    \LD_j:=\frac{\kappa}{2}\partial_j^2
    +\sum_{\ell\ne j}\left(\frac{2}{x_\ell-x_j}\partial_\ell
    -\frac{2\mathfrak{b}}{(x_\ell-x_j)^2}\right),
    \qquad j\in\{1,\ldots,n\};
\end{equation}
and denote the differential operators in the first and the second Ward identities~\eqref{eqn::Ward} and~\eqref{eqn::Ward_2nd} by
\begin{equation}\label{eqn::chordal_Ward_12}
    \LT:=\sum_{j=1}^n\partial_j; \qquad \LE:=\sum_{j=1}^n x_j\partial_j.
\end{equation}
For operators or matrices $A$ and $B$, write $[A,B]:=AB-BA$.
Direct computation gives the following commutator identities (see also~\cite[Page~1807]{DubedatCommutationSLE}):
\begin{align}
    &[\LD_i,\LD_j]+\frac{4}{(x_i-x_j)^2}(\LD_i-\LD_j)=0,
    \qquad \text{for }i\ne j; \label{eqn::chordal_operator_commutator}\\
    &[\LD_j,\LT]=0,\qquad [\LD_j,\LE]=2\LD_j,
    \qquad \text{for }j\in\{1,\ldots,n\}; \label{eqn::BPZ_dilation}\\
    &[\LT,\LE]=\LT. \label{eqn::translation_dilation}
\end{align}

Following the first-order formulation of the radial BPZ system in~\cite[Section~4]{HuangWuCommutationRadial}, we give the corresponding construction in the chordal setting. It identifies the BPZ equations with a linear system.
For a multi-index $\bs{\beta}\in\Z_{\ge0}^n$, write $|\bs{\beta}|=\sum_j\beta_j$ and $\partial^{\bs{\beta}}=\partial_1^{\beta_1}\cdots\partial_n^{\beta_n}$. Let $\bs{0}:=(0,\ldots,0)$, and for $j\in\{1,\ldots,n\}$, let $\bs{e}_j$ denote the $j$th standard basis vector.
For a vector-valued function $\bs{u}=(u_{\bs{\beta}})_{\bs{\beta}\in\{0,1\}^n}$, consider the linear system
    \begin{equation}\label{eqn::chordal_first_order_system}
        \partial_j\bs{u}(\bs{x})=D^{(\lambda)}_j(\bs{x})\bs{u}(\bs{x}), \qquad \text{for }j\in\{1,\ldots,n\},
    \end{equation}
    defined by the following recursion (using the Leibniz rule):
    \begin{equation}\label{eqn::chordal_jet_recursion}
\partial_j u_{\bs{\beta}}=
\begin{cases}
    u_{\bs{\beta}+\bs{e}_j}, &\text{if }\beta_j=0,\\[1ex]
    \displaystyle\frac{2}{\kappa}\partial^{\bs{\beta}-\bs{e}_j}
    \left(\frac{\lambda}{\kappa}u_{\bs{0}}
    -\sum_{\ell\ne j}\frac{2}{x_\ell-x_j}u_{\bs{e}_\ell}
    +\sum_{\ell\ne j}\frac{2\mathfrak{b}}{(x_\ell-x_j)^2}u_{\bs{0}}\right),
    &\text{if }\beta_j=1.
\end{cases}
\end{equation}

\begin{proposition}\label{prop::Dmatrix}
Fix $n\ge2$ and $\kappa>0$ and $\lambda\in\R$.
The recursion~\eqref{eqn::chordal_jet_recursion} uniquely determines the $2^n\times2^n$ matrices $D_j^{(\lambda)}(\bs{x})$, whose entries are real analytic on $\LX_n$.
For any $\bs{x}\in\LX_n$, the linear map
\begin{equation}\label{eqn::chordal_zeta}
    \zeta_{\bs{x}}:\LSchord{n}^{(\lambda)}\longrightarrow \R^{2^n},\qquad
    \LZ\longmapsto (\partial^{\bs{\beta}}\LZ(\bs{x}))_{\bs{\beta}\in\{0,1\}^n}
\end{equation} 
is an isomorphism.
In particular, $\dim\LSchord{n}^{(\lambda)}=2^n$.
\end{proposition}

\begin{proposition}\label{prop::Tmatrix}
Fix $n\ge2$ and $\kappa>0$ and $\lambda\in\R$. Define 
    \begin{equation}\label{eqn::Tmatrix_def}
       T^{(\lambda)}(\bs{x}):=\kappa\sum_{j=1}^n D^{(\lambda)}_j(\bs{x}).
    \end{equation}
    For each $\bs{x}\in\LX_n$ and $\nu\in\R$, the restriction of $\zeta_{\bs{x}}$ in~\eqref{eqn::chordal_zeta} gives an isomorphism \[
    \zeta_{\bs{x}}:\LSchord{n}^{(\lambda;\nu)}\longrightarrow \ker\left(T^{(\lambda)}(\bs{x})-\nu I\right).\] 
\end{proposition}

\begin{proposition}\label{prop::Ematrix}
Fix $n\ge 2$ and $\kappa>0$ and $\lambda=\nu=0$. Define 
\begin{equation}\label{eqn::chordal_dilation_matrix}
\begin{aligned}
    E(\bs{x})
    &:=\kappa\left(\Lambda+\sum_{j=1}^n x_j D^{(0)}_j(\bs{x})\right),
    \qquad \Lambda:=\operatorname{diag}\bigl(|\bs{\beta}|\bigr)_{\bs{\beta}\in\{0,1\}^n}.
\end{aligned}
\end{equation}
For each $\bs{x}\in \LX_n$ and $\xi\in \mathbb{R}$, the restriction of $\zeta_{\bs{x}}$ in~\eqref{eqn::chordal_zeta} gives an isomorphism
\[\zeta_{\bs{x}}:\LSchord{n}^{(0;0;\xi)}\longrightarrow \ker T^{(0)}(\bs{x})\cap\ker(E(\bs{x})-\xi I).\]
\end{proposition}
\paragraph*{Organization.}
In Section~\ref{subsec::renormalization}, we prove Proposition~\ref{prop::Dmatrix}. We renormalize the BPZ system and make a triangular change of coordinates to obtain an equivalent first-order system with explicit coefficient matrices.
In Section~\ref{subsec::Ward_1st}, we prove Proposition~\ref{prop::Tmatrix}, express the translation matrix in the new coordinates, and complete the proof of Theorem~\ref{thm::BPZ_decomposition}.
We prove Theorem~\ref{thm::BPZ_decomposition_degenerate} in Section~\ref{subsec::BPZ_decomposition_degenerate}. 
In Section~\ref{subsec::Ward_2nd}, we prove Proposition~\ref{prop::Ematrix}, analyze subspaces selected by the second Ward identity and complete the proof of Theorem~\ref{thm::xi_p_decomposition}. 
We analyze the structure of the dilation operator in Section~\ref{subsec::dilation} using the associated $\mathfrak{sl}_2$ representation.
We analyze subspaces selected by the third Ward identity in Section~\ref{subsec::Ward_3rd} and complete the proof of Theorem~\ref{thm::UST_decomposition}.

The following table summarizes the operators, matrices and solution-space identifications used below. All matrices are evaluated at a fixed $\bs{x}\in\LX_n$, with this dependence suppressed. The BPZ row records the coefficient matrices of the associated first-order system, and the Ward rows identify the subspaces selected by successive Ward identities with eigenspaces or intersections of kernels. The last column gives the matrices in the coordinates $\bs v=M^{-1}\bs u$ introduced in Lemma~\ref{lem::chordal_triangular_change}.
\begin{equation*}
\begin{array}{|c|c|c|c|c|}
\hline
&&&&\\
&\text{Operator}&\text{Matrix}&\text{Isomorphism via $\zeta_{\boldsymbol x}$}
&\begin{gathered}\text{Change of}\\\text{coordinates}\end{gathered}\\
\hline
&&&&\\
\text{BPZ}& \LD_j &D_j^{(\lambda)}&\LSchord{n}^{(\lambda)}\simeq\mathbb R^{2^n}&\widetilde{D}_j^{(\lambda)}\\
\lambda\in\R&\text{Eq.~\eqref{eqn::chordal_BPZ_operators}}&\text{Prop.~\ref{prop::Dmatrix}}&\text{Prop.~\ref{prop::Dmatrix}}&\text{Lemma~\ref{lem::chordal_renormalized_first_order}}\\
\hline
&&&&\\
\text{1st Ward}&\kappa\LT& T^{(\lambda)}&\LSchord{n}^{(\lambda;\nu)}\simeq  \ker(T^{(\lambda)}-\nu I)&\widetilde{T}^{(\lambda)}\\
\lambda, \nu\in\R&\text{Eq.~\eqref{eqn::chordal_Ward_12}}&\text{Prop.~\ref{prop::Tmatrix}}&\text{Prop.~\ref{prop::Tmatrix}}&\text{Lemma~\ref{lem::Tlambda_tilde}}\\
\hline
&&&&\\
\text{2nd Ward}&\kappa\LE& E &\LSchord{n}^{(0;0;\xi)}\simeq \ker T^{(0)}\cap\ker(E-\xi I)&\widetilde{E}\\
\lambda=\nu=0&\text{Eq.~\eqref{eqn::chordal_Ward_12}}&\text{Prop.~\ref{prop::Ematrix}}&\text{Prop.~\ref{prop::Ematrix}}&\text{Lemma~\ref{lem::Ematrix_tilde}}\\
\hline
&&&&\\
\text{3rd Ward}&\LW &W
& \LSchord{n}^{\mathrm{UST}}\simeq
\ker T^{(0)} \cap\ker(E-nI)\cap\ker W
&\widetilde{W}\\
\lambda=\nu=0, \xi=\xi_0&\text{Eq.~\eqref{eqn::chordal_third_Ward_operator}}&\text{Prop.~\ref{prop::Wmatrix}}&\text{Prop.~\ref{prop::Wmatrix}}&\text{Lemma~\ref{lem::Wmatrix_tilde}}\\
\hline
\end{array}
\end{equation*}

\subsection{First-order formulations of the BPZ system}
\label{subsec::PDE_pre}
\label{subsec::renormalization}
We prove Proposition~\ref{prop::Dmatrix} by analyzing two first-order formulations of the BPZ system.
We begin with the system of square-free derivatives introduced in~\eqref{eqn::chordal_first_order_system}.

\begin{lemma}\label{lem::chordal_first_order_system}
Fix $n\ge2$ and $\kappa>0$ and $\lambda\in\R$.
\begin{itemize}
\item The recursion~\eqref{eqn::chordal_jet_recursion} uniquely determines the $2^n\times2^n$ matrices $D_j^{(\lambda)}(\bs{x})$, whose entries are real analytic on $\LX_n$.
\item Every $\LZ\in\LSchord{n}^{(\lambda)}$ is real analytic. The map
\[
    \LZ\longmapsto (\partial^{\bs{\beta}}\LZ)_{\bs{\beta}\in\{0,1\}^n}
\]
is a bijection from $\LSchord{n}^{(\lambda)}$ onto the smooth solutions of~\eqref{eqn::chordal_first_order_system}, with inverse $\bs{u}\mapsto u_{\bs{0}}$.
\item For any $\bs{x}\in\LX_n$, the map $\zeta_{\bs{x}}$ defined in~\eqref{eqn::chordal_zeta} is injective.
\end{itemize}
\end{lemma}

\begin{proof}
We write $D_j = D_j^{(\lambda)}(\bs{x})$. Summing the BPZ equations gives an elliptic equation with analytic coefficients on $\LX_n$. Elliptic regularity therefore implies that every $\LZ\in\LSchord{n}^{(\lambda)}$ is real analytic, as in~\cite[Lemma~4.1]{HuangWuCommutationRadial}.
Solving the $j$-th BPZ equation for $\partial_j^2\LZ$ and applying the Leibniz rule gives~\eqref{eqn::chordal_jet_recursion}.
This lowers the total derivative order, so the recursion terminates and uniquely determines the analytic coefficient matrices $D_j$, as in~\cite[Lemma~4.2]{HuangWuCommutationRadial}.
Conversely, for any smooth solution $\bs{u}$ of the first-order system, the equation~\eqref{eqn::chordal_jet_recursion} gives $u_{\bs{\beta}}=\partial^{\bs{\beta}}u_{\bs{0}}$, and $\partial_j u_{\bs{\beta}}$ with $\bs{\beta}=\bs{e}_j$ gives the BPZ equation for $u_{\bs{0}}$.
This proves the asserted equivalence.
\medbreak
To prove injectivity, fix $\bs{x}_0\in\LX_n$ and suppose that $\zeta_{\bs{x}_0}(\LZ)=0$.
Set $\bs{u}=(\partial^{\bs{\beta}}\LZ)_{\bs{\beta}\in\{0,1\}^n}$.
For any $\bs{y}\in\LX_n$, the segment $\bs{\gamma}(t):=\bs{x}_0+t(\bs{y}-\bs{x}_0)$, $0\le t\le1$, lies in $\LX_n$ since $\LX_n$ is convex.
By~\eqref{eqn::chordal_first_order_system},
\[
    \frac{\ud }{\ud t}\bs{u}(\bs{\gamma}(t))
    =\left(\sum_{j=1}^n(y_j-x_{0,j})D_j(\bs{\gamma}(t))\right)
    \bs{u}(\bs{\gamma}(t)),
    \qquad \bs{u}(\bs{\gamma}(0))=0.
\]
Uniqueness for this linear ODE gives $\bs{u}(\bs{y})=0$.
Since $\bs{y}$ is arbitrary, we have $\LZ=0$, and this proves injectivity.
\end{proof}

To complete the proof of Proposition~\ref{prop::Dmatrix}, we construct an equivalent first-order system with explicit coefficient matrices and relating its coordinates to the original derivative coordinates. 

\paragraph*{Renormalization.}
The renormalization of the BPZ equations uses the half-watermelon partition function (see Appendix~\ref{subsec::halfwatermelon_SLE}):
\begin{align}\label{eqn::LZshuffle_def}
    \LZ_{\shuffle_n}(\bs{x}):=\LZ_{\shuffle_n}(\HH;\bs{x})
    =\prod_{j<\ell}(x_\ell-x_j)^{2/\kappa},\qquad \text{for }\bs{x}=(x_1, \ldots, x_n)\in\LX_n.
\end{align}
Note that $\LD_j\LZ_{\shuffle_n}=0$ for every $j$. We write $\LZ=\LZ_{\shuffle_n}\LF$, then a direct calculation gives
\begin{equation}\label{eqn::chordal_renormalized_BPZ}
    \LZ_{\shuffle_n}^{-1}\LD_j(\LZ_{\shuffle_n}\LF)
    =\frac{\kappa}{2}\partial_j^2\LF
    +2\sum_{\ell\ne j}\frac{\partial_j\LF-\partial_\ell \LF}{x_j-x_\ell}.
\end{equation}
Thus $\LZ=\LZ_{\shuffle_n}\LF$ satisfies the chordal BPZ system~\eqref{eqn::BPZ_chordal} if and only if $\LF$ satisfies
\begin{equation}\label{eqn::chordal_renormalized_system}
    \frac{\kappa}{2}\partial_j^2\LF
    +2\sum_{\ell\ne j}\frac{\partial_j\LF-\partial_\ell \LF}{x_j-x_\ell}
    =\frac{\lambda}{\kappa}\LF,\qquad j\in\{1,\ldots,n\}.
\end{equation}

The original first-order system~\eqref{eqn::chordal_first_order_system} uses the square-free derivatives of $\LZ$ as coordinates.
For the renormalized system~\eqref{eqn::chordal_renormalized_system}, we use triangular combinations of the square-free derivatives of $\LF$.
This choice leads to the explicit matrices below.
Identify $\bs{\beta}\in\{0,1\}^n$ with the subset $S=\{j:\beta_j=1\}$, write $\partial_S=\prod_{j\in S}\partial_j$, and index the coordinates of $\R^{2^n}$ by these subsets.
Throughout, we order subsets by increasing cardinality, with an arbitrary fixed order among subsets of the same size.


For $\bs{v}=(v_S)_S$, define the constant matrices $L_j^{(\lambda)}$ and $P_{j\ell}$ by
\begin{equation}\label{eqn::def_of_L_operator}
    \left(L_j^{(\lambda)}\bs{v}\right)_S=
    \begin{cases}
        v_{S\cup\{j\}},&j\notin S,\\[1ex]
        \displaystyle\frac{2\lambda}{\kappa^2}v_{S\setminus\{j\}},&j\in S,
    \end{cases}
    \qquad
    (P_{j\ell}\bs{v})_S=v_{\tau_{j\ell}(S)},
\end{equation}
where $\tau_{j\ell}(S)$ is obtained from $S$ by exchanging the labels $j$ and $\ell$.
Set
\begin{equation}\label{eqn::chordal_explicit_connection}
    \widetilde{D}^{(\lambda)}_j(\bs{x}):=L_j^{(\lambda)}+\frac{4}{\kappa}
    \sum_{\ell\ne j}\frac{P_{j\ell}-I}{x_j-x_\ell}.
\end{equation}
\begin{lemma}\label{lem::chordal_renormalized_first_order}
Fix $n\ge2$ and $\kappa>0$ and $\lambda\in\R$.
Consider the first-order system
\begin{equation}\label{eqn::chordal_first_order_system2}
    \partial_j\bs{v}(\bs{x})=\widetilde D_j^{(\lambda)}(\bs{x})\bs{v}(\bs{x}),
    \qquad \text{for }j\in\{1,\ldots,n\}.
\end{equation}
Its component equations are
\begin{equation}\label{eqn::recursion_of_v}
    \partial_jv_S=\mathbf{1}_{\{j\notin S\}}v_{S\cup\{j\}}
    +\frac{2\lambda}{\kappa^2}\mathbf{1}_{\{j\in S\}}v_{S\setminus\{j\}}
    +\frac{4}{\kappa}\sum_{\ell\ne j}
    \frac{v_{\tau_{j\ell}(S)}-v_S}{x_j-x_\ell}.
\end{equation}
\begin{itemize}
    \item The matrices $\widetilde D_j^{(\lambda)}(\bs{x})$ satisfy
    \begin{equation}\label{eqn::chordal_explicit_flatness}
    \partial_i\widetilde D^{(\lambda)}_j(\bs{x})-\partial_j\widetilde D^{(\lambda)}_i(\bs{x})=0,\qquad 
    \left[\widetilde D^{(\lambda)}_j(\bs{x}),\widetilde D^{(\lambda)}_i(\bs{x})\right]=0,\qquad \text{for }i\ne j.
\end{equation}
Thus, the first system~\eqref{eqn::chordal_first_order_system2} is compatible.
    \item The function $\LF$ satisfies the renormalized BPZ system~\eqref{eqn::chordal_renormalized_system} if and only if $\bs{v}=(\mathcal J_S\LF)_S$ satisfies~\eqref{eqn::recursion_of_v}, where $\mathcal J_\varnothing=I$ and, for each nonempty $S$,
\begin{equation}\label{eqn::chordal_corrected_jet_recursion}
    \mathcal J_S:=\partial_j\circ\mathcal J_{S\setminus\{j\}}
    -\frac{4}{\kappa}\sum_{\ell\in S\setminus\{j\}}
    \frac{1}{x_j-x_\ell}
    \left(\mathcal J_{S\setminus\{\ell\}}-\mathcal J_{S\setminus\{j\}}\right),
    \qquad j=\min S.
\end{equation}
\end{itemize}
\end{lemma}

\begin{proof}
We write $\widetilde{D}_j=\widetilde{D}_j^{(\lambda)}(\bs{x})$ and $L_j=L_j^{(\lambda)}$.
We first prove compatibility. The definitions of $L_j$ and $P_{j\ell}$ give
\[
\begin{aligned}
    &[L_i,L_j]=0,\qquad [L_i+L_j,P_{ij}]=0, \qquad \text{for }i\ne j;\\
    &[L_i,P_{jk}]=0,\qquad [P_{ij},P_{ik}+P_{jk}]=0, 
    \qquad \text{for }i,j,k\text{ distinct};\\
    &[P_{ij},P_{k\ell}]=0, \qquad \text{for }i,j,k,\ell\text{ distinct}.
\end{aligned}
\]
Expanding the commutator and grouping the terms with three distinct indices, we obtain
\begin{align}
    \left[\widetilde D_i,\widetilde D_j\right]=\frac{16}{\kappa^2}\sum_{k\ne i,j}[P_{ij},P_{jk}] \times\left(
        \frac{1}{(x_i-x_j)(x_j-x_k)}
        -\frac{1}{(x_i-x_k)(x_i-x_j)}
        -\frac{1}{(x_i-x_k)(x_j-x_k)}
    \right)=0. \label{eqn::chordal_explicit_commutator}
\end{align}
Moreover,
\[
    \partial_i\widetilde D_j=\partial_j\widetilde D_i
    =\frac{4}{\kappa}\frac{P_{ij}-I}{(x_i-x_j)^2},
    \qquad \text{for }i\ne j.
\]
Hence, we obtain~\eqref{eqn::chordal_explicit_flatness} as desired.
Since $\LX_n$ is convex and the matrices $\widetilde D_j$ are smooth there, the condition~\eqref{eqn::chordal_explicit_flatness} guarantees a unique smooth solution of~\eqref{eqn::chordal_first_order_system2} on $\LX_n$ for every prescribed value at any point of $\LX_n$.
\medbreak
We next prove the equivalence with the renormalized BPZ system.
Let $\bs{v}$ solve~\eqref{eqn::chordal_first_order_system2}. Its component equations for $\varnothing$ and $\{j\}$ are
\[
    \partial_jv_\varnothing=v_{\{j\}},\qquad
    \partial_jv_{\{j\}}=\frac{2\lambda}{\kappa^2}v_\varnothing
    -\frac{4}{\kappa}\sum_{\ell\ne j}
    \frac{v_{\{j\}}-v_{\{\ell\}}}{x_j-x_\ell}.
\]
Thus $\LF=v_\varnothing$ satisfies~\eqref{eqn::chordal_renormalized_system}.
For each nonempty $S$, the component equation~\eqref{eqn::recursion_of_v} for $S\setminus\{j\}$, with $j=\min S$, determines $v_S$ from the components indexed by proper subsets of $S$, exactly as in~\eqref{eqn::chordal_corrected_jet_recursion}.
Induction therefore gives $v_S=\mathcal J_S\LF$. In particular, every solution of the first-order system is uniquely determined by its component $\LF=v_\varnothing$, which solves the renormalized BPZ system.

Conversely, let $\LF$ solve~\eqref{eqn::chordal_renormalized_system}, and put $\bs v=(\mathcal J_S\LF)_S$.
We must show that this vector satisfies the full first-order system.
Induction on $|S|$ in~\eqref{eqn::chordal_corrected_jet_recursion} gives
\begin{equation}\label{eqn::chordal_triangular_change_of_v}
    \mathcal J_S=\partial_S+\sum_{U\subsetneq S}b_{S,U}(\bs{x})\,\partial_U,
\end{equation}
where the coefficients $b_{S,U}$ are smooth on $\LX_n$.
Thus the change from $(\partial_S \LF)_S$ to $\bs{v} = (\LJ_S \LF)_S$ is a lower triangular change of coordinates with unit diagonal.

Fix $\bs{x}_0\in\LX_n$. By compatibility, there is a solution $\widehat{\bs{v}}$ of~\eqref{eqn::chordal_first_order_system2} with $\widehat{\bs{v}}(\bs{x}_0)=\bs{v}(\bs{x}_0)$.
By the preceding argument, $\widehat \LF:=\widehat v_\varnothing$ solves~\eqref{eqn::chordal_renormalized_system}, and $\widehat v_S=\mathcal J_S\widehat\LF$.
The equality of the initial vectors and the triangular relation~\eqref{eqn::chordal_triangular_change_of_v} imply
\[
    \partial_S\widehat \LF(\bs{x}_0)=\partial_S\LF(\bs{x}_0)
    \qquad\text{for every }S.
\]
It follows that $\LZ_{\shuffle_n}\widehat \LF$ and $\LZ_{\shuffle_n}\LF$ are solutions of the original BPZ system with the same square-free derivatives at $\bs{x}_0$.
The injectivity of $\zeta_{\bs{x}_0}$ in~\eqref{eqn::chordal_zeta} therefore gives $\widehat \LF=\LF$.
Thus $\widehat{\bs{v}}=(\mathcal J_S\widehat\LF)_S=(\mathcal J_S\LF)_S=\bs v$, so $\bs v$ satisfies the full first-order system. This proves the converse implication and thus asserted isomorphism.
\end{proof}

The renormalization $\LZ=\LZ_{\shuffle_n}\LF$, together with the triangular relation~\eqref{eqn::chordal_triangular_change_of_v}, gives an invertible lower triangular change from the coordinates $\bs v$ to $\bs u$. The following lemma records this change and the induced relation between the coefficient matrices of the two first-order systems.

\begin{lemma}\label{lem::chordal_triangular_change}
For each $\bs{x}\in \LX_n$, there is a smooth invertible lower triangular matrix $M(\bs{x})$, with diagonal entries $M_{S,S}=\LZ_{\shuffle_n}(\bs{x})$, such that for every smooth function $\LF$ and $\LZ=\LZ_{\shuffle_n}\LF$,
\begin{equation}\label{eqn::chordal_triangular_change}
    \bs{u}(\bs{x})=(\partial_S\LZ(\bs{x}))_S=M(\bs{x})\bs{v}(\bs{x}),
\end{equation}
where $\bs v=(\mathcal J_S\LF)_S$.
Moreover, the matrices of the two first-order systems satisfy
\begin{equation}\label{eqn::chordal_connection_change}
    \widetilde{D}^{(\lambda)}_j(\bs{x})=M^{-1}(\bs{x})D^{(\lambda)}_j(\bs{x})M(\bs{x})-M^{-1}(\bs{x})(\partial_jM(\bs{x})),
    \qquad \text{for }j\in\{1,\ldots,n\}.
\end{equation}
\end{lemma}

\begin{proof}
We write $D_j=D^{(\lambda)}_j(\bs{x})$ and $\widetilde{D}_j=\widetilde{D}^{(\lambda)}_j(\bs{x})$.
Set $b_{S,S}:=1$. For an arbitrary smooth $\LZ=\LZ_{\shuffle_n}\LF$, applying~\eqref{eqn::chordal_triangular_change_of_v} to $\LF$ and using the product rule gives
\[
\begin{aligned}
    v_S=\mathcal J_S\LF
    =\sum_{U\subseteq S}b_{S,U}(\bs{x})\,
      \partial_U(\LZ_{\shuffle_n}^{-1}\LZ)=\sum_{U\subseteq S}b_{S,U}(\bs{x})
      \sum_{R\subseteq U}\partial_{U\setminus R}(\LZ_{\shuffle_n}^{-1})\,u_R.
\end{aligned}
\]
Collecting the coefficients of $u_R$, we obtain $\bs v=M^{-1}\bs u$, where
\begin{equation}\label{eqn::chordal_triangular_definition_of_M}
    (M^{-1})_{S,R}(\bs{x}):=
    \begin{cases}
        \displaystyle\sum_{R\subseteq U\subseteq S}
        b_{S,U}(\bs{x})\,\partial_{U\setminus R}(\LZ_{\shuffle_n}^{-1}),&R\subseteq S,\\[1ex]
        0,&R\not\subseteq S.
    \end{cases}
\end{equation}
This matrix is smooth and lower triangular, with diagonal entries $\LZ_{\shuffle_n}^{-1}>0$. Its inverse $M$ therefore has the asserted properties, proving~\eqref{eqn::chordal_triangular_change} for arbitrary smooth $\LF$.
\medbreak
Now suppose that $\LZ$ solves the chordal BPZ system~\eqref{eqn::BPZ_chordal}. Differentiating~\eqref{eqn::chordal_triangular_change} yields
\[
    D_jM\bs{v}=\partial_j\bs{u}
    =\left(\partial_jM+M\widetilde D_j\right)\bs{v}.
\]
By compatibility in Lemma~\ref{lem::chordal_renormalized_first_order}, solutions $\bs{v}$ can attain arbitrary values at any point of $\LX_n$.
Hence $D_jM=\partial_jM+M\widetilde D_j$, which gives~\eqref{eqn::chordal_connection_change}.
\end{proof}

\begin{proof}[Proof of Proposition~\ref{prop::Dmatrix}]
The construction of the matrices and the injectivity of $\zeta_{\bs{x}}$ are established in Lemma~\ref{lem::chordal_first_order_system}.
Fix $\bs{x}_0\in\LX_n$ and $\bs{a}\in\R^{2^n}$.
By Lemma~\ref{lem::chordal_renormalized_first_order}, the explicit system~\eqref{eqn::chordal_first_order_system2} has a solution $\bs{v}$ with
\[
    \bs{v}(\bs{x}_0)=M(\bs{x}_0)^{-1}\bs{a}.
\]
Set $\LF=v_\varnothing$ and $\LZ=\LZ_{\shuffle_n}\LF$.
Then $\LZ\in\LSchord{n}^{(\lambda)}$, and Lemma~\ref{lem::chordal_triangular_change} gives
\[
    \zeta_{\bs{x}_0}(\LZ)=M(\bs{x}_0)\bs{v}(\bs{x}_0)=\bs{a}.
\]
Thus $\zeta_{\bs{x}_0}$ is surjective.
Together with injectivity, this proves that $\zeta_{\bs{x}_0}$ is an isomorphism and gives $\dim\LSchord{n}^{(\lambda)}=2^n$.
\end{proof}

\subsection{Subspaces selected by the first Ward identity}
\label{subsec::Ward_1st}
In this subsection, we fix $n\ge2$, $\kappa>0$ and $\lambda\in\R$.
We first prove Proposition~\ref{prop::Tmatrix} by identifying $T^{(\lambda)}(\bs{x})$ as the representation of $\kappa\LT$ in the derivative coordinates.
We then use the triangular change of coordinates to obtain the explicit constant matrix $\widetilde T^{(\lambda)}$ in Lemma~\ref{lem::Tlambda_tilde}.
For $\lambda\ne0$, its eigenbasis and characteristic polynomial are computed in Lemma~\ref{lem::Tlambda_tilde_polynomial}, giving the decomposition in Theorem~\ref{thm::BPZ_decomposition}.

\begin{proof}[Proof of Proposition~\ref{prop::Tmatrix}]
We write $D_j=D^{(\lambda)}_j(\bs{x})$ and $T(\bs{x})=T^{(\lambda)}(\bs{x})$. By~\eqref{eqn::BPZ_dilation}, for every $\LZ\in\LSchord{n}^{(\lambda)}$, the function $\LT\LZ$ also satisfies the chordal BPZ system~\eqref{eqn::BPZ_chordal}. Since $\LT$ commutes with every $\partial^{\bs{\beta}}$, the first-order system~\eqref{eqn::chordal_first_order_system} gives
\begin{equation}\label{eqn::chordal_Ward_intertwining}
\begin{aligned}
    \zeta_{\bs{x}}(\kappa\LT\LZ)
    &=\kappa\bigl((\LT\partial^{\bs{\beta}}\LZ)(\bs{x})\bigr)_{\bs{\beta}\in\{0,1\}^n}=\kappa\sum_{j=1}^n D_j(\bs{x})\zeta_{\bs{x}}(\LZ)
    =T(\bs{x})\zeta_{\bs{x}}(\LZ).
\end{aligned}
\end{equation}
As $\zeta_{\bs{x}}$ is an isomorphism, this identity yields
\begin{equation}\label{eqn::chordal_Ward_representation}
    T(\bs{x})=\zeta_{\bs{x}}\circ(\kappa\LT)\circ\zeta_{\bs{x}}^{-1}.
\end{equation}
Thus $T(\bs{x})$ represents $\kappa\LT$ in the coordinates given by $\zeta_{\bs{x}}$.
\medbreak
Fix $\bs{x}\in\LX_n$ and $\nu\in\R$. If $\LZ\in\LSchord{n}^{(\lambda;\nu)}$, then $\kappa\LT\LZ=\nu\LZ$, and~\eqref{eqn::chordal_Ward_intertwining} implies
\[
    (T(\bs{x})-\nu I)\zeta_{\bs{x}}(\LZ)
    =\zeta_{\bs{x}}\bigl((\kappa\LT-\nu)\LZ\bigr)=0.
\]
Hence $\zeta_{\bs{x}}$ maps $\LSchord{n}^{(\lambda;\nu)}$ into $\ker(T(\bs{x})-\nu I)$, and its restriction is injective because $\zeta_{\bs{x}}$ is injective on $\LSchord{n}^{(\lambda)}$.

Conversely, let $\bs{u}\in\ker(T(\bs{x})-\nu I)$. The isomorphism~\eqref{eqn::chordal_zeta} gives a solution $\LZ\in\LSchord{n}^{(\lambda)}$ with $\zeta_{\bs{x}}(\LZ)=\bs{u}$. As $\LT$ preserves this solution space, $(\kappa\LT-\nu)\LZ\in\LSchord{n}^{(\lambda)}$, and~\eqref{eqn::chordal_Ward_intertwining} gives
\[
    \zeta_{\bs{x}}\bigl((\kappa\LT-\nu)\LZ\bigr)
    =(T(\bs{x})-\nu I)\bs{u}=0.
\]
Injectivity of $\zeta_{\bs{x}}$ implies $(\kappa\LT-\nu)\LZ=0$ on $\LX_n$. Thus $\LZ\in\LSchord{n}^{(\lambda;\nu)}$, proving surjectivity.
\end{proof}

\begin{lemma}\label{lem::Tlambda_tilde}
For every $\bs{x}\in\LX_n$, let $M(\bs{x})$ be the matrix from Lemma~\ref{lem::chordal_triangular_change}. Then
\begin{equation}\label{eqn::chordal_Ward_similarity}
    \widetilde{T}^{(\lambda)}:=M(\bs{x})^{-1}T^{(\lambda)}(\bs{x})M(\bs{x})=\kappa\sum_{j=1}^nL_j^{(\lambda)}.
\end{equation}
\end{lemma}
\begin{proof}
The coefficients in~\eqref{eqn::chordal_corrected_jet_recursion} depend only on coordinate differences, and differentiation preserves this property. Induction shows that the coefficients $b_{S,U}(\bs{x})$ in~\eqref{eqn::chordal_triangular_change_of_v} have the same invariance.
The function $\LZ_{\shuffle_n}(\bs{x})$ is also translation invariant, so the formula~\eqref{eqn::chordal_triangular_definition_of_M} shows that $M^{-1}(\bs{x})$, and hence $M(\bs{x})$, is translation invariant. Thus
\[
    \sum_{j=1}^n\partial_jM(\bs{x})=0.
\]
Summing~\eqref{eqn::chordal_connection_change} and using $P_{j\ell}=P_{\ell j}$, we obtain
\[
\begin{aligned}
    M^{-1}(\bs{x})T^{(\lambda)}(\bs{x})M(\bs{x})
    &=\kappa\sum_{j=1}^n\widetilde{D}^{(\lambda)}_j(\bs{x})
      +\kappa M^{-1}(\bs{x})\sum_{j=1}^n\partial_jM(\bs{x})\\
    &=\kappa\sum_{j=1}^nL_j^{(\lambda)}
      +4\sum_{j<\ell}(P_{j\ell}-I)
      \left(\frac{1}{x_j-x_\ell}+\frac{1}{x_\ell-x_j}\right)\\
    &=\kappa\sum_{j=1}^nL_j^{(\lambda)}.
\end{aligned}
\]
\end{proof}

We now compute the spectrum of this constant matrix for $\lambda\ne0$.
\begin{lemma}\label{lem::Tlambda_tilde_polynomial}
Fix $\lambda\neq 0$.
The matrix $\widetilde{T}^{(\lambda)}$ in~\eqref{eqn::chordal_Ward_similarity} is diagonalizable with the characteristic polynomial
\begin{equation}\label{eqn::chordal_Ward_characteristic_polynomial}
    \det\left(t I-\widetilde{T}^{(\lambda)}\right) = \prod_{\mathfrak{r}=0}^n \left(t-(n-2\mathfrak{r})\sqrt{2\lambda}\right)^{\binom{n}{\mathfrak{r}}}.
\end{equation}
\end{lemma}
\begin{proof}
For each $\bs{\sigma}\in\{\pm1\}^n$, define the vector $\bs{v}_{\bs{\sigma}}$ by its components
\begin{equation}\label{eqn::chordal_explicit_eigenvectors}
    (\bs{v}_{\bs{\sigma}})_{\bs{\beta}}
    :=\prod_{i=1}^n\left(\frac{\sqrt{2\lambda}}{\kappa}\sigma_i\right)^{\beta_i},
    \qquad \bs{\beta}\in\{0,1\}^n.
\end{equation}
It is directly verified from~\eqref{eqn::def_of_L_operator} that
\begin{equation}\label{eqn::chordal_explicit_joint_spectrum}
    L_j^{(\lambda)}\bs{v}_{\bs{\sigma}}
    =\frac{\sqrt{2\lambda}}{\kappa}\sigma_j\bs{v}_{\bs{\sigma}},
    \qquad j\in\{1,\ldots,n\}.
\end{equation}
Since $\widetilde{T}^{(\lambda)}=\kappa\sum_jL_j^{(\lambda)}$, we obtain
\begin{equation}\label{eqn::chordal_explicit_Ward_spectrum}
    \widetilde{T}^{(\lambda)}\bs{v}_{\bs{\sigma}}
    =\sqrt{2\lambda}\left(\sum_{j=1}^n\sigma_j\right)\bs{v}_{\bs{\sigma}}.
\end{equation}
It follows directly from the component formula~\eqref{eqn::chordal_explicit_eigenvectors} that the $2^n$ vectors $\bs{v}_{\bs{\sigma}}$ are linearly independent when $\lambda\ne0$.
Thus they form a basis of eigenvectors of $\widetilde{T}^{(\lambda)}$ over $\C$, and the eigenvalue $(n-2\mathfrak{r})\sqrt{2\lambda}$ has multiplicity $\binom{n}{\mathfrak{r}}$.
\end{proof}

\begin{proof}[Proof of Theorem~\ref{thm::BPZ_decomposition}]
The dimension of $\LSchord{n}^{(\lambda)}$ is derived in Proposition~\ref{prop::Tmatrix}. 
Fix $\bs{x}\in\LX_n$. By Lemma~\ref{lem::Tlambda_tilde} and Lemma~\ref{lem::Tlambda_tilde_polynomial}, when $\lambda\ne0$, the matrix $T^{(\lambda)}(\bs{x})$ is diagonalizable over $\C$, with eigenvalues
\[
    \nu_{\mathfrak{r}}=(n-2\mathfrak{r})\sqrt{2\lambda},
    \qquad \mathfrak{r}\in\{0,\ldots,n\},
\]
of multiplicities $\binom{n}{\mathfrak{r}}$.
\begin{itemize}
\item When $\lambda>0$, these eigenvalues are real, and the isomorphism in Proposition~\ref{prop::Tmatrix} therefore gives~\eqref{eqn::lambda_nu_dimension} and~\eqref{eqn::lambda_nu_decomposition}.
\item When $\lambda<0$, the only possible real eigenvalue is $0$, which occurs precisely when $n$ is even.
In that case, we have $\dim\ker T^{(\lambda)}(\bs{x})=\binom{n}{n/2}$.
Proposition~\ref{prop::Tmatrix} now gives~\eqref{eqn::lambda_negative_nu_dimension}.
\end{itemize}
The assertion about zeros when $\lambda<0$ follows from Lemma~\ref{lem::chordal_negative_lambda_zero} below.
\end{proof}

\begin{lemma}\label{lem::chordal_negative_lambda_zero}
Fix $\lambda<0$. Any solution $\LZ\in\LSchord{n}^{(\lambda)}$ has at least one zero in $\LX_n$.
\end{lemma}
\begin{proof}
Suppose that $\LZ$ has no zero in $\LX_n$. Since $\LX_n$ is connected, replacing $\LZ$ by $-\LZ$ if necessary, we may assume that $\LZ>0$.
Write $\LZ=\LZ_{\shuffle_n}\LF$ as in Section~\ref{subsec::PDE_pre}, and $\LF>0$.
Summing the renormalized equations~\eqref{eqn::chordal_renormalized_system} and using
$\partial_j\log\LZ_{\shuffle_n}=(2/\kappa)\sum_{\ell\ne j}(x_j-x_\ell)^{-1}$ gives
\[
    \frac{\kappa^2}{2}\Delta \LF
    +2\kappa^2\nabla\log\LZ_{\shuffle_n}\cdot\nabla\LF=n\lambda \LF.
\]
Equivalently,
\begin{equation}\label{eqn::chordal_renormalized_divergence}
    \frac{\kappa^2}{2}\operatorname{div}(\LZ_{\shuffle_n}^4\nabla \LF)
    =n\lambda \LZ_{\shuffle_n}^4\LF.
\end{equation}
For any real-valued $f\in C_c^\infty(\LX_n)$, multiplying~\eqref{eqn::chordal_renormalized_divergence} by $f^2/\LF$ and integrating by parts gives
\begin{equation}\label{eqn::chordal_weighted_energy_identity}
\begin{aligned}
    n\lambda\int_{\LX_n}\LZ_{\shuffle_n}^4f^2\,\ud\bs{x}
    &=-\frac{\kappa^2}{2}\int_{\LX_n}
    \LZ_{\shuffle_n}^4\nabla\LF\cdot\nabla\left(\frac{f^2}{\LF}\right)\,\ud\bs{x}\\
    &=\frac{\kappa^2}{2}\int_{\LX_n}\LZ_{\shuffle_n}^4\LF^2
    \left|\nabla\left(\frac{f}{\LF}\right)\right|^2\,\ud\bs{x}
    -\frac{\kappa^2}{2}\int_{\LX_n}\LZ_{\shuffle_n}^4|\nabla f|^2\,\ud\bs{x}\\
    &\ge-\frac{\kappa^2}{2}\int_{\LX_n}\LZ_{\shuffle_n}^4|\nabla f|^2\,\ud\bs{x}.
\end{aligned}
\end{equation}
Fix a nonzero such $f$, and set $f_R(\bs{x}):=f(\bs{x}/R)$ for $R>0$.
Since $\LX_n$ is invariant under positive dilations and $\LZ_{\shuffle_n}(R\bs{x})=R^{n(n-1)/\kappa}\LZ_{\shuffle_n}(\bs{x})$, applying~\eqref{eqn::chordal_weighted_energy_identity} to $f_R$ and changing variables gives
\[
    n\lambda
    \ge-\frac{\kappa^2}{2}
    \frac{\displaystyle\int_{\LX_n}\LZ_{\shuffle_n}^4|\nabla f_R|^2\,\ud\bs{x}}
         {\displaystyle\int_{\LX_n}\LZ_{\shuffle_n}^4f_R^2\,\ud\bs{x}}
    =-\frac{\kappa^2}{2R^2}
    \frac{\displaystyle\int_{\LX_n}\LZ_{\shuffle_n}^4|\nabla f|^2\,\ud\bs{x}}
         {\displaystyle\int_{\LX_n}\LZ_{\shuffle_n}^4f^2\,\ud\bs{x}}.
\]
Letting $R\to\infty$ gives $n\lambda\ge0$, contradicting $\lambda<0$.
\end{proof}

\subsection{Proof of Theorem~\ref{thm::BPZ_decomposition_degenerate}}
\label{subsec::BPZ_decomposition_degenerate}
In this subsection, we fix $\kappa>0$ and $\lambda=0$.
Fix $n\ge2$ and recall that $\mathsf{P}_n=\{p\in\{0,1,\ldots, n\}: n-p\text{ is even}\}$ is defined in~\eqref{eqn::parityn_def}. For $p\in\mathsf{P}_n$, recall that $\mathsf{B}_{n,p}$ are defined in~\eqref{eqn::bnp_def}: 
\begin{align*}
\mathsf{B}_{n,p}:=\binom{n}{(n-p)/2}-\binom{n}{(n-p)/2-1}, 
\end{align*}
where $\binom{n}{-1}:=0$.
We compute the Jordan form of $\widetilde{T}^{(0)}$ in the following lemma and then deduce Theorem~\ref{thm::BPZ_decomposition_degenerate}.

\begin{lemma}\label{lem::chordal_Ward_Jordan_normal_form}
The matrix $\widetilde{T}^{(0)}$ in~\eqref{eqn::chordal_Ward_similarity} is conjugate over $\R$ to
\begin{equation}\label{eqn::chordal_explicit_Jordan_degenerate}
    \bigoplus_{p\in\mathsf{P}_n}J_{p+1}(0)^{\oplus\mathsf{B}_{n,p}},
\end{equation}
where $J_s(0)$ denotes a nilpotent Jordan block of size $s$.
\end{lemma}
The Jordan form~\eqref{eqn::chordal_explicit_Jordan_degenerate} is reminiscent of the iterated Clebsch--Gordan decomposition
\[
    (\C^2)^{\otimes n}
    \cong\bigoplus_{p\in\mathsf{P}_n}
    \bigl(\operatorname{Sym}^p(\C^2)\bigr)^{\oplus\mathsf{B}_{n,p}}
\]
as $\mathfrak{sl}_2(\C)$-modules.
Each summand $\operatorname{Sym}^p(\C^2)$ is irreducible of dimension $p+1$, and a nonzero nilpotent element of $\mathfrak{sl}_2(\C)$ acts on it as a single Jordan block of size $p+1$.
The proof below uses only elementary linear algebra.
\begin{proof}
We determine the Jordan form of $N_n:=\kappa^{-1}\widetilde{T}^{(0)}=\sum_{j=1}^nL_j^{(0)}$ by induction on $n$.
Grouping the coordinates according to the last index $\beta_n\in\{0,1\}$ gives the block decomposition
\begin{equation}\label{eqn::chordal_explicit_block_recursion}
    N_n=\begin{pmatrix}N_{n-1}&I_{2^{n-1}}\\0&N_{n-1}\end{pmatrix},
    \qquad N_1=J_2(0)=\begin{pmatrix}0&1\\0&0\end{pmatrix}.
\end{equation}
We show that $N_n$ is conjugate to~\eqref{eqn::chordal_explicit_Jordan_degenerate}.
The base case $n=1$ is clear. For the induction step, assume the claim for $N_{n-1}$.
Apply the change of basis putting $N_{n-1}$ in Jordan form to both copies in~\eqref{eqn::chordal_explicit_block_recursion}, and then reorder the basis vectors to obtain a block diagonal matrix with blocks
\[
    M_s:=\begin{pmatrix}J_s(0)&I_s\\0&J_s(0)\end{pmatrix}.
\]
The multiplicity of $M_s$ is the same as that of $J_s(0)$ in $N_{n-1}$, namely $\mathsf{B}_{n-1,s-1}$.
A direct calculation gives
\[
    M_s\sim J_{s+1}(0)\oplus J_{s-1}(0),
\]
with the second block omitted when $s=1$.
For $p\in\mathsf{P}_n$, the multiplicity of $J_{p+1}(0)$ in $N_n$ is therefore
\[
    \mathsf{B}_{n-1,p-1}+\mathsf{B}_{n-1,p+1}=\mathsf{B}_{n,p}.
\]
This proves the induction step.
Finally, since $\kappa J_s(0)$ is conjugate to $J_s(0)$ for $\kappa\ne0$, the same Jordan form holds for $\widetilde{T}^{(0)}=\kappa N_n$.
\end{proof}

\begin{proof}[Proof of Theorem~\ref{thm::BPZ_decomposition_degenerate}]
Fix $\bs{x}\in\LX_n$ and write $T=T^{(0)}(\bs{x})$. Iterating~\eqref{eqn::chordal_Ward_intertwining} gives
\[
    \zeta_{\bs{x}}\bigl((\kappa\LT)^q\LZ\bigr)
    =T^q\zeta_{\bs{x}}(\LZ).
\]
Since $\zeta_{\bs{x}}$ is an isomorphism and $\kappa\ne0$, its restriction identifies $\LSchord{n}^{\langle q\rangle}$ with $\ker T^q$. Lemma~\ref{lem::Tlambda_tilde} and Lemma~\ref{lem::chordal_Ward_Jordan_normal_form} therefore gives the asserted dimension formula.
\end{proof}

\subsection{Subspaces selected by the second Ward identity}
\label{subsec::Ward_2nd}
In this subsection, we fix $n\ge2$, $\kappa>0$ and $\lambda=\nu=0$.
We first prove Proposition~\ref{prop::Ematrix}, using $E(\bs{x})$ to represent the action of $\kappa\LE$ in the derivative coordinates.
We then conjugate $E(\bs{x})$ by $M(\bs{x})$ and restrict the resulting matrix to $\ker\widetilde T^{(0)}$ in Lemma~\ref{lem::Ematrix_tilde}.
The characteristic and annihilating polynomials of this restriction, $\widehat E(\bs{x})$, are stated in Proposition~\ref{prop::widehatE} and yield Theorem~\ref{thm::xi_p_decomposition}.

%

\begin{proof}[Proof of Proposition~\ref{prop::Ematrix}]
The proof is similar to that of Proposition~\ref{prop::Tmatrix}. 
Using~\eqref{eqn::BPZ_dilation} and~\eqref{eqn::translation_dilation}, for every $\LZ\in\LSchord{n}^{(0;0)}$, the function $\LE\LZ$ also belongs to $\LSchord{n}^{(0;0)}$. Since
\begin{equation}\label{eqn::chordal_dilation_commutator_with_partial}
    [\partial^{\bs{\beta}},\LE] = |\bs{\beta}|\partial^{\bs{\beta}},
\end{equation}
the first-order system~\eqref{eqn::chordal_first_order_system} gives
\begin{equation}\label{eqn::chordal_dilation_intertwining}
    \zeta_{\bs{x}}(\kappa\LE\LZ)
    =E(\bs{x})\zeta_{\bs{x}}(\LZ).
\end{equation}

Fix $\bs{x}\in\LX_n$ and $\xi\in\R$, and write $T=T^{(0)}(\bs{x})$ and $E=E(\bs{x})$. If $\LZ\in\LSchord{n}^{(0;0;\xi)}$, then Proposition~\ref{prop::Tmatrix} and~\eqref{eqn::chordal_dilation_intertwining} give
\[
    T\zeta_{\bs{x}}(\LZ)=0,
    \qquad (E-\xi I)\zeta_{\bs{x}}(\LZ)
    =\zeta_{\bs{x}}\bigl((\kappa\LE-\xi)\LZ\bigr)=0.
\]
Thus $\zeta_{\bs{x}}$ maps $\LSchord{n}^{(0;0;\xi)}$ injectively into the stated intersection.

Conversely, let $\bs{u}\in\ker T\cap\ker(E-\xi I)$. By Proposition~\ref{prop::Tmatrix}, there is a unique $\LZ\in\LSchord{n}^{(0;0)}$ with $\zeta_{\bs{x}}(\LZ)=\bs{u}$. Eq.~\eqref{eqn::chordal_dilation_intertwining} gives
\[
    \zeta_{\bs{x}}\bigl((\kappa\LE-\xi)\LZ\bigr)=(E-\xi I)\bs{u}=0.
\]
Since $(\kappa\LE-\xi)\LZ\in\LSchord{n}^{(0;0)}$, injectivity of $\zeta_{\bs{x}}$ implies $\kappa\LE\LZ=\xi\LZ$. Hence $\LZ\in\LSchord{n}^{(0;0;\xi)}$, proving surjectivity.
\end{proof}

\begin{lemma}\label{lem::Ematrix_tilde}
For every $\bs{x}\in\LX_n$, let $M(\bs{x})$ be the matrix from Lemma~\ref{lem::chordal_triangular_change}. Then
\begin{equation}\label{eqn::chordal_dilation_similarity}
\begin{aligned}
    \widetilde E(\bs{x}):=M(\bs{x})^{-1}E(\bs{x})M(\bs{x})&=\kappa\sum_{j=1}^n x_jL_j^{(0)}+\kappa \Lambda
    +4\sum_{j<\ell}\left(P_{j\ell}-\frac12I\right),
\end{aligned}
\end{equation}
where $E(\bs{x})$ and $\Lambda$ are defined in~\eqref{eqn::chordal_dilation_matrix}.
Moreover, $\widetilde E(\bs{x})$ preserves the subspace $\ker\widetilde{T}^{(0)}$, which is independent of $\bs{x}$ and $\kappa$. Define the restriction
\begin{equation}\label{eqn::chordal_dilation_restriction}
    \widehat E(\bs{x}):=\widetilde E(\bs{x})|_{\ker\widetilde{T}^{(0)}}:\ker\widetilde{T}^{(0)}\longrightarrow \ker\widetilde{T}^{(0)}.
\end{equation}
For each $\xi\in\R$, the map $\LZ\mapsto M(\bs{x})^{-1}\zeta_{\bs{x}}(\LZ)$ is an isomorphism from $\LSchord{n}^{(0;0;\xi)}$ onto $\ker\left(\widehat E(\bs{x})-\xi I\right)$. In particular,
\begin{equation}\label{eqn::chordal_dilation_dimension}
    \dim\LSchord{n}^{(0;0;\xi)}
    =\dim\ker\left(\widehat E(\bs{x})-\xi I\right).
\end{equation}
\end{lemma}
\begin{proof}
We write $T=T^{(0)}$, $\widetilde T=\widetilde T^{(0)}$, $\widetilde D_j=\widetilde D_j^{(0)}$ and $L_j=L_j^{(0)}$ throughout this proof.
Using the equations~\eqref{eqn::chordal_connection_change} and~\eqref{eqn::chordal_dilation_matrix}, we obtain
\[
    M^{-1}EM
    =\kappa\sum_{j=1}^n x_j\widetilde D_j+\kappa M^{-1}\Lambda M
      +\kappa M^{-1}(\LE M).
\]
We calculate these terms separately.
\begin{itemize}
    \item We first calculate the terms involving $M$.
    Put $\bs u=(\partial_S(\LZ_{\shuffle_n}\LF))_S$ and $\bs v=(\mathcal J_S\LF)_S$, so that $\bs u=M\bs v$ by Lemma~\ref{lem::chordal_triangular_change}.
    In the original coordinates, Eq.~\eqref{eqn::chordal_dilation_commutator_with_partial} gives
    \[
        \bigl(\partial_S(\LE(\LZ_{\shuffle_n}\LF))\bigr)_S
        =(\LE+\Lambda)\bs u=(\LE+\Lambda)(M\bs v) = \left(\LE M\right)\bs v+M\LE \bs{v} + \Lambda M\bs v.
    \]
    In the new coordinates, since $\LZ_{\shuffle_n}$ is homogeneous of degree $n(n-1)/\kappa$, for every smooth function $\LF$ we have
    \[
        \LZ_{\shuffle_n}^{-1}\LE(\LZ_{\shuffle_n}\LF)
        =\left(\LE+\frac{n(n-1)}{\kappa}\right)\LF.
    \]
    The identities $[\partial_j,\LE]=\partial_j$ and $[(x_j-x_\ell)^{-1},\LE]=(x_j-x_\ell)^{-1}$ give, by induction in~\eqref{eqn::chordal_corrected_jet_recursion},
    \[
        [\mathcal J_S,\LE]=|S|\mathcal J_S.
    \]
    Thus we obtain
    \[
    \begin{aligned}
        \bigl(\partial_S(\LE(\LZ_{\shuffle_n}\LF))\bigr)_S
        &=M\left(\mathcal J_S\left(\left(\LE+\frac{n(n-1)}{\kappa}\right)\LF\right)\right)_S\\
        &=M\left(\LE+\frac{n(n-1)}{\kappa}I+\Lambda\right)\bs v.
    \end{aligned}
    \]
    Compare the two expressions of  $\bigl(\partial_S(\LE(\LZ_{\shuffle_n}\LF))\bigr)_S$. Since the triangular relation~\eqref{eqn::chordal_triangular_change_of_v} allows $\bs v(\bs{x})$ to be prescribed arbitrarily, we have
    \begin{equation*}\label{eqn::chordal_dilation_of_M}
        \LE M+\Lambda M=\frac{n(n-1)}{\kappa}M+M\Lambda.
    \end{equation*}
    \item The explicit formula~\eqref{eqn::chordal_explicit_connection} and $P_{j\ell}=P_{\ell j}$ give
    \[
    \kappa\sum_{j=1}^n x_j\widetilde D_j = \kappa\sum_{j=1}^n x_jL_j+4\sum_{j<\ell}(P_{j\ell}-I).
    \]
\end{itemize}
Combining the two identities above proves~\eqref{eqn::chordal_dilation_similarity}.

At $\lambda=0$, the matrices $L_j$ are independent of $\bs{x}$ and $\kappa$, so the same holds for $\ker\widetilde{T}$.
For any $\bs{w}\in \ker\widetilde{T}$, there exists a unique $\LZ\in\LSchord{n}^{(0;0)}$ for which
\[
\bs{w} = M(\bs{x})^{-1}\zeta_{\bs{x}}(\LZ),
\]
and since $\LE$ preserves $\LSchord{n}^{(0;0)}$, the intertwining identity~\eqref{eqn::chordal_dilation_intertwining} gives
\[
    \widetilde E(\bs{x})M(\bs{x})^{-1}\zeta_{\bs{x}}(\LZ)
    =M(\bs{x})^{-1}\zeta_{\bs{x}}(\kappa\LE\LZ)\in \ker\widetilde{T}.
\]
Thus $\widetilde E(\bs{x})$ preserves $\ker\widetilde{T}$, and $\widehat E(\bs{x})$ is well-defined. Finally, for each $\xi\in\R$,
\[
    M^{-1}\bigl(\ker T\cap\ker(E-\xi I)\bigr)
    =\ker\widetilde{T}\cap\ker\left(\widetilde E-\xi I\right)
    =\ker\left(\widehat E-\xi I\right).
\]
Together with Proposition~\ref{prop::Ematrix}, this identity proves the asserted isomorphism and the dimension formula~\eqref{eqn::chordal_dilation_dimension}.
\end{proof}

To derive the dimension of $\LSchord{n}^{(0;0;\xi)}$, we will analyze the kernel of $\widehat{E}(\bs{x})-\xi I$. 
\begin{proposition}\label{prop::widehatE}
Fix $n\ge 2$ and $\kappa>0$ and $\lambda=\nu=0$.
Recall that $\xi_p$ are defined in~\eqref{eqn::xi_p_def}. 
The characteristic polynomial for $\widehat{E}(\bs{x})$ is 
\begin{equation}\label{eqn::chordal_dilation_characteristic}
    \det\left(tI-\widehat E(\bs{x})\right)
    =\prod_{p\in\mathsf P_n}(t-\xi_p)^{\mathsf B_{n,p}},
\end{equation}
and $\widehat{E}(\bs{x})$ is annihilated by the following polynomial:
\begin{equation}\label{eqn::chordal_dilation_annihilating}
    Q_\kappa(t):=\prod_{p\in\mathsf{P}_n}(t-\xi_p),\qquad Q_\kappa\bigl(\widehat E(\bs{x})\bigr)=0.
\end{equation}
\end{proposition}
We will prove Proposition~\ref{prop::widehatE} in Section~\ref{subsec::dilation}.
In fact, Proposition~\ref{prop::widehatE} could be proved using the solutions constructed in Proposition~\ref{prop::LZmix_BPZ}, see Remark~\ref{rem::E_polynomial_alternativeproof}. But as we need the detailed structure of the dilation operator for Theorem~\ref{thm::UST_decomposition}, we give the proof for Proposition~\ref{prop::widehatE} using the analysis in Section~\ref{subsec::dilation}.
Assuming Proposition~\ref{prop::widehatE}, we are ready to prove Theorem~\ref{thm::xi_p_decomposition}.

\begin{proof}[Proof of Theorem~\ref{thm::xi_p_decomposition}]
For $p_1<p_2$ in $\mathsf{P}_n$, we have
\[
    \xi_{p_2}-\xi_{p_1}=(p_2-p_1)\left(p_2+p_1+2-\frac{\kappa}{2}\right).
\]
Thus for $\kappa\notin \{2(p_1+p_2+2):p_1, p_2\in\mathsf{P}_n,\ p_1<p_2\}$, the exponents $\xi_p$ are distinct, and Theorem~\ref{thm::xi_p_decomposition} follows from Lemma~\ref{lem::Ematrix_tilde} and Proposition~\ref{prop::widehatE}.
\end{proof}

\begin{proof}[Proof of Corollary~\ref{cor::BPZ_Catalan}]
Eq.~\eqref{eqn::solution_Catalan_dimension} follows from Theorem~\ref{thm::xi_p_decomposition}.
The construction of pure partition functions $\{\LZ_\alpha:\alpha\in\LP_N\}$ using multiple SLEs of~\cite[Theorem~1.1]{PeltolaWuGlobalMultipleSLEs} and~\cite[Theorem~1.10]{FengLiuPeltolaWu2024} tells that they are linearly independent (due to the asymptotics), hence they form a basis of $\LSchord{n}^{(0;0;\xi_0)}$. 
Their M\"obius covariance implies that every function in this space satisfies the third Ward identity~\eqref{eqn::Ward_3rd}.  
\end{proof}

\begin{remark}\label{rem::xi_p_fail}
The conclusion in Theorem~\ref{thm::xi_p_decomposition} fails at $\kappa=8$ when $n=2$. For $n=2$ and $\kappa=8$, we have $\xi_0=\xi_2=2$ and
\[
    \LSchord{2}^{(0;0)}
    =\operatorname{span}_{\R}\{(x_2-x_1)^{1/4},\, (x_2-x_1)^{1/4}\log (x_2-x_1)\},\qquad 
    \LSchord{2}^{(0;0;2)}
    =\operatorname{span}_{\R}\{(x_2-x_1)^{1/4}\}.
\]
The eigenspace $ \LSchord{2}^{(0;0;2)}$ therefore does not span $\LSchord{2}^{(0;0)}$. The decomposition~\eqref{eqn::solution_xi_decomposition} in Theorem~\ref{thm::xi_p_decomposition} fails in this case.
\end{remark}

\subsection{Structure of the dilation operator}
\label{subsec::dilation}
We study the structure of the dilation operator through its matrix $\widehat E(\bs{x})$ on $\ker\widetilde T^{(0)}$. Decomposing this kernel into the highest-weight spaces of an $\mathfrak{sl}_2$ representation gives a block upper triangular form with scalar diagonal blocks and a lowering operator linking adjacent spaces. This description yields the polynomial identities in Proposition~\ref{prop::widehatE}. It also allows us to determine the eigenspace at the collision $\xi_0=\xi_2$ when $\kappa=8$ and $n=2N$, proving the first dimension formula in Theorem~\ref{thm::UST_decomposition}.

We begin with arbitrary $n\ge2$ and $\kappa>0$, keeping $\lambda=\nu=0$. The $\mathfrak{sl}_2$ structure motivates the following definitions, while the proofs in this section use only elementary calculations.
Let $\{\bs{e}_S:S\subseteq\{1,\ldots,n\}\}$ be the standard basis of $\R^{2^n}$, indexed by subsets. For $0\le m\le n$, set
\[
    V_m:=\operatorname{Span}\{\bs{e}_S:|S|=m\}.
\]
Define the operators $L$ and $R$ by
\begin{alignat*}{3}
    L&:=\frac{1}{\kappa}\tilde{T}^{(0)}=\sum_{j=1}^nL_j^{(0)}:V_m\to V_{m-1},
    &\qquad (L\bs{v})_S&=\sum_{j\notin S}v_{S\cup\{j\}},
    &\qquad L\bs{e}_S&=\sum_{j\in S}\bs{e}_{S\setminus\{j\}},\\
    R&:=L^{\mathsf T}:V_{m-1}\to V_m,
    &\qquad (R\bs{v})_S&=\sum_{j\in S}v_{S\setminus\{j\}},
    &\qquad R\bs{e}_S&=\sum_{j\notin S}\bs{e}_{S\cup\{j\}}.
\end{alignat*}
Set $H:=nI-2\Lambda$. A direct calculation gives
\[
    [L,R]=H,\qquad [H,L]=2L,\qquad [H,R]=-2R.
\]
These relations define an $\mathfrak{sl}_2$ representation, in which $V_m$ is the weight space of weight $n-2m$. We denote the corresponding space of highest-weight vectors by
\[
    K_m:=V_m\cap\ker L = \{\bs{v}: L\bs{v}=0, H\bs{v}=(n-2m)\bs{v}\}.
\]
The next lemma decomposes $\ker\widetilde T^{(0)}$ into these highest-weight spaces and gives their dimensions.

\begin{lemma}\label{lem::chordal_graded_kernel}
Set $N:=\lfloor n/2\rfloor$. Then
\begin{equation}\label{eqn::chordal_UST_graded_kernel}
    \ker\widetilde{T}^{(0)}=\bigoplus_{m=0}^{N} K_m,\qquad
    \dim K_m=\binom{n}{m}-\binom{n}{m-1},\qquad 0\le m\le N,
\end{equation}
where $\binom{n}{-1}:=0$.
\end{lemma}

\begin{proof}
We write $\widetilde T=\widetilde T^{(0)}$ throughout this proof.
Since $R=L^{\mathsf T}$ and $[L,R]=H=nI-2\Lambda$, for $\bs{v}\in V_m$ the standard Euclidean norm satisfies
\begin{equation}\label{eqn::chordal_UST_norm_identity}
    \|R\bs{v}\|^2-\|L\bs{v}\|^2
    =\langle\bs{v},[L,R]\bs{v}\rangle
    =(n-2m)\|\bs{v}\|^2.
\end{equation}
For $m>N$, we have $n-2m<0$, and the identity shows that $K_m=\{0\}$. Thus we obtain
\[
    \ker\widetilde{T}=\ker L
    =\bigoplus_{m=0}^nK_m
    =\bigoplus_{m=0}^{N}K_m.
\]
For $1\le m\le N$, applying~\eqref{eqn::chordal_UST_norm_identity} to $\bs{v}\in V_{m-1}$ gives $\|R\bs{v}\|^2\geq \|\bs{v}\|^2$. Thus $R:V_{m-1}\to V_m$ is injective, and its transpose $L:V_m\to V_{m-1}$ is surjective. It follows that
\[
    \dim K_m=\dim V_m-\dim V_{m-1}
    =\binom{n}{m}-\binom{n}{m-1}.
\]
Since $K_0=V_0$ has dimension $1$, the same formula holds for $m=0$ with the convention $\binom{n}{-1}=0$.
This proves~\eqref{eqn::chordal_UST_graded_kernel}.
\end{proof}

To express $\widetilde E(\bs{x})$ in terms of this representation, we introduce its Casimir operator:
\begin{equation}\label{eqn::chordal_casimir_operator}
    \Omega:=H^2+2H+4RL.
\end{equation}
It commutes with $L$, $R$ and $H$.
For $\bs{w}\in K_m$, we have $L\bs{w}=0$ and $H\bs{w}=(n-2m)\bs{w}$, thus
\[
    \Omega\bs{w}=(n-2m)(n-2m+2)\bs{w}.
\]
We also introduce a family of weighted lowering operators. For each $\bs{x}\in\LX_n$, set
\begin{equation}\label{eqn::chordal_weighted_lowering}
    X_r:=\sum_{j=1}^n x_j^rL_j^{(0)},\qquad \text{for }r=0,1,\ldots,
    \qquad \text{where }X_0=L.
\end{equation}
The matrices $L_j^{(0)}$ commute, giving $X_rX_s=X_sX_r$ for all $r,s\ge0$, and in particular $LX_r=X_rL$. Each $X_r$ lowers the degree by one, thus
\[
    X_rK_0=\{0\},\qquad X_rK_m\subseteq K_{m-1},
    \qquad r\ge0,\quad 1\le m\le N.
\]
The next lemma expresses the full matrix $\widetilde E(\bs{x})$ in terms of $\Omega$, $\Lambda$ and $X_1$.

\begin{lemma}\label{lem::chordal_dilation_casimir}
For every $\bs{x}\in\LX_n$, the matrix $\widetilde E(\bs{x})$ defined in~\eqref{eqn::chordal_dilation_similarity} satisfies
\begin{equation}\label{eqn::chordal_dilation_casimir}
    \widetilde E(\bs{x})=\Omega+\kappa\Lambda-3nI+\kappa X_1.
\end{equation}
\end{lemma}
\begin{proof}
We express the permutation term in~\eqref{eqn::chordal_dilation_similarity} using $L$, $R$ and $\Lambda$.
Fix $S$ with $|S|=m$. A transposition fixes $S$ when both labels lie in $S$ or both lie outside $S$; otherwise, it replaces one element of $S$ by one outside $S$. Also, in $RL\bs e_S$, first removing and then restoring the same label gives $m\bs e_S$, while all other terms make a replacement. Thus
\[
\begin{aligned}
    \sum_{i<j}P_{ij}\bs e_S
    &=\left[\binom{m}{2}+\binom{n-m}{2}\right]\bs e_S
      +\sum_{\substack{i\in S\\j\notin S}}
         \bs e_{(S\setminus\{i\})\cup\{j\}},\\
    RL\bs e_S
    &=m\bs e_S+\sum_{\substack{i\in S\\j\notin S}}
         \bs e_{(S\setminus\{i\})\cup\{j\}}.
\end{aligned}
\]
Subtracting these two identities gives
\[
    \left(\sum_{j<\ell}P_{j\ell}-RL\right)\bs{e}_S
    =\left[\binom{m}{2}+\binom{n-m}{2}-m\right]\bs{e}_S.
\]
Since $\Lambda\bs{e}_S=m\bs{e}_S$, this yields the matrix identity
\[
    \sum_{j<\ell}P_{j\ell}
    =RL+\Lambda^2-(n+1)\Lambda+\frac{n(n-1)}2I.
\]
Substituting this identity into~\eqref{eqn::chordal_dilation_similarity} and using $H=nI-2\Lambda$ and the definition of $\Omega$, we obtain
\[
\begin{aligned}
    \widetilde E(\bs{x})-\kappa X_1
    &=\kappa\Lambda+4RL+4\Lambda^2-4(n+1)\Lambda+n(n-1)I\\
    &=\kappa\Lambda+\Omega-3nI.
\end{aligned}
\]
This proves~\eqref{eqn::chordal_dilation_casimir}.
\end{proof}

Restricting this identity to $\ker\widetilde T^{(0)}=\bigoplus_{m=0}^N K_m$ gives the action of $\widehat E(\bs{x})$. Since $\Omega$ and $\Lambda$ act by scalars on each $K_m$, while $X_1$ maps $K_m$ into $K_{m-1}$, the following lemma gives a block upper triangular representation of $\widehat E(\bs{x})$.

\begin{lemma}\label{lem::chordal_UST_triangular_matrix}
For each $\bs{x}\in\LX_n$ and $N:=\lfloor n/2 \rfloor$, we have
\begin{equation}\label{eqn::chordal_UST_graded_dilation}
    \widehat E(\bs{x})\bs{v}
    =\xi_{n-2m}\bs{v}+\kappa X_1\bs{v},
    \qquad \text{for }\bs{v}\in K_m,\quad 0\le m\le N.
\end{equation}
Write $d_m:=\dim K_m$ as given in~\eqref{eqn::chordal_UST_graded_kernel}, and choose a basis $\mathcal B_m$ of each $K_m$. For $1\le m\le N$, let $Y_m(\bs{x})$ be the $d_{m-1}\times d_m$ matrix representing $X_1:K_m\to K_{m-1}$ in these bases.
In the concatenated basis $\mathcal B=(\mathcal B_0,\ldots,\mathcal B_N)$, the matrix of $\widehat E(\bs{x})$ is block upper triangular:
\begin{equation}\label{eqn::chordal_UST_triangular_matrix}
    [\widehat E(\bs{x})]_{\mathcal B}
    =
    \begin{pmatrix}
        \xi_n I_{d_0} & \kappa Y_1(\bs{x}) & 0 & \cdots & 0\\
        0 & \xi_{n-2}I_{d_1} & \kappa Y_2(\bs{x}) & \ddots & \vdots\\
        \vdots & \ddots & \ddots & \ddots & 0\\
        0 & \cdots & 0 & \xi_{n-2N+2}I_{d_{N-1}} & \kappa Y_N(\bs{x})\\
        0 & \cdots & \cdots & 0 & \xi_{n-2N}I_{d_N}
    \end{pmatrix}.
\end{equation}
\end{lemma}

\begin{proof}
For $0\le m\le N$ and $\bs{v}\in K_m$, we have $\Lambda\bs{v}=m\bs{v}$ and $\Omega\bs{v}=(n-2m)(n-2m+2)\bs{v}$. Since $\widehat E(\bs{x})$ is the restriction of $\widetilde E(\bs{x})$ to $\ker\widetilde T^{(0)}$, Lemma~\ref{lem::chordal_dilation_casimir} and~\eqref{eqn::xi_p_def} give
\[
    \widehat E(\bs{x})\bs{v}
    =\bigl[(n-2m)(n-2m+2)+\kappa m-3n\bigr]\bs{v}+\kappa X_1\bs{v}
    =\xi_{n-2m}\bs{v}+\kappa X_1\bs{v}.
\]
This proves~\eqref{eqn::chordal_UST_graded_dilation}. Together with $X_1K_0=\{0\}$ and $X_1K_m\subseteq K_{m-1}$, it gives the matrix form~\eqref{eqn::chordal_UST_triangular_matrix}.
\end{proof}

The block representation immediately gives the polynomial identities in Proposition~\ref{prop::widehatE}.
\begin{proof}[Proof of Proposition~\ref{prop::widehatE}]
The characteristic polynomial~\eqref{eqn::chordal_dilation_characteristic} immediately follows from the block upper triangular form~\eqref{eqn::chordal_UST_triangular_matrix}.
To prove the annihilating identity, observe that~\eqref{eqn::chordal_UST_graded_dilation} and $X_1:K_m\to K_{m-1}$ gives
\[
    \bigl(\widehat E(\bs{x})-\xi_{n-2m}I\bigr)
    \left(\bigoplus_{j=0}^mK_j\right)
    \subseteq\bigoplus_{j=0}^{m-1}K_j,\qquad 1\le m\le N,
\]
and $\bigl(\widehat E(\bs{x})-\xi_nI\bigr)K_0=\{0\}$. Applying these factors successively for $m=N,N-1,\ldots,0$ therefore annihilates $\bigoplus_{j=0}^NK_j=\ker\widetilde T^{(0)}$. Since all factors are polynomials in $\widehat E(\bs{x})$, they commute, and their product is $Q_\kappa(\widehat E(\bs{x}))$. This proves~\eqref{eqn::chordal_dilation_annihilating}.
\end{proof}

\medbreak

We now turn to the first dimension formula in Theorem~\ref{thm::UST_decomposition}. From now on, assume $n=2N\ge2$, still with $\lambda=\nu=0$. The key step is to prove that $X_1:K_N\to K_{N-1}$ is injective.
To prove this injectivity, we use the family $X_r$ defined in~\eqref{eqn::chordal_weighted_lowering}, together with the weighted degree matrix $\Lambda_{\bs{x}}$, the weighted analogue of $\Lambda$:
\begin{equation}\label{eqn::chordal_degree_matrices}
    (\Lambda_{\bs{x}}\bs{v})_S=\left(\sum_{j\in S}x_j\right)v_S,
    \qquad
    \Lambda_{\bs{x}}\bs{e}_S=\left(\sum_{j\in S}x_j\right)\bs{e}_S.
\end{equation}

\begin{lemma}\label{lem::chordal_middle_injectivity}
For every $\bs{x}\in\LX_n$, the map
\[
    X_1:K_N\longrightarrow K_{N-1}
\]
is injective.
\end{lemma}
\begin{proof}
We write $L_j=L_j^{(0)}$ throughout this proof. Let $\bs{w}\in K_N$.
Since $L\bs{w}=0$ and $n=2N$, the norm identity~\eqref{eqn::chordal_UST_norm_identity} gives $R\bs{w}=0$.
\medbreak
Now assume that $X_1\bs{w}=0$, and we show that $X_r \bs{w}=0$ for every $r\ge 0$ by induction. The key is that $[\Lambda_{\bs{x}},L_j]=-x_jL_j$, and hence $[\Lambda_{\bs{x}},X_r]=-X_{r+1}$ for $r\ge0$. If $X_r\bs{w}=0$, then
\[
X_{r+1}\bs{w} = X_r \Lambda_{\bs{x}}\bs{w}.
\]
A direct calculation gives $[R,X_1]=2\Lambda_{\bs{x}}-\sum_{j=1}^n x_j I$.
Since $R\bs{w}=X_1\bs{w}=0$, this identity yields
\[
    \Lambda_{\bs{x}}\bs{w}=\frac12 \sum_{j=1}^n x_j\bs{w}, \qquad X_{r+1}\bs{w} = \frac12\sum_{j=1}^n x_jX_r\bs{w} = 0.
\]
Starting from $X_0\bs{w}=L\bs{w}=0$, induction therefore gives $X_r\bs{w}=0$ for every $r\ge0$.
This shows that $\sum_{j=1}^nx_j^rL_j\bs{w}=0$ for $r=0,\ldots,n-1$.
Since $x_1,\ldots,x_n$ are distinct, the Vandermonde matrix is invertible, thus $L_j\bs{w}=0$ for every $j$.
For any $S$ with $|S|=N$ and any $j\in S$, we obtain
\[
    w_S=(L_j\bs{w})_{S\setminus\{j\}}=0.
\]
Since $\bs{w}\in V_N$, this proves $\bs{w}=0$ and hence injectivity.
\end{proof}

When $\kappa=8$, the eigenvalues of $\xi_0$ and $\xi_2$ coincide. The following lemma identifies the corresponding eigenspace by projection onto $K_{N-1}$ and yields the first dimension formula in Theorem~\ref{thm::UST_decomposition}.
\begin{lemma}\label{lem::chordal_UST_eigenspace}
Suppose $\kappa=8$ and fix $\bs{x}\in\LX_n$.
For $\bs{v}\in\ker\widetilde T^{(0)}$, write $\bs{v}=\sum_{m=0}^N\bs{v}_m$ with $\bs{v}_m\in K_m$.
The projection is an isomorphism
\begin{equation}\label{eqn::chordal_UST_eigenspace_decomposition}
\begin{aligned}
    \ker\bigr(\widehat E(\bs{x})-nI\bigr)
    \longrightarrow K_{N-1},\qquad
    \bs{v}\longmapsto \bs{v}_{N-1}.
\end{aligned}
\end{equation}
In particular,
\[
    \dim\LSchord{2N}^{(0;0;\xi_0)}
    =\dim\ker\bigr(\widehat E(\bs{x})-nI\bigr)
    =\binom{2N}{N-1}-\binom{2N}{N-2}
    =\frac{3N}{N+2}\Cat_N.
\]
\end{lemma}

\begin{proof}
Write $\widehat E=\widehat E(\bs{x})$ throughout this proof. At $\kappa=8$, we have $\xi_p=n+p(p-2)$, thus $\xi_0=\xi_2=n$.
By~\eqref{eqn::chordal_UST_graded_dilation}, the equation $\bigr(\widehat E-nI\bigr)\bs{v}=0$ is equivalent to
\begin{equation}\label{eqn::chordal_UST_eigen_recursion}
    (\xi_{n-2m}-n)\bs{v}_m+8X_1\bs{v}_{m+1}=0,
    \qquad 0\le m\le N,
\end{equation}
where $\bs{v}_{N+1}:=0$.
The equation at degree $N$ is automatic, while the equation at degree $N-1$ gives $X_1\bs{v}_N=0$. Lemma~\ref{lem::chordal_middle_injectivity} therefore implies $\bs{v}_N=0$.
For $0\le m\le N-2$, we have
\[
    \bs{v}_m
    =-\frac{8}{(n-2m)(n-2m-2)}X_1\bs{v}_{m+1}.
\]
The inclusion $X_1K_{m+1}\subseteq K_m$ ensures $\bs{v}_m\in K_m$ inductively, thus $\bs{v}_{N-1}$ uniquely determines all lower components, proving the isomorphism~\eqref{eqn::chordal_UST_eigenspace_decomposition}.
The stated dimension formula now follows from~\eqref{eqn::chordal_UST_graded_kernel} and Lemma~\ref{lem::Ematrix_tilde}.
\end{proof}

\subsection{Subspaces selected by the third Ward identity}\label{subsec::Ward_3rd}
Throughout this subsection, assume $\kappa=8$ and $n=2N\ge2$, with $\lambda=\nu=0$. We introduce the differential operator in the third Ward identity~\eqref{eqn::Ward_3rd} by setting
\begin{equation}\label{eqn::chordal_third_Ward_operator}
    \LW:=\sum_{j=1}^n x_j^2\partial_j-\frac14\sum_{j=1}^n x_j.
\end{equation}
Direct computation gives
\begin{equation}\label{eqn::chordal_third_Ward_operator_relations}
    [\LD_j,\LW]=4x_j\LD_j+6\LT,
    \qquad [\LT,\LW]=2\LE-\frac n4,
    \qquad [\LE,\LW]=\LW.
\end{equation}
We first obtain the matrix counterpart $W$ of $\LW$ in the first-order system~\eqref{eqn::chordal_first_order_system} in Proposition~\ref{prop::Wmatrix}.
We then derive its renormalized matrix $\widetilde{W}$ in the first-order system~\eqref{eqn::chordal_first_order_system2} in Lemma~\ref{lem::Wmatrix_tilde}. For $N\ge2$, Lemma~\ref{lem::chordal_UST_third_Ward_matrix} expresses the third Ward operator restricted in highest-weight components as an explicit map $Q_{\bs{x}}:K_{N-1}\to K_{N-2}$. Lemma~\ref{lem::chordal_UST_third_Ward} determines the kernel of $Q_{\bs{x}}$ and shows that it has dimension $\Cat_N$. Together with a direct check for $N=1$, this completes the proof of Theorem~\ref{thm::UST_decomposition}.

\begin{proposition}\label{prop::Wmatrix}
Fix $\kappa=8$ and $n=2N\ge 2$ and $\lambda=\nu=0$. 
Define 
\begin{equation}\label{eqn::chordal_third_Ward_matrix}
    W(\bs{x}):=\sum_{j=1}^n x_j^2D_j^{(0)}(\bs{x})
        -\frac14\sum_{j=1}^n x_j I+2\Lambda_{\bs{x}}-\frac14R.
\end{equation}
Then for each $\bs{x}\in\LX_n$, the restriction of $\zeta_{\bs{x}}$ in~\eqref{eqn::chordal_zeta} fives an isomorphism 
\[
  \zeta_{\bs{x}}: \LSchord{n}^{\mathrm{UST}} \longrightarrow \ker T^{(0)}(\bs{x})\cap\ker(E(\bs{x})-nI)\cap\ker W(\bs{x}).
\]
\end{proposition}

\begin{proof}
The proof is similar to that of Proposition~\ref{prop::Ematrix}.
Using~\eqref{eqn::chordal_third_Ward_operator_relations}, for every $\LZ\in\LSchord{n}^{(0;0;n)}$, the function $\LW\LZ$ belongs to $\LSchord{n}^{(0;0;n+8)}$. Thus
\begin{equation}\label{eqn::chordal_third_Ward_solution_map}
    \LW:\LSchord{n}^{(0;0;n)}
    \longrightarrow\LSchord{n}^{(0;0;n+8)}.
\end{equation}
Since, for $\bs{\beta}\in\{0,1\}^n$,
\begin{equation}\label{eqn::chordal_third_Ward_solution_map_derivative}
    \partial^{\bs{\beta}}(\LW\LZ)
    =\LW\partial^{\bs{\beta}}\LZ
      +2\left(\sum_{j=1}^n\beta_jx_j\right)\partial^{\bs{\beta}}\LZ
      -\frac14\sum_{j:\beta_j=1}\partial^{\bs{\beta}-\bs{e}_j}\LZ,
\end{equation}
the first-order system~\eqref{eqn::chordal_first_order_system} gives
\begin{equation}\label{eqn::chordal_third_Ward_original_intertwining}
    \zeta_{\bs{x}}(\LW\LZ)=W(\bs{x})\zeta_{\bs{x}}(\LZ).
\end{equation}

Fix $\bs{x}\in\LX_n$, and write $T=T^{(0)}(\bs{x})$, $E=E(\bs{x})$ and $W=W(\bs{x})$. If $\LZ\in\LSchord{n}^{\mathrm{UST}}$, then Proposition~\ref{prop::Ematrix} and~\eqref{eqn::chordal_third_Ward_original_intertwining} give
\[
    T\zeta_{\bs{x}}(\LZ)=0,
    \qquad (E-nI)\zeta_{\bs{x}}(\LZ)=0,
    \qquad W\zeta_{\bs{x}}(\LZ)=\zeta_{\bs{x}}(\LW\LZ)=0.
\]
Thus $\zeta_{\bs{x}}$ maps $\LSchord{n}^{\mathrm{UST}}$ injectively into the stated intersection.

Conversely, let $\bs{u}\in\ker T\cap\ker(E-nI)\cap\ker W$. By Proposition~\ref{prop::Ematrix}, there is a unique $\LZ\in\LSchord{n}^{(0;0;n)}$ with $\zeta_{\bs{x}}(\LZ)=\bs{u}$. Eq.~\eqref{eqn::chordal_third_Ward_original_intertwining} gives
\[
    \zeta_{\bs{x}}(\LW\LZ)=W\bs{u}=0.
\]
Since $\LW\LZ\in\LSchord{n}^{(0;0;n+8)}$, injectivity of $\zeta_{\bs{x}}$ implies $\LW\LZ=0$. Hence $\LZ\in\LSchord{n}^{\mathrm{UST}}$, proving surjectivity.
\end{proof}

\begin{lemma}\label{lem::Wmatrix_tilde}
For every $\bs{x}\in\LX_n$, let $M(\bs{x})$ be the matrix from Lemma~\ref{lem::chordal_triangular_change}. Then
\begin{equation}\label{eqn::chordal_third_Ward_similarity}
\begin{aligned}
    \widetilde W(\bs{x})
    &:=M(\bs{x})^{-1}W(\bs{x})M(\bs{x})\\
    &=X_2+\frac12\sum_{i<j}(x_i+x_j)(P_{ij}-I)
      +2\Lambda_{\bs{x}}+\frac{N-1}{2}\sum_{j=1}^n x_jI
      +\frac12R\bigl((N-1)I-\Lambda\bigr),
\end{aligned}
\end{equation}
where $W(\bs{x})$ is defined in~\eqref{eqn::chordal_third_Ward_matrix}. Moreover, $\widetilde W(\bs{x})$ maps $\ker(\widehat E(\bs{x})-nI)$ into $\ker(\widehat E(\bs{x})-(n+8)I)$. Denote this restriction by
\begin{equation}\label{eqn::chordal_third_Ward_restriction}
    \widehat W(\bs{x}):
    \ker(\widehat E(\bs{x})-nI)
    \longrightarrow\ker(\widehat E(\bs{x})-(n+8)I).
\end{equation}
The map $\LZ\mapsto M(\bs{x})^{-1}\zeta_{\bs{x}}(\LZ)$ restricts to an isomorphism from $\LSchord{n}^{\mathrm{UST}}$ onto $\ker\widehat W(\bs{x})$. In particular,
\begin{equation}\label{eqn::chordal_third_Ward_kernel_dimension}
    \dim\LSchord{n}^{\mathrm{UST}}
    =\dim\ker\widehat W(\bs{x}).
\end{equation}
\end{lemma}

\begin{proof}
We write $T=T^{(0)}$, $\widetilde T=\widetilde T^{(0)}$ and $\widetilde D_j=\widetilde D_j^{(0)}$ throughout this proof.
Using~\eqref{eqn::chordal_connection_change} and~\eqref{eqn::chordal_third_Ward_matrix}, we obtain
\[
\begin{aligned}
    M^{-1}WM
    &=\sum_{j=1}^n x_j^2\widetilde D_j+M^{-1}\left(\LW M\right)+
      M^{-1}\left(2\Lambda_{\bs{x}}-\frac14R\right)M.
\end{aligned}
\]
We calculate these terms separately.
\begin{itemize}
    \item We first calculate the terms involving $M$.
    For a smooth function $\LF$, put $\bs u=(\partial_S(\LZ_{\shuffle_n}\LF))_S$ and $\bs v=(\mathcal J_S\LF)_S$, so that $\bs u=M\bs v$ by Lemma~\ref{lem::chordal_triangular_change}.
    In the original coordinates, Eq.~\eqref{eqn::chordal_third_Ward_solution_map_derivative} gives
    \[
    \begin{aligned}
        \bigl(\partial_S(\LW(\LZ_{\shuffle_n}\LF))\bigr)_S
        &=\left(\mathcal W+2\Lambda_{\bs{x}}-\frac14R\right)\bs u=\left(\mathcal W+2\Lambda_{\bs{x}}-\frac14R\right)(M\bs v) \\
        &=\left(\LW M\right)\bs{v} + M\LW \bs{v} + \left(\frac14\sum_{j=1}^n x_j+2\Lambda_{\bs{x}}-\frac14R\right)M\bs v.
    \end{aligned}
    \]
    In the new coordinates, since $\partial_j\log\LZ_{\shuffle_n}=\frac14\sum_{\ell\ne j}(x_j-x_\ell)^{-1}$, pairing the terms indexed by $j,\ell$ gives
    \[
        \LZ_{\shuffle_n}^{-1}\LW(\LZ_{\shuffle_n}\LF)
        =\left(\mathcal W+\frac{2N-1}{4}\sum_{j=1}^n x_j\right)\LF.
    \]
    The identities
    \[
    \begin{aligned}
        \left[\partial_j,\LW+\frac{2N-1}{4}\sum_{i=1}^n x_i\right]
        &=2x_j\partial_j+\frac{N-1}{2}I,\\
        \left[(x_j-x_\ell)^{-1},\LW+\frac{2N-1}{4}\sum_{i=1}^n x_i\right]
        &=\frac{x_j+x_\ell}{x_j-x_\ell}.
    \end{aligned}
    \]
    give, by induction in~\eqref{eqn::chordal_corrected_jet_recursion},
    \begin{equation}\label{eqn::chordal_third_Ward_jet_commutator}
        \left[\mathcal J_S,\LW+\frac{2N-1}{4}\sum_{j=1}^n x_j\right]
        =2\left(\sum_{j\in S}x_j\right)\mathcal J_S
          +\frac{N-|S|}{2}
           \sum_{i\in S}\mathcal J_{S\setminus\{i\}}.
    \end{equation}
    Thus we obtain
    \[
    \begin{aligned}
        \bigl(\partial_S(\LW(\LZ_{\shuffle_n}\LF))\bigr)_S
        &=M\left(\mathcal J_S\left(\left(\LW+\frac{2N-1}{4}\sum_{j=1}^n x_j\right)\LF\right)\right)_S\\
        &=M\left(\LW+2\Lambda_{\bs{x}}+\frac{2N-1}{4}\sum_{j=1}^n x_jI
          +\frac12R\bigl((N-1)I-\Lambda\bigr)\right)\bs v.
    \end{aligned}
    \]
    Compare the two expressions of  $\bigl(\partial_S(\LW(\LZ_{\shuffle_n}\LF))\bigr)_S$.
    Since the triangular relation~\eqref{eqn::chordal_triangular_change_of_v} allows $\bs v(\bs{x})$ to be prescribed arbitrarily, we have
    \[
        \LW M+\left(2\Lambda_{\bs{x}}-\frac14R\right)M
        =M\left(2\Lambda_{\bs{x}}+\frac{N-1}{2}\sum_{j=1}^n x_jI
          +\frac12R\bigl((N-1)I-\Lambda\bigr)\right).
    \]
    \item The explicit formula~\eqref{eqn::chordal_explicit_connection} and $P_{j\ell}=P_{\ell j}$ give
    \[
    \begin{aligned}
        \sum_{j=1}^n x_j^2\widetilde D_j
        &=X_2+\frac12\sum_{j<\ell}(x_j+x_\ell)(P_{j\ell}-I).
    \end{aligned}
    \]
\end{itemize}
Combining the two identities above proves~\eqref{eqn::chordal_third_Ward_similarity}.
\medbreak
For any $\bs w\in\ker(\widehat E(\bs{x})-nI)$, Lemma~\ref{lem::Ematrix_tilde} gives a unique $\LZ\in\LSchord{n}^{(0;0;n)}$ for which
\[
    \bs w=M(\bs{x})^{-1}\zeta_{\bs{x}}(\LZ).
\]
Since $\LW\LZ\in\LSchord{n}^{(0;0;n+8)}$ by~\eqref{eqn::chordal_third_Ward_solution_map}, the intertwining identity~\eqref{eqn::chordal_third_Ward_original_intertwining} and Lemma~\ref{lem::Ematrix_tilde} give
\[
    \widetilde W(\bs{x})M(\bs{x})^{-1}\zeta_{\bs{x}}(\LZ)
    =M(\bs{x})^{-1}\zeta_{\bs{x}}(\LW\LZ)
    \in\ker(\widehat E(\bs{x})-(n+8)I).
\]
Thus $\widetilde W(\bs{x})$ maps $\ker(\widehat E(\bs{x})-nI)$ into $\ker(\widehat E(\bs{x})-(n+8)I)$, and $\widehat W(\bs{x})$ is well-defined. Finally, suppressing the dependence on $\bs{x}$, the similarity identities give
\[
\begin{aligned}
    M^{-1}\bigl(\ker T\cap\ker(E-nI)\cap\ker W\bigr)
    &=\ker\widetilde T\cap\ker(\widetilde E-nI)\cap\ker\widetilde W=\ker\widehat W.
\end{aligned}
\]
Combining this identity with Proposition~\ref{prop::Wmatrix} proves the isomorphism and~\eqref{eqn::chordal_third_Ward_kernel_dimension}.
\end{proof}

Lemma~\ref{lem::chordal_UST_eigenspace} identifies the source space of $\widehat W(\bs{x})$ with $K_{N-1}$ by a projection. For $N\ge2$, the next lemma identifies the target space with $K_{N-2}$ and expresses $\widehat W(\bs{x})$ as a map $Q_{\bs{x}}:K_{N-1}\to K_{N-2}$ under these identifications. For $\bs v\in\ker\widetilde T^{(0)}$, we write $\bs v=\sum_{m=0}^N\bs v_m$, with $\bs v_m\in K_m$, according to the decomposition~\eqref{eqn::chordal_UST_graded_kernel}.
\begin{lemma}\label{lem::chordal_UST_third_Ward_matrix}
Suppose $N\ge2$. For each $\bs{x}\in\LX_n$, define
\begin{equation}\label{eqn::chordal_UST_third_Ward_matrix}
    Q_{\bs{x}}
    :=X_2+\left(\Lambda_{\bs{x}}-\frac12\sum_{j=1}^n x_jI\right)X_1
      -\frac16RX_1^2:K_{N-1}\longrightarrow K_{N-2}.
\end{equation}
Here $X_1^2$ means $X_1\circ X_1$. For $\bs v\in\ker(\widehat E(\bs{x})-nI)$, we have
\begin{equation}\label{eqn::chordal_UST_third_Ward_intertwining}
    \bigl(\widehat W(\bs{x})\bs v\bigr)_{N-2}
    =Q_{\bs{x}}\bs v_{N-1}.
\end{equation}
The projection isomorphism $\bs v\mapsto\bs v_{N-1}$ in Lemma~\ref{lem::chordal_UST_eigenspace} restricts to an isomorphism from $\ker\widehat W(\bs{x})$ onto $\ker Q_{\bs{x}}$. Consequently, the map
\[
    \LZ\longmapsto\bigl(M(\bs{x})^{-1}\zeta_{\bs{x}}(\LZ)\bigr)_{N-1}
\]
is an isomorphism from $\LSchord{n}^{\mathrm{UST}}$ onto $\ker Q_{\bs{x}}$. In particular,
\begin{equation}\label{eqn::chordal_UST_third_Ward_dimension}
    \dim\LSchord{n}^{\mathrm{UST}}=\dim\ker Q_{\bs{x}}.
\end{equation}
\end{lemma}
\begin{proof}
We first compute the action of $\widetilde W(\bs{x})$ on each $K_m$. A weighted version of the permutation calculation in Lemma~\ref{lem::chordal_dilation_casimir} gives, for $\bs u\in K_m$,
\begin{equation}\label{eqn::chordal_weighted_permutation_sum}
    \sum_{i<j}(x_i+x_j)(P_{ij}-I)\bs u
    =RX_1\bs u-(n-2m+2)\Lambda_{\bs{x}}\bs u
      -m\left(\sum_{j=1}^n x_j\right)\bs u.
\end{equation}
Indeed, fix $S$ with $|S|=m$. As in the proof of Lemma~\ref{lem::chordal_dilation_casimir}, only transpositions exchanging a label in $S$ with one outside $S$ contribute. Hence
\[
\begin{aligned}
    \left(\sum_{i<j}(x_i+x_j)(P_{ij}-I)\bs u\right)_S
    &=\sum_{\substack{i\in S\\j\notin S}}
        (x_i+x_j)u_{(S\setminus\{i\})\cup\{j\}}-\left[(n-2m)\sum_{i\in S}x_i
         +m\sum_{j=1}^n x_j\right]u_S,
\end{aligned}
\]
where we used $\sum_{i\in S,j\notin S}(x_i+x_j)=(n-m)\sum_{i\in S}x_i+m\sum_{j\notin S}x_j$.
The identities $(L\bs u)_{S\setminus\{i\}}=0$ and $(RX_1\bs u)_S=\sum_{i\in S}(X_1\bs u)_{S\setminus\{i\}}$ give
\[
\begin{aligned}
    \sum_{\substack{i\in S\\j\notin S}}
       x_i u_{(S\setminus\{i\})\cup\{j\}}
       &=-\left(\sum_{i\in S}x_i\right)u_S,\\
    \sum_{\substack{i\in S\\j\notin S}}
       x_j u_{(S\setminus\{i\})\cup\{j\}}
       &=(RX_1\bs u)_S-\left(\sum_{i\in S}x_i\right)u_S.
\end{aligned}
\]
Substituting these into the preceding expansion and using $(\Lambda_{\bs{x}}\bs u)_S=(\sum_{i\in S}x_i)u_S$ proves~\eqref{eqn::chordal_weighted_permutation_sum}.
\medbreak
Substituting~\eqref{eqn::chordal_weighted_permutation_sum} into~\eqref{eqn::chordal_third_Ward_similarity}, we obtain
\begin{equation}\label{eqn::chordal_UST_third_Ward_graded_action}
\begin{aligned}
    \widetilde W(\bs{x})\bs u
    &=X_2\bs u+
      \left[\frac12RX_1+(m-N+1)
        \left(\Lambda_{\bs{x}}-\frac12\sum_{j=1}^n x_jI\right)\right]\bs u+\frac{N-1-m}{2}R\bs u,
    \quad \bs u\in K_m.
\end{aligned}
\end{equation}
The three terms have degrees $m-1$, $m$ and $m+1$, respectively.
\medbreak We next show that the projection
\begin{equation}\label{eqn::chordal_UST_third_Ward_target_projection}
    \ker\bigl(\widehat E(\bs{x})-(n+8)I\bigr)
    \longrightarrow K_{N-2},\qquad \bs v\longmapsto\bs v_{N-2}.
\end{equation}
is an isomorphism. The proof is similar to that of Lemma~\ref{lem::chordal_UST_eigenspace}. For $\bs v=\sum_{m=0}^N\bs v_m\in\ker\widetilde T^{(0)}$, Eq.~\eqref{eqn::chordal_UST_graded_dilation} shows that $\widehat E(\bs{x})\bs v=(n+8)\bs v$ is equivalent to
\[
    \bigl[(n-2m)(n-2m-2)-8\bigr]\bs v_m
       +8X_1\bs v_{m+1}=0,\qquad 0\le m\le N,
\]
where $\bs v_{N+1}:=0$.
Solving from degree $N$ downwards gives $\bs v_N=\bs v_{N-1}=0$, and the component $\bs v_{N-2}$ is arbitrary and uniquely determines all lower components. This gives the desired isomorphism.
\medbreak
Finally, we prove~\eqref{eqn::chordal_UST_third_Ward_intertwining}. Due to the isomorphism~\eqref{eqn::chordal_UST_third_Ward_target_projection}, the action of $\widehat W(\bs{x})$ is determined by its degree-$(N-2)$ component. Take $\bs v\in\ker(\widehat E(\bs{x})-nI)$ and put $\bs w:=\bs v_{N-1}$. Lemma~\ref{lem::chordal_UST_eigenspace} and~\eqref{eqn::chordal_UST_eigen_recursion} give
\[
    \bs v_N=0,\qquad \bs v_{N-2}=-X_1\bs w,
    \qquad \bs v_{N-3}=\frac13X_1^2\bs w.
\]
For $N=2$, the last identity is understood with $\bs v_{-1}:=0$, since $X_1^2\bs w=0$ for $\bs w\in K_1$. Reading the degree-$(N-2)$ component in~\eqref{eqn::chordal_UST_third_Ward_graded_action} therefore gives
\[
\begin{aligned}
    \bigl(\widehat W(\bs{x})\bs v\bigr)_{N-2}
    &=X_2\bs{v}_{N-1}+\left(\frac12RX_1
      -\Lambda_{\bs{x}}+\frac12\sum_{j=1}^n x_jI\right)\bs{v}_{N-2}
      +R\bs{v}_{N-3}\\
    &=X_2\bs w-\frac12RX_1^2\bs w
      +\left(\Lambda_{\bs{x}}-\frac12\sum_{j=1}^n x_jI\right)X_1\bs w
      +\frac13RX_1^2\bs w\\
    &=Q_{\bs{x}}\bs w.
\end{aligned}
\]
This proves~\eqref{eqn::chordal_UST_third_Ward_intertwining} and, combining the isomorphisms~\eqref{eqn::chordal_UST_eigenspace_decomposition} and~\eqref{eqn::chordal_UST_third_Ward_target_projection}, gives the asserted isomorphism from $\ker\widehat W(\bs{x})$ onto $\ker Q_{\bs{x}}$.
Composing with the isomorphism in Lemma~\ref{lem::Wmatrix_tilde} proves the remaining claims.
\end{proof}

It remains to determine $\ker Q_{\bs{x}}$ as a subspace of $K_{N-1}$. The next lemma identifies this kernel with $X_1K_N$; the injectivity of $X_1$ on $K_N$ from Lemma~\ref{lem::chordal_middle_injectivity} then gives its Catalan dimension.
\begin{lemma}\label{lem::chordal_UST_third_Ward}
For $N\ge2$ and every $\bs{x}\in\LX_n$, the map $Q_{\bs{x}}$ in~\eqref{eqn::chordal_UST_third_Ward_matrix} satisfies
\begin{equation}\label{eqn::chordal_UST_third_Ward_kernel}
    \ker Q_{\bs{x}}=X_1K_N,
    \qquad \dim\ker Q_{\bs{x}}=\dim K_N=\Cat_N.
\end{equation}
\end{lemma}

\begin{proof}
We write $L_j=L_j^{(0)}$ throughout this proof.
We first prove that $X_1K_N\subseteq \ker Q_{\bs{x}}$, and $\dim\ker Q_{\bs{x}}\geq \dim K_N = \Cat_N$. For $\bs u\in K_N$, the norm identity~\eqref{eqn::chordal_UST_norm_identity} and $n=2N$ give $R\bs u=0$. Using the commutator identities from the proof of Lemma~\ref{lem::chordal_middle_injectivity}, together with $[X_1,X_2]=0$, we calculate
\[
\begin{aligned}
    RX_1^3\bs u
    &=\bigl([R,X_1]X_1^2+X_1[R,X_1]X_1+X_1^2[R,X_1]\bigr)\bs u\\
    &=6\left(\Lambda_{\bs{x}}-\frac12\sum_{j=1}^n x_jI\right)X_1^2\bs u
      +6X_2X_1\bs u.
\end{aligned}
\]
Thus $Q_{\bs{x}}X_1\bs u=0$, and thus $X_1K_N\subseteq\ker Q_{\bs{x}}$. By the injectivity of $X_1:K_N\to K_{N-1}$ in Lemma~\ref{lem::chordal_middle_injectivity}, we obtain
\[
\dim\ker Q_{\bs{x}}\ge\dim X_1K_N = \dim K_N=\Cat_N.
\]
\medbreak
Conversely, we prove that $\dim\ker Q_{\bs{x}}\le\Cat_N$ for all $\bs{x}\in\LX_n$. Due to~\eqref{eqn::chordal_UST_third_Ward_dimension}, $\dim\ker Q_{\bs{x}}$ is independent of $\bs{x}\in\LX_n$. We calculate $\dim\ker Q_{\bs{x}}$ at a specific point.
\begin{itemize}
    \item The map $Q_{\bs{x}}: K_{N-1}\to K_{N-2}$ is well-defined even when coordinates coincide. Observe that $K_{N-1}$ and $K_{N-2}$ are independent of $\bs{x}$, and the matrix entries of $Q_{\bs{x}}$ in fixed bases are polynomials in $\bs{x}$, see~\eqref{eqn::chordal_UST_third_Ward_matrix}. Thus the formula for $Q_{\bs{x}}$ extends to coincident coordinates.
Using $[L,\Lambda_{\bs{x}}]=X_1$, $[L,X_1]=[L,X_2]=0$ and $[L,R]=nI-2\Lambda$, we obtain, for $\bs w\in K_{N-1}$,
\[
    LQ_{\bs{x}}\bs w
    =X_1^2\bs w-\frac16LRX_1^2\bs w
    =X_1^2\bs w-\frac16\cdot6X_1^2\bs w=0.
\]
Here $X_1^2\bs w\in K_{N-3}$ when $N\ge3$; for $N=2$, $X_1^2\bs w=0$. Thus $Q_{\bs{x}}$ still maps $K_{N-1}$ to $K_{N-2}$.
\item We evaluate $Q_{\bs{x}}$ at $\bs{x}_*:=(0,1,\ldots,1)$. In this case,
\[
    Q_{\bs{x}_*}=-\frac12X_1=\frac12L_1
    \quad\text{on }K_{N-1}.
\]
The condition $L_1\bs w=0$ says exactly that all coordinates of $\bs w$ indexed by sets containing $1$ vanish. So $L{\bs w}=0$ reduces to $\left(\sum_{j=2}^{2N}L_{j}\right)\bs w=0$. This forms the degree-$(N-1)$ kernel of the lowering operator $\left(\sum_{j=2}^{2N}L_{j}\right)$. Applying Lemma~\ref{lem::chordal_graded_kernel} with $n=2N-1, m=N-1$ gives
\[
    \dim\ker Q_{\bs{x}_*}
    =\binom{2N-1}{N-1}-\binom{2N-1}{N-2}
    =\Cat_N.
\]
\item The limit decreases the rank of a matrix, so $\dim\ker Q_{\bs{x}}\le\dim\ker Q_{\bs{x}_*}=\Cat_N$ for all $\bs{x}\in\LX_n$.
\end{itemize}
Thus we prove that $\dim\ker Q_{\bs{x}}\le\Cat_N$ for all $\bs{x}\in\LX_n$. Combining this with the lower bound $\dim\ker Q_{\bs{x}}\ge\Cat_N$ proves~\eqref{eqn::chordal_UST_third_Ward_kernel}.
\end{proof}

\begin{proof}[Proof of Theorem~\ref{thm::UST_decomposition}]
The first dimension formula follows from Lemma~\ref{lem::chordal_UST_eigenspace}.
For $N\ge2$, Lemmas~\ref{lem::chordal_UST_third_Ward_matrix} and~\ref{lem::chordal_UST_third_Ward} give
\[
    \dim\LSchord{n}^{\mathrm{UST}}=\dim\ker Q_{\bs{x}}=\Cat_N.
\]
For $N=1$, Remark~\ref{rem::xi_p_fail} gives $\LSchord{2}^{(0;0;2)}=\operatorname{span}_{\mathbb{R}}\{(x_2-x_1)^{1/4}\}$. This function is annihilated by $\LW$, so $\LSchord{2}^{\mathrm{UST}}=\LSchord{2}^{(0;0;2)}$ and $\dim\LSchord{2}^{\mathrm{UST}}=1=\Cat_1$.
\end{proof}

\section{Coulomb gas integral and Gaussian free field level lines}
\label{sec::GFF}
Fix $n\ge 2$ and $\kappa>0$ and $\lambda>0$. In this section, we first construct solutions to the BPZ system~\eqref{eqn::BPZ_chordal} using Coulomb gas integral. 
For $\mathfrak{r}\in \{0,1,\ldots, n\}$ and $\bs{x}=(x_1, \ldots, x_n)\in\LX_n$ and $\bs{u}=(u_1, \ldots, u_{\mathfrak{r}})$, define the (multivalued) integrand function by
\begin{align}\label{eqn::CGI_integrand}
\begin{split}
\integrandzero{n}(\bs{x}; \bs{u}):=&\prod_{1\le i<j\le n}(x_j-x_i)^{2/\kappa}\prod_{\substack{1\le \ell\le n\\1\le m\le \mathfrak{r}}}(u_m-x_{\ell})^{-4/\kappa}\prod_{1\le p<q\le \mathfrak{r}}(u_q-u_p)^{8/\kappa},\\
\integrand{n}(\bs{x}; \bs{u}):=&\exp\left(\frac{\sqrt{2\lambda}}{\kappa}\left(\sum_{j=1}^n x_j-2\sum_{m=1}^{\mathfrak{r}}u_m\right)\right)\times\integrandzero{n}(\bs{x}; \bs{u}). 
\end{split}
\end{align}
For $\sigma=(\sigma_1, \ldots, \sigma_n)\in\{\pm 1\}^n$, suppose $\sum_{j=1}^n\sigma_j=n-2\mathfrak{r}$ for some $\mathfrak{r}\in\{0,1,\ldots, n\}$ and let $s_1<\cdots<s_{\mathfrak{r}}$ be the indexes in $\{j\in\{1, \ldots, n\}: \sigma_j=-1\}$. 
Define the integral by 
\begin{align}\label{eqn::CGI}
\LGchord{n}(\bs{x})=\cstchord{n}\ointclockwise_{\vartheta_{s_1}}\ud u_1\cdots\ointclockwise_{\vartheta_{s_{\mathfrak{r}}}}\ud u_{\mathfrak{r}} \;\integrand{n}(\bs{x}; \bs{u}),
\end{align}
where 
\begin{align}\label{eqn::CGI_cst_def}
\cstchord{n}=(2\pi \ii)^{-\mathfrak{r}}
       \exp\left(\frac{4\pi \ii}{\kappa}
          \left(\sum_{j: \sigma_j=-1}(n-j)-\mathfrak{r}\right)\right),
\end{align}
and $\vartheta_s$ is an integration contour defined in the follow way: we fix an angle $\theta\in (0,\pi/2)$ and take $\vartheta_s$ to be a clockwise Hankel contour around the ray $x_s+\ee^{\ii\theta}[0,\infty)$, see Figure~\ref{fig::CGI_Hankel}.

\begin{figure}[htbp]
\centering
\begin{tikzpicture}[x=0.95cm,y=0.95cm,>=stealth,color=black]
    \draw[->] (-1.5,0)--(10.7,0) node[right] {};
    \foreach \x/\j in {0/1,3/2,6/3} {
        \fill (\x,0) circle (1.6pt);
        \node[below=5pt] at (\x,0) {$x_{\j}$};
        \begin{scope}[shift={(\x,0)},rotate=35]
            \draw[dashed,black] (0,0)--(4.9,0);
            \draw[thick] (4.6,-0.16)--(0,-0.16)
                arc[start angle=-90,end angle=-270,radius=0.16]
                --(4.6,0.16);
            \draw[thick,->] (3.6,-0.16)--(2.8,-0.16);
            \draw[thick,->] (2.8,0.16)--(3.6,0.16);
            \node[above] at (3.9,0.2) {$\vartheta_{\j}$};
        \end{scope}
    }
    \draw (0.8,0) arc[start angle=0,end angle=35,radius=0.8];
    \node at (1.1,0.26) {$\theta$};
\end{tikzpicture}
\caption{Disjoint clockwise Hankel contours around parallel rays. it is the boundary of a thin tubular neighborhood of this ray: it approaches $x_s$ from infinity along the lower side, turns clockwise around $x_s$, and returns to infinity along the upper side.
Choose the width sufficiently small that these neighborhoods are pairwise disjoint and contain no other marked points. Only the contours $\vartheta_{s_1},\ldots,\vartheta_{s_{\mathfrak r}}$, corresponding to $\sigma_j=-1$, are used in~\eqref{eqn::CGI}.}
\label{fig::CGI_Hankel}
\end{figure}
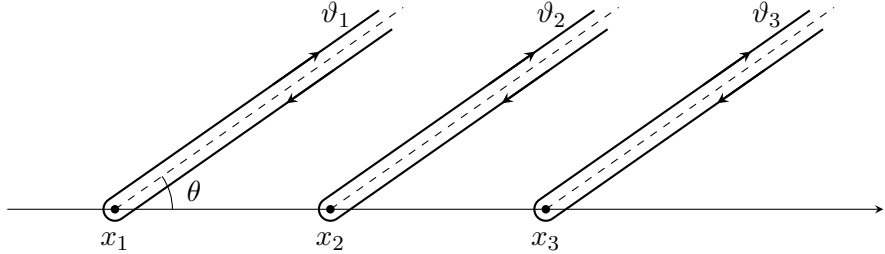

\begin{proposition}\label{prop::CGI}
Fix $n\ge 2$ and $\kappa>0$ and $\lambda>0$. 
The integral $\LGchord{n}$ in~\eqref{eqn::CGI} is well-defined and satisfies the BPZ system~\eqref{eqn::BPZ_chordal} and the first Ward identity~\eqref{eqn::Ward} with $\nu=\sqrt{2\lambda}\sum_{j=1}^n\sigma_j$. Moreover, the collection of $2^n$ functions
\begin{align}\label{eqn::CGI_basis}
\left\{\Re\LGchord{n}(\bs{x}):\sigma\in\{\pm 1\}^n\right\}
\end{align}
is a basis for the space $\LSchord{n}^{(\lambda)}$.
\end{proposition}

\begin{proposition}\label{prop::CGI_kappa4}
Fix $n\ge2$ and $\kappa=4$ and $\lambda>0$. For $\sigma\in\{\pm 1\}^n$, recall that $\LFchord{n}$ is defined in~\eqref{eqn::LFchord_sigma_def}.
Then
\begin{equation}\label{eqn::CGI_kappa4_residue}
    \LGchordGFF{n}(\bs x)
    =\LFchord{n}(\bs x),
    \qquad \text{for all }\bs x\in\LX_n.
\end{equation}
\end{proposition}

\subsection{Proof of Proposition~\ref{prop::CGI}}
\label{subsec::CGI_basis}

\begin{lemma}\label{lem::CGI_equations}
Assume the same setup as in Proposition~\ref{prop::CGI}. 
The integral $\LGchord{n}$ in~\eqref{eqn::CGI} is well-defined and satisfies the BPZ system~\eqref{eqn::BPZ_chordal} and the first Ward identity~\eqref{eqn::Ward} with $\nu=\sqrt{2\lambda}\sum_{j=1}^n\sigma_j$.
\end{lemma}

\begin{proof}
	We first prove the absolute convergence of the integral~\eqref{eqn::CGI}. On the chosen contours, the algebraic factors grow at most polynomially along the tails, whereas
$\exp(-2\tfrac{\sqrt{2\lambda}}{\kappa}\sum_m\Re u_m)$ decays exponentially.
Thus the integral converges absolutely, and thus is well-defined.
\medbreak
It is clear that the integral is real analytic on $\LX_n$.
Recall that the differential operators $\LD_j$ and $\LT$ are defined in~\eqref{eqn::chordal_BPZ_operators} and~\eqref{eqn::chordal_Ward_12}.
A direct calculation gives the total-derivative identity
\begin{equation}\label{eqn::CGI_BPZ_total_derivative}
    \left(\LD_j-\frac{\lambda}{\kappa}\right)
      \integrand{n}(\bs x;\bs u)
    =-2\sum_{m=1}^{\mathfrak r}\partial_{u_m}
      \left(\frac{\integrand{n}(\bs x;\bs u)}{u_m-x_j}\right).
\end{equation}
Integrating along the Hankel contours gives zero on the right, since the boundary terms vanish at both ends.
This proves the BPZ equations. Furthermore, simultaneous translation of all $x_j$ and $u_m$ leaves the algebraic factors unchanged, thus
\[
    \left(\sum_{j=1}^n\partial_{x_j}+
          \sum_{m=1}^{\mathfrak r}\partial_{u_m}\right)
      \integrand{n}
    =\frac{\sqrt{2\lambda}}{\kappa}(n-2\mathfrak r)\integrand{n}.
\]
Integrating again gives
$\kappa\LT\LGchord{n}=\sqrt{2\lambda}(n-2\mathfrak{r})\LGchord{n}$. This proves the first Ward identity. 
\end{proof}
\begin{lemma}\label{lem::CGI_independence}
Assume the same setup as in Proposition~\ref{prop::CGI}. 
The collection of $2^n$ functions
\begin{align}\label{eqn::CGI_complex_family}
\left\{\LGchord{n}(\bs{x}):
        \sigma\in\{\pm 1\}^n\right\}
\end{align}
is linearly independent over $\C$. 
\end{lemma}
\begin{proof}
Fix $\sigma\in\{\pm1\}^n$, with $\mathfrak r$ and $s_1<\cdots<s_{\mathfrak r}$ as in the definition of~\eqref{eqn::CGI}.
Fix $\bs{y}=(y_1, \ldots, y_n)\in\LX_n$ and $\bs{t}=(t_1, \ldots, t_{\mathfrak{r}})$, and for $R>0$, we set $x_j=Ry_j$ for $j\in\{1, \ldots, n\}$ and $u_m=t_m+Ry_{s_m}$ for $m\in\{1, \ldots, \mathfrak{r}\}$. Let $R\to\infty$,
expanding the algebraic factors gives
\[
\begin{aligned}
    \integrand{n}(\bs x;\bs u)
    ={}&R^{((n-2\mathfrak r)^2-n)/\kappa}
       \exp\left(\frac{\sqrt{2\lambda} R}{\kappa}
          \left(\sum_jy_j-2\sum_my_{s_m}\right)
          \right)\\
    &\times
       \frac{\displaystyle\prod_{i<j}(y_j-y_i)^{2/\kappa}
             \prod_{p<q}(y_{s_q}-y_{s_p})^{8/\kappa}}
            {\displaystyle\prod_{m=1}^{\mathfrak r}
             \prod_{\ell\ne s_m}(y_{s_m}-y_\ell)^{4/\kappa}}\times\prod_{m=1}^{\mathfrak r}
        \left(\ee^{-2\sqrt{2\lambda} t_m/\kappa}t_m^{-4/\kappa}\right)
       \left(1+O(R^{-1})\right).
\end{aligned}
\]
Note that $\sum_jy_j-2\sum_my_{s_m}=\sum_j\sigma_jy_j$.
	 Taking argument $\pi$ for each negative real factor, we have
	\begin{equation}\label{eqn::CGI_gap_product}
\begin{aligned}
\prod_{m=1}^{\mathfrak r}\prod_{\ell\ne s_m}(y_{s_m}-y_\ell)^{4/\kappa}
={}&\exp\left(\frac{4\pi \ii}{\kappa}\sum_m(n-s_m)\right)
  \prod_{\substack{i<j\\\sigma_i=\sigma_j=-1}}(y_j-y_i)^{8/\kappa}
  \prod_{\substack{i: \sigma_i=-1\\j: \sigma_j=+1}}|y_i-y_j|^{4/\kappa}.
\end{aligned}
\end{equation}
Since $\ud u_m=\ud t_m$, the integrated $\bs t$-dependent factor is (see~\cite[Eq.~(5.9.2)]{NISTDLMF})
\[
\begin{aligned}
	\ointclockwise_{\mathcal{H}_{s_1}}\ud t_1\cdots\ointclockwise_{\mathcal{H}_{s_{\mathfrak{r}}}}\ud t_{\mathfrak{r}} \;\prod_{m=1}^{\mathfrak r}
          \left(\ee^{-2\sqrt{2\lambda} t_m/\kappa}t_m^{-4/\kappa}\right)
    &=\left(\ointclockwise_{\mathcal{H}}
          \ee^{-2\sqrt{2\lambda} t/\kappa}t^{-4/\kappa}\,\ud t\right)^{\mathfrak r}\\
    &=\left(\frac{2\pi \ii\,\ee^{4\pi \ii/\kappa}}{\Gamma(4/\kappa)}
          \left(\frac{2\sqrt{2\lambda}}{\kappa}\right)^{4/\kappa-1}\right)^{\mathfrak r},
\end{aligned}
\]
where $\mathcal{H}_{s}$ is the translate of the Hankel contour.
The normalization in~\eqref{eqn::CGI_cst_def} cancels the phase factors, thus the normalized integral has the asymptotic expansion
\begin{equation}\label{eqn::CGI_large_separation}
\begin{split}
    \LGchord{n}(R\bs y)
    ={}&
    \left(\frac{(2\sqrt{2\lambda}/\kappa)^{4/\kappa-1}}{\Gamma(4/\kappa)}\right)^{\mathfrak r}
    R^{((n-2\mathfrak r)^2-n)/\kappa}
    \exp\left(\frac{\sqrt{2\lambda} R}{\kappa}\sum_{j=1}^n\sigma_jy_j\right)
      \prod_{i<j}(y_j-y_i)^{2\sigma_i\sigma_j/\kappa}(1+o(1)),
\end{split}
\end{equation}
whose leading coefficient is positive for every $\kappa>0$ and $\lambda>0$.
\medbreak
Choose $y_j=2^j$.
The signed sums $\sum_j\sigma_jy_j$ are distinct as $\sigma$ ranges over $\{\pm1\}^n$, so the $2^n$ functions have distinct exponential rates $\sum_j\sigma_jy_j$.
Suppose there is a nontrivial complex linear relation, and choose a term with nonzero coefficient whose exponential rate is maximal.
Divide the relation by this term and let $R\to\infty$.
Every other term with nonzero coefficient tends to zero, since its smaller exponential rate dominates any difference in powers of $R$.
The limit therefore forces the chosen coefficient to vanish, a contradiction.
\end{proof}

\begin{proof}[Proof of Proposition~\ref{prop::CGI}]
Lemma~\ref{lem::CGI_equations} proves the assertions about the integral and its equations.
Since the equations have real coefficients, $\Re\LGchord{n}$ satisfies the same BPZ equations and first Ward identities.
The leading coefficient in~\eqref{eqn::CGI_large_separation} is positive, so taking the real part preserves the leading asymptotic.
The argument in the proof of Lemma~\ref{lem::CGI_independence} therefore also proves linear independence of the real parts.
Together with the dimension formula in Theorem~\ref{thm::BPZ_decomposition}, this shows that the collection~\eqref{eqn::CGI_basis} is a basis of $\LSchord{n}^{(\lambda)}$.
\end{proof}

\begin{samepage}
\subsection{Proof of Propositions~\ref{prop::CGI_kappa4} and~\ref{prop::LFchordsigma}}
\begin{proof}[Proof of Proposition~\ref{prop::CGI_kappa4}]
At $\kappa=4$, the integrand is meromorphic in the screening variables.
Exponential decay allows us to close each Hankel contour at infinity and apply the residue theorem at its only enclosed pole $u_m=x_{s_m}$.
Since $\sum_jx_j-2\sum_mx_{s_m}=\sum_j\sigma_jx_j$, successive residues give
\[
\begin{aligned}
    \LGchordGFF{n}(\bs x)
    ={}&\cstchordGFF{n}(-2\pi \ii)^{\mathfrak r}
      \exp\left(\frac{\sqrt{2\lambda}}{4}\sum_j\sigma_jx_j\right)
       \times
      \frac{\displaystyle\prod_{i<j}(x_j-x_i)^{1/2}
             \prod_{p<q}(x_{s_q}-x_{s_p})^2}
           {\displaystyle\prod_{m=1}^{\mathfrak r}
              \prod_{\ell\ne s_m}(x_{s_m}-x_\ell)}\\
    ={}&\LFchord{n}(\bs x).
\end{aligned}
\]
Here the last equality uses~\eqref{eqn::CGI_gap_product} with $\kappa=4$, and the normalization in~\eqref{eqn::CGI_cst_def} cancels the phase factors.
\end{proof}
\end{samepage}
\begin{proof}[Proof of Proposition~\ref{prop::LFchordsigma}]
The assertion follows from Propositions~\ref{prop::CGI}-\ref{prop::CGI_kappa4}, and Theorem~\ref{thm::BPZ_decomposition}.
\end{proof}

\subsection{Proof of Proposition~\ref{prop::GFF_chord}}
\label{subsec::GFF_coupling}

\paragraph*{Half-plane capacity.}
An $\HH$-hull is a relatively closed subset $K$ of $\HH$ such that $K$ is bounded and $\HH\setminus K$ is simply connected. 
Let $g_K$ be
the conformal map from $\HH\setminus K$ onto $\HH$ with the normalization $\lim_{z\to\infty}|g_K(z)-z|=0$. We call $g_K$ the mapping-out function of $K$. Note that there exists a constant $c_K\ge 0$ such that 
\[g_K(z)=z+\frac{2c_K}{z}+o(1/|z|), \qquad \text{as }z\to\infty.\]
The constant $c_K$ is called the half-plane capacity and denoted by $\mathrm{hcap}(K)$. 

\paragraph*{Loewner chain.}
Fix $x\in\R$ and $T>0$. Let $\eta: [0,T]\to\overline{\HH}$ be a continuous simple curve such that $\eta_0=x$ and $\eta_{(0,T)}\subset\HH$. For each $t\in[0,T]$, let $g_t$ be the mapping-out function of $\eta_{[0,t]}$. We say that $\eta$ is parameterized by the half-plane capacity if, for all $t\in [0,T]$, 
\[g_t(z)=z+\frac{2t}{z}+o(1/|z|), \qquad\text{as }z\to\infty.\]
These maps $\{g_t: t\in [0,T]\}$ solve the Loewner equation
\begin{equation*}
	\partial_t g_t(z)=\frac{2}{g_t(z)-W_t}, \qquad g_0(z)=z,
\end{equation*}
where $W: [0,T]\to \R$ is a continuous function called the driving function of $\eta$. 

Fix $\kappa>0$ and $x\in\R$. The process $\SLE_{\kappa}$ in $(\HH; x, \infty)$ is defined as Loewner chain with driving function $W_t=\sqrt{\kappa}B_t+x$ where $B_t$ is standard one-dimensional Brownian motion. When $\kappa\in (0,4]$, it is almost surely a continuous simple curve. 

\paragraph*{Green's function.}
For the upper half-plane $\HH$, the Green's function is defined as 
\begin{equation*}
	\Green_{\HH}(z,w)= \log\left|\frac{z-\overline{w}}{z-w}\right|,\qquad z,w\in\HH. 
\end{equation*}
For a simply connected domain $D\subsetneq \C$, we define the Green's function by conformal invariance:
\begin{equation*}
	\Green(D;z,w)=\Green(\HH;\varphi(z),\varphi(w)), \qquad z,w\in D,
\end{equation*}
where $\varphi$ is any conformal map from $D$ onto $\HH$. 

\paragraph*{Gaussian free field (GFF).}
For a domain $D\subsetneq \mathbb{C}$ and $f,g\in L^2(D)$, we denote their $L^2(D)$ inner product by $(f,g)$.
We denote by $H_s(D)$ the space of real-valued smooth functions compactly supported in $D$. We endow this space with the Dirichlet inner product
\[
(f, g)_{\nabla} := \frac{1}{2\pi} \int_D \nabla f(z) \cdot \nabla g(z)d^2z.
\]
We denote by $H(D)$ the Hilbert space completion of \( H_s(D) \) with respect to the Dirichlet inner product.

The Dirichlet GFF on $D$ is a random sum of the form $\Gamma = \sum_{j=1}^\infty f_j X_j$, where $X_j$ are i.i.d. standard normal random variables and $ (f_j)_{j \geq 1} $ is an orthonormal basis for $H(D)$. This sum almost surely diverges in $H(D)$; however, it converges almost surely in the space of distributions. See e.g.~\cite{SheffieldGFFMath} and~\cite{BerestyckiPowellGFFandLQG} for more details.  
In general, for any harmonic function $\varphi$ on $D$, we define the GFF with boundary data $\varphi$ as the Dirichlet GFF plus $\varphi$. 
Note that $\Green(D;\cdot,\cdot)$ is the covariance kernel of the Dirichlet GFF $\Gamma$ on $D$: for $f,g\in H_s(D)$,
\begin{equation*}
	\operatorname{Cov}\big(\langle\Gamma,f\rangle,\langle\Gamma,g\rangle\big) =\int_D\int_D f(z)\Green(D;z,w)g(w)\,d^2z\,d^2w.
\end{equation*}

\paragraph*{Level lines of the GFF.}
The coupling between $\mathrm{SLE}_4$ and GFF level line is characterized by a domain Markov property. The coupling between $\mathrm{SLE}_4$ and GFF level lines was established in~\cite{DubedatSLEFreeField,SchrammSheffieldContinuumGFF} and subsequently developed for more general boundary data, see e.g., \cite{WangWuLevellinesGFFI,PeltolaWuGlobalMultipleSLEs,AruSepulvedaWernerBTLStwoDGFF,AruLupuSepulvedaFPSGFF,KarrilaPeltolaSchougGFFdegeneratelevelline}.
Conditional on an initial segment of the level line, the restriction of the field to the remaining domain is an independent Dirichlet GFF plus the harmonic function obtained by transporting the boundary data under the corresponding Loewner map. Equivalently, along the Loewner evolution the harmonic mean is a local martingale, and its quadratic covariation exactly compensates for the decrease of the Green's function.

\begin{proof}[Proof of Proposition~\ref{prop::GFF_chord}]
Assume $\eta$ is driven by the Loewner chain whose driving function satisfies~\eqref{eqn::LFchord_sigma_SDE}:
\begin{align} \label{eqn::GFF_chord_aux2}
	\begin{split}
		\ud W_t= & 2\ud B_t+4\left(\partial_j\log\LFchord{n}\right)(g_t(x_1), \ldots, g_t(x_{j-1}), W_t, g_t(x_{j+1}), \ldots, g_t(x_{n}))\ud t \\
		= & 2\ud B_t+\left(2\sigma_j\sum_{i\neq j}\frac{\sigma_i}{W_t-g_t(x_i)}+\sqrt{2\lambda}\sigma_j\right)\ud t.
	\end{split}
\end{align}

First, we derive a martingale process for $\eta$. 
For $z\in\HH$, let us prove that the process
\begin{align}\label{eqn::GFF_chord_aux3}
	M_t(z):={}&\pi\sum_{i=1}^n\sigma_i-\sigma_j\arg(g_t(z)-W_t)
	-\sum_{i\neq j}\sigma_i\arg(g_t(z)-g_t(x_i))-\frac{\sqrt{2\lambda}}{2}\Im(g_t(z)),
\end{align}
is a local martingale.
By~\eqref{eqn::GFF_chord_aux2} and It\^o's formula, we have
\begin{align*}
	\ud\arg(g_t(z)-W_t)
	={}&-2\Im\left(\frac{1}{g_t(z)-W_t}\right)\ud B_t-\left(2\sigma_j\sum_{i\neq j}\frac{\sigma_i}{W_t-g_t(x_i)}+\sqrt{2\lambda}\sigma_j\right)
	\Im\left(\frac{1}{g_t(z)-W_t}\right)\ud t,\\
	\ud\arg(g_t(z)-g_t(x_i))
	={}&-\frac{2}{g_t(x_i)-W_t}\Im\left(\frac{1}{g_t(z)-W_t}\right)\ud t,
	\qquad i\neq j,\\
	\ud\Im(g_t(z))
	={}&2\Im\left(\frac{1}{g_t(z)-W_t}\right)\ud t.
\end{align*}
Plugging these identities into~\eqref{eqn::GFF_chord_aux3}, we have
\begin{align}\label{eqn::GFF_chord_martingale_diff}
	\ud M_t(z)=2\sigma_j\Im\left(\frac{1}{g_t(z)-W_t}\right)\ud B_t.
\end{align}
Thus, $M_t(z)$ is a local martingale.

\medbreak
Second, we derive the quadratic variation of the local martingale. Recalling that $\Green_{\HH}(z,w)=\log|(z-\overline{w})/(z-w)|$. For $z,w\in\HH$,~\eqref{eqn::GFF_chord_martingale_diff} gives
\begin{align}\label{eqn::GFF_chord_Hadamard}
	\ud\langle M_t(z),M_t(w)\rangle
	=4\Im\left(\frac{1}{g_t(z)-W_t}\right)
	\Im\left(\frac{1}{g_t(w)-W_t}\right)\ud t
	=-\ud \Green_{\HH}(g_t(z),g_t(w)).
\end{align}

\medbreak
Third, we derive the coupling with $\eta$ and the GFF $\Phi$. Combining~\eqref{eqn::GFF_chord_aux3} and~\eqref{eqn::GFF_chord_Hadamard}, the pair $(\Phi, \eta)$ can be coupled such that given $\eta_{[0,t]}$, the field $\Phi(\cdot)+\varphi_{\sigma}(\bs{x};\cdot)-\frac{\sqrt{2\lambda}}{2}\Im(\cdot)$ restricted to $\HH\setminus\eta_{[0,t]}$ has the same law as
\begin{align*}
	\bigg(\Phi(\cdot)+\varphi_{\sigma}(g_t(x_1),\ldots,g_t(x_{j-1}),W_t,
	g_t(x_{j+1}),\ldots,g_t(x_n);\cdot)-\frac{\sqrt{2\lambda}}{2}\Im(\cdot)\bigg)\circ g_t.
\end{align*}
This confirms the coupling so that $\eta$ is the level line of $\Phi+\varphi_{\sigma}(\bs{x};\cdot)-\frac{\sqrt{2\lambda}}{2}\Im(\cdot)$ starting from $x_j$.
\end{proof}

\appendix
\section{Dub\'edat's commutation relation}
\label{appendix::Dubedat}
\paragraph*{Axioms of commutation relation.}
Let $H\subset \HH$ be a simply connected domain such that $\HH\setminus H$ is a hull and let $x_1,\cdots,x_n$ be distinct prime ends of $\partial H$. Let $U_1,\cdots,U_n$ be, respectively, compact neighborhoods of $x_1,\cdots,x_n$ in $\overline H$ and such that $U_i\cap U_j=\emptyset$ for $1\le i<j\le n$. We will consider probability measures $\PP_{(H;x_1,\ldots,x_n)}^{(U_1,\ldots,U_n)}$ on $n$-tuples of unparametrized continuous curves in $U_1,\ldots,U_n$ starting from $x_1,\ldots,x_n$ and exiting $U_1,\ldots,U_n$ almost surely. We call such a family of measures indexed by different choices of $(U_1,\ldots,U_n)$ compatible if for all $U_j\subset U'_j$, $1\le j\le n$, we have $\PP_{(H;x_1,\ldots,x_n)}^{(U_1,\ldots,U_n)}$ is obtained by restricting the curves under $\PP_{(H;x_1,\ldots,x_n)}^{(U'_1,\ldots,U'_n)}$ to their parts before they first exit $U_1,\ldots,U_n$, respectively.

Fix $\kappa>0$ and $n\ge2$. A locally (weak) commuting $n$-chordal $\SLE_\kappa$ is a compatible family of measures $\PP_{(H; x_1,\ldots, x_n)}^{(U_1, \ldots, U_n)}$ on $n$-tuples of continuous non-self-crossing curves $(\eta^1, \ldots, \eta^n)$ for all $(H; x_1, \ldots, x_n)$ and $(U_1, \ldots, U_n)$ as above, and satisfies additionally the (CI) ((wCI) for the weak case), (DMP), (MARG) conditions below.

\begin{itemize}
\item[(CI)] Conformal invariance. If $\varphi:H\to H'$ is a conformal map fixing $\infty$, then
\begin{equation}\label{eqn::chordal_CI}
	\varphi^*\PP_{(H';\varphi(x_1),\ldots,\varphi(x_n))}^{(\varphi(U_1),\ldots,\varphi(U_n))}
	=\PP_{(H;x_1,\ldots,x_n)}^{(U_1,\ldots,U_n)}.
\end{equation}
\item[(wCI)] Weak conformal invariance. The same identity~\eqref{eqn::chordal_CI} only holds for conformal maps $\varphi:H\to H'$ fixing infinity, with the normalization $\varphi(z)/z\to1$ as $z\to \infty$.
\item[(DMP)] Domain Markov property: Let $\left(\eta^{1},\ldots,\eta^{n}\right)\sim \PP_{(H; x_1,\ldots, x_n)}^{(U_1,\ldots,U_n)}$ and parametrize $\eta^{1},\ldots,\eta^{n}$ by their own half-plane capacity. Let $\bs{t}=(t_1,\ldots,t_n)$, where $t_j$ is a stopping time for $\eta^{j}$ and $\eta^{j}_{[0,t_j]}$ is contained in the interior of $U_j$. Let
\begin{equation*}
	\tilde{U}_j=U_j\setminus \eta^{j}_{[0,t_j)}, \quad \tilde{\eta}^{j}=\eta^{j}\setminus \eta^{j}_{[0,t_j]}, \quad 1\le j\le n; \quad \tilde{H}=H \setminus\left(\bigcup_{j=1}^n \eta^{j}_{[0,t_j]} \right).
\end{equation*}
Then conditionally on $\bigcup_{j=1}^n \eta^{j}_{[0,t_j]}$, we have 
\begin{equation*}
	\left(\tilde{\eta}^{1},\ldots,\tilde{\eta}^{n}\right) \sim \PP_{(\tilde{H};\eta_{t_1}^{1},\ldots,\eta_{t_n}^{n})}^{(\tilde{U}_1,\ldots,\tilde{U}_n)}.
\end{equation*}
\item[(MARG)] Marginal laws: Let $\left( \eta^{1},\ldots,\eta^{n} \right)\sim \PP_{(\HH;x_1,\ldots,x_n)}^{(U_1,\ldots,U_n)}$. For $1\le j\le n$, we assume that there exists a $C^2$ function $b_j:\LX_n \to \R$ such that the half-plane capacity parametrized Loewner driving function $t\mapsto W_t^{j}$ of $\eta^{j}$ and the Loewner chain $g_t^{j}$ associated with $\eta^{j}$ satisfy
\begin{equation}
	\label{eqn::MAGN_SDE}
	\begin{cases}
		\ud W_t^{j}=\sqrt{\kappa} \ud B_t^{j} + b_j\left(g_t^{j}(x_1),\ldots,g_t^{j}(x_{j-1}),W_t^{j},g_t^{j}(x_{j+1}),\ldots,g_t^{j}(x_n)\right) \ud t,\quad W_0^{j}=x_j; \\
		\ud g_t^{j}(x_k)= \frac{2}{g_t^{j}(x_k)-W_t^{j}} \ud t \text{ for } k\neq j,\quad g_0^{j}(x_k)=x_k \text{ for } k\neq j;
	\end{cases}
\end{equation}
where $B^{j}$ is a one-dimensional standard Brownian motion.
\end{itemize}

The difference between the two conformal invariance axioms is important here. For M\"obius transformations of $\HH$, a map fixing infinity has the form $\varphi(z)=az+b$, with $a>0$. The normalization in (wCI) forces $a=1$. Among M\"obius transformations of $\HH$, the axiom (wCI) implies translation invariance of the law; whereas, the axiom (CI) implies translation invariance as well as certain scaling invariance requirement.

\begin{proposition}\label{prop::chordal_commutation_BPZ}
Fix $\kappa>0$ and $n\ge2$. 
\begin{itemize}
	\item The existence of locally weak commuting $n$-chordal $\SLE_{\kappa}$ is equivalent to the existence of $\LZ: \LX_n\to \R_{>0}$, called partition function, such that
	$b_j=\kappa\partial_j\log\LZ$ for all $1\le j\le n$ 
	and $\LZ$ satisfies the BPZ system~\eqref{eqn::BPZ_chordal} and the first Ward identity~\eqref{eqn::Ward} for some constants $\lambda,\nu\in\R$.
	\item The existence of locally commuting $n$-chordal $\SLE_{\kappa}$ is equivalent to the existence of $\LZ: \LX_n\to \R_{>0}$, called partition function, such that
	$b_j=\kappa\partial_j\log\LZ$ for all $1\le j\le n$ 
	and $\LZ$ satisfies the BPZ system~\eqref{eqn::BPZ_chordal} with $\lambda=0$, the first Ward identity~\eqref{eqn::Ward} with $\nu=0$, and the second Ward identity~\eqref{eqn::Ward_2nd} for some constant $\xi\in\R$.
\end{itemize}
\end{proposition}

The second item is proved in \cite[Sections~3.1 and~4.1]{DubedatCommutationSLE}. We briefly summarize the proof below. 
\begin{lemma}\label{lem::weak_commutation}
Suppose that the curves satisfy local weak commutation, i.e. (wCI), (DMP) and (MARG). Then there exists $\LZ: \LX_n\to \R_{>0}$ such that $b_j=\kappa\partial_j\log\LZ$ for all $1\le j\le n$ and $\LZ$ satisfies the BPZ system~\eqref{eqn::BPZ_chordal} and the first Ward identity~\eqref{eqn::Ward} for some constants $\lambda,\nu\in\R$.
\end{lemma}
\begin{proof}
 The infinitesimal generators associated with~\eqref{eqn::MAGN_SDE} are
\begin{equation*}
	\mathcal{L}_j=\frac{\kappa}{2}\partial_j^2+b_j\partial_j+\sum_{\ell\neq j}\frac{2}{x_\ell-x_j}\partial_\ell,\qquad \text{for }j\in\{1,\ldots,n\}.
\end{equation*}
Comparing the two orders of growing short initial segments of $\eta^i$ and $\eta^j$, \textup{(DMP)} gives
\begin{equation} \label{eqn::commutation_chordal_classify_aux1}
	[\mathcal{L}_i,\mathcal{L}_j]=\frac{4(\mathcal{L}_j-\mathcal{L}_i)}{(x_i-x_j)^2},\qquad \text{for }i\neq j.
\end{equation}
Comparing the coefficients of $\partial_i\partial_j$, we obtain $\partial_i b_j=\partial_j b_i$. Then there exists a positive function $\LZ$ such that $b_j=\kappa\partial_j\log\LZ$ for $1\le j\le n$. Plugging $b_j=\kappa\partial_j\log\LZ$ into the first-order coefficients~\eqref{eqn::commutation_chordal_classify_aux1}, we obtain
\begin{equation*}
	\partial_i\left(\frac{\LD_j\LZ}{\LZ}\right)=0,\qquad \text{for }i\neq j,
\end{equation*}
and $\LD_j$ is defined in~\eqref{eqn::chordal_BPZ_operators}. Hence $\LD_j\LZ=F_j(x_j)\LZ$ for some functions $F_j$. Applying \textup{(wCI)} to translations of $\HH$ shows that the drift functions are translation-invariant:
\begin{equation*}
	b_j(\bs{x}+a)=b_j(\bs{x}),\qquad\text{for all }j\in\{1, \ldots, n\}\text{ and }a\in\R.
\end{equation*}
It follows that $\kappa\partial_j(\LT\log\LZ)=\LT b_j=0$ for all $j$, where $\LT=\sum_j\partial_j$ is defined in~\eqref{eqn::chordal_Ward_12}. Therefore $\kappa\LT\log\LZ=\nu$ for some $\nu\in\R$, which gives~\eqref{eqn::Ward}. Moreover, $\LD_j\LZ/\LZ$ is translation-invariant, so each $F_j$ is a constant. Applying the commutator identity~\eqref{eqn::chordal_operator_commutator} to $\LZ$ now gives
\begin{equation*}
	0=[\LD_i,\LD_j]\LZ=\frac{4(F_j-F_i)}{(x_i-x_j)^2}\LZ.
\end{equation*}
Since $\LZ>0$, all the constants $F_j$ agree. Writing their common value as $\lambda/\kappa$, we obtain~\eqref{eqn::BPZ_chordal}. This completes the proof.
\end{proof}

\begin{lemma}\label{lem::commutation}
Suppose that the curves satisfy local commutation, i.e. (CI), (DMP) and (MARG). Then there exists $\LZ: \LX_n\to \R_{>0}$ such that $b_j=\kappa\partial_j\log\LZ$ for all $1\le j\le n$ and $\LZ$ satisfies the BPZ system~\eqref{eqn::BPZ_chordal} with $\lambda=0$, the first Ward identity~\eqref{eqn::Ward} with $\nu=0$, and the second Ward identity~\eqref{eqn::Ward_2nd} for some constant $\xi\in\R$.
\end{lemma}
\begin{proof}
If the curves satisfy \textup{(CI)}, then they in particular satisfy \textup{(wCI)}, thus the analysis in Lemma~\ref{lem::weak_commutation} apply. Applying \textup{(CI)} to scaling of $\HH$ shows that the drift functions are homogeneous of degree $-1$:
\begin{equation*}
	b_j(a\bs{x})=a^{-1}b_j(\bs{x}),\qquad\text{for all }j\in\{1,\ldots,n\}\text{ and }a>0.
\end{equation*}
It follows that $\kappa\partial_j(\LE\log\LZ)=\LE b_j+b_j=0$ for all $j$, where $\LE=\sum_jx_j\partial_j$ is defined in~\eqref{eqn::chordal_Ward_12}. Therefore $\kappa\LE\log\LZ=\xi$ for some $\xi\in\R$, which gives~\eqref{eqn::Ward_2nd}. Applying the commutator identities~\eqref{eqn::BPZ_dilation} and~\eqref{eqn::translation_dilation} to $\LZ$ gives
\begin{equation*}
	0=[\LT,\LE]\LZ=\LT\LZ=\frac{\nu}{\kappa}\LZ,
	\qquad
	0=[\LD_j,\LE]\LZ=2\LD_j\LZ=\frac{2\lambda}{\kappa}\LZ.
\end{equation*}
Since $\LZ>0$, we obtain $\nu=\lambda=0$.
\end{proof}

To address the inverse direction in Proposition~\ref{prop::chordal_commutation_BPZ}, we introduce Brownian loop measure and multi-time martingale. 
\paragraph*{Brownian loop measure.} Brownian loop measure is a $\sigma$-finite measure on planar unrooted Brownian loops. We denote this measure by $\blm^\mathrm{loop}$ and refer to~\cite{LawlerWernerBrownianLoopsoup} for its definition and properties. The total mass of $\blm^\mathrm{loop}$ is infinite, but the mass on macroscopic loops is finite: suppose $\Omega$ is a domain and $K_1, K_2$ are two disjoint compact subsets of $\overline{\Omega}$, the total mass of Brownian loops that stay in $\Omega$ and intersect both $K_1$ and $K_2$ is finite and we denote this quantity by $\blm(\Omega; K_1, K_2)$. 
In general, for $n\ge 2$, suppose $K_1, \ldots, K_n$ are disjoint compact subsets of $\overline{\Omega}$, we denote
\begin{equation}\label{eqn::blm_def}
\blm(\Omega; K_1, \ldots, K_n)=\sum_{j=2}^n \blm^{\mathrm{loop}}\left(\ell\subset\Omega: \ell\cap K_i\neq \emptyset\text{ for at least }j \text{ of the }i\in\{1, \ldots, n\}\right). 
\end{equation}
See~\cite{LawlerPartitionFunctionsSLE} for its properties and see~\cite{DubedatEulerIntegralsCommutingSLEs, DubedatCommutationSLE, KozdronLawlerMultipleSLEs, PeltolaWuGlobalMultipleSLEs} for its alternative forms. 

\paragraph*{Multi-time martingale.}
We fix
\begin{align}
\lambda\in\R, \qquad \kappa>0, \qquad \mathfrak{b}=\frac{6-\kappa}{2\kappa}, \qquad \mathfrak{c}=\frac{(6-\kappa)(3\kappa-8)}{2\kappa}. 
\end{align}
We construct the locally commuting chordal $\SLE$ by tilting independent chordal SLEs with the following multi-time martingale constructed from $\LZ$. Let $\eta^1,\ldots,\eta^n$ be independent chordal $\SLE_\kappa$ curves in $(\HH;x_1,\infty),\ldots,(\HH;x_n,\infty)$, respectively, parametrized by their own half-plane capacities and with driving functions $W^1,\ldots,W^n$. For $\bs{t}=(t_1,\ldots,t_n)$, let $g^j_{t_j}$ be the mapping-out function of $\eta^j_{[0,t_j]}$, and let $g_{\bs{t}}$ be the mapping-out function of $K_{\bs{t}}=\bigcup_{j=1}^n\eta^j_{[0,t_j]}$.
On the event $\mathcal E_{\bs{t}}$ that these initial segments are pairwise disjoint, define $g_{\bs{t},j}$ by $g_{\bs{t}}=g_{\bs{t},j}\circ g^j_{t_j}$ and let $\bs{X}_{\bs{t}}=(X^1_{\bs{t}},\ldots,X^n_{\bs{t}})$, where $X^j_{\bs{t}}=g_{\bs{t},j}(W^j_{t_j})$. 
\begin{lemma}[{\cite[Proposition~2.1]{HuangWuYangMultipleSLEsDysonBM}}]
\label{lem::multitime_mart}
The process
\begin{equation}\label{eqn::chordal_multitime_martingale}
	M_{\bs{t}}(\LZ)
	=\one_{\mathcal E_{\bs{t}}}\LZ(\bs{X}_{\bs{t}})
	\prod_{j=1}^n g'_{\bs{t},j}(W^j_{t_j})^{\mathfrak{b}}
	\exp\left(\frac{\mathfrak{c}}{2}\blm\big(\HH;\eta^1_{[0,t_1]},\ldots,\eta^n_{[0,t_n]}\big)-\frac{\lambda}{\kappa}\hcap(K_{\bs{t}})\right)
\end{equation}
is a multi-time local martingale if and only if $\LZ$ is smooth and satisfies the BPZ system~\eqref{eqn::BPZ_chordal}. 
\end{lemma}

\begin{proof}[Proof of Proposition~\ref{prop::chordal_commutation_BPZ}]
Lemmas~\ref{lem::weak_commutation} and~\ref{lem::commutation} give the existence of the partition functions. 
Next, suppose that $\LZ>0$ satisfies~\eqref{eqn::BPZ_chordal} and~\eqref{eqn::Ward}, and let $b_j=\kappa\partial_j\log\LZ$. 
We weight independent chordal $\SLE_{\kappa}$ $(\eta^1, \ldots, \eta^n)$ by the multi-time local martingale $M_{\bs{t}}(\LZ)$ defined by~\eqref{eqn::chordal_multitime_martingale}.
Girsanov's theorem gives the marginal laws~\eqref{eqn::MAGN_SDE}. The martingale property gives \textup{(DMP)}. The first Ward identity gives the translation invariance of the measure, resulting in \textup{(wCI)}. If in addition $\lambda=\nu=0$ and $\LZ$ satisfies~\eqref{eqn::Ward_2nd}, then the second Ward identity further gives the scaling invariance of the measure, resulting in \textup{(CI)}. This yields the required locally (weak) commuting system.
\end{proof}

A special case of Lemma~\ref{lem::multitime_mart} will be used later and we record it here. 
\begin{corollary}\label{cor::SLE_mart}
Fix $j\in\{1,\ldots,n\}$, let $\eta$ be chordal $\SLE_{\kappa}$ in $(\HH;x_j,\infty)$. For a function $\LZ:\LX_{n}\to \R$, define
\begin{equation}\label{eqn::BPZ_chordal_mart}
	M_t(\LZ)=\LZ (g_t(x_1), \ldots, g_t(x_{j-1}), W_t, g_t(x_{j+1}), \ldots, g_t(x_{n})) \prod_{\ell \neq j} g'_t(x_{\ell})^{\mathfrak{b}}.
\end{equation}
Then the process $M(\LZ)$ is a local martingale with respect to $\eta$ if and only if $\LZ$ is smooth and satisfies the following BPZ equation:
\begin{equation}\label{eqn::BPZ_chordal_j}
	\frac{\kappa}{2}\partial_j^2\LZ	+\sum_{\ell\neq j}\left( \frac{2}{x_\ell-x_j}\partial_\ell\LZ -\frac{2\mathfrak{b}}{(x_\ell-x_j)^2}\LZ \right) =0.
\end{equation}
\end{corollary}

\section{Positive solutions with \texorpdfstring{$\lambda=\nu=0$}{lambda = nu = 0}}
\label{sec::solutions_lambdazero}
In this section, we first give preliminaries in Section~\ref{subsec::SLE_pre}. 
We recall conclusions for multi-chordal $\SLE_{\kappa}$ in Section~\ref{subsec::multichordal_SLE}. 
We recall conclusions for half-watermelon $\SLE_{\kappa}$ in Section~\ref{subsec::halfwatermelon_SLE}. 
We construct positive solutions for the space $\LSchord{n}^{(0;0)}$ in Section~\ref{subsec::LZmix} using results from Sections~\ref{subsec::multichordal_SLE}-\ref{subsec::halfwatermelon_SLE}. As our construction relies crucially on the boundary perturbation of multi-chordal SLE (Lemma~\ref{lem::multichordalSLE_bp}) and of half-watermelon SLE (Lemma~\ref{lem::halfwatermelonSLE_bp}), the construction only works for $\kappa\in (0,4]$. 
We fix
\begin{align}
\kappa\in (0,4], \qquad \mathfrak{b}=\frac{6-\kappa}{2\kappa}, \qquad \mathfrak{c}=\frac{(6-\kappa)(3\kappa-8)}{2\kappa}. 
\end{align}

\subsection{Preliminaries}
\label{subsec::SLE_pre}

\paragraph*{Polygon.} We say that $(\Omega; \bs{x})=(\Omega; x_1, \ldots, x_n)$ 
is a (topological) $n$-polygon if $\Omega\subsetneq\C$ is simply connected, $\partial\Omega$ is locally connected, and $x_1, x_2, \ldots, x_n \in \partial\Omega$ are distinct points lying counterclockwise along the boundary. 
We say that $(\Omega; \bs{x})=(\Omega; x_1, \ldots, x_n)$ 
is a nice polygon if we assume further that the marked boundary points $x_1, x_2, \ldots, x_n$ lie on $C^{1+\eps}$-boundary segments, for some $\eps>0$, so that derivatives of conformal maps on $\Omega$ are defined there.

\paragraph*{Poisson kernel.}
For 1-polygon $(\HH;x)$ with $z\in\HH$, the Poisson kernel is defined as 
\begin{equation} \label{eqn::Poisson_H}
	\Poisson(\HH;x;z):=\frac{2\Im(z)}{|z-x|^2}, \qquad x\in \R,z\in \HH.
\end{equation}
For nice 1-polygon $(\Omega;x)$ with $z\in\Omega$, we extend its definition via conformal covariance:
\begin{equation} \label{eqn::Poisson_cov}
	\Poisson(\Omega;x;z):= |\varphi'(x)| \Poisson(\HH;\varphi(x);\varphi(z)),
\end{equation}
where $\varphi$ is any conformal map from $\Omega$ to $\HH$.

\paragraph*{Boundary Poisson kernel.}
For the 2-polygon $(\HH; x_1, x_2)$, the boundary Poisson kernel is defined as
\begin{equation}\label{eqn::bPoisson_H}
\Poisson(\HH; x_1, x_2)=|x_1-x_2|^{-2},\qquad x_1, x_2\in\R. 
\end{equation}
For nice 2-polygon $(\Omega; x_1, x_2)$, we extend its definition via conformal covariance: 
\begin{equation}\label{eqn::bPoisson_cov}
\Poisson(\Omega; x_1, x_2)=|\varphi'(x_1)\varphi'(x_2)|\Poisson(\HH; \varphi(x_1), \varphi(x_2)), 
\end{equation}
where $\varphi$ is any conformal map from $\Omega\to \HH$. 
Poisson kernel satisfies the following monotonicity.
Let $(\Omega; x_1, x_2)$ be a nice 2-polygon and let $U\subset\Omega$ be a simply connected subdomain that agrees with $\Omega$ in neighborhoods of $x_1$ and of $x_2$. Then
\begin{equation}\label{eqn::bPoisson_mono}
\Poisson(U; x_1, x_2)\le \Poisson(\Omega; x_1, x_2).
\end{equation}

For a $2$-polygon $(\Omega; x_1, x_2)$, let $\varphi: \Omega\to\HH$ be a conformal map with $\varphi(x_1)=x$ and $\varphi(x_2)=\infty$. Then $\SLE_{\kappa}$ in $(\Omega; x_1, x_2)$ is defined as the image of $\SLE_{\kappa}$ in $(\HH; x, \infty)$ under $\varphi^{-1}$. 
We denote the law of chordal $\SLE_{\kappa}$ in $(\Omega; x_1, x_2)$ by 
$\Pchordtwo(\Omega; x_1, x_2)$.
Its partition function is given by 
\begin{equation*}
\LZtwo(\Omega; x_1, x_2)=\Poisson(\Omega; x_1, x_2)^{\mathfrak{b}}. 
\end{equation*}

\paragraph*{SLE with force points.} Fix $\kappa>0$ and $\ell \ge 1, r\ge 1$. Denote $\bs{x}^L=(x^L_{\ell}, \ldots, x^L_1), \bs{x}^R=(x^R_1, \ldots, x^R_r)$ with $x^L_{\ell}<\cdots<x^L_1<0<x_1^R<\cdots<x_r^R$ and $\bs{\rho}^L=(\rho^L_{\ell}, \ldots, \rho^L_1)\in\R^{\ell}, \bs{\rho}^R=(\rho^R_1, \ldots, \rho^R_r)\in \R^r$. The process $\SLE_{\kappa}(\bs{\rho}^L; \bs{\rho}^R)$ in $\HH$ from $0$ to $\infty$ with force points $(\bs{x}^L; \bs{x}^R)$ is defined as the Loewner chain with driving function $W_t$ solves the SDE
\begin{align}
	\begin{cases}
		\displaystyle\ud W_t=\sqrt{\kappa}\ud B_t+\sum_{j=1}^{\ell}\frac{\rho_j^L\ud t}{W_t-V_t^{L, j}}+\sum_{j=1}^{r}\frac{\rho_j^R\ud t}{W_t-V_t^{R, j}}, \qquad W_0=0;\\
		\displaystyle\ud V_t^{L, j}=\frac{2\ud t}{V_t^{L, j}-W_t},\qquad V_0^{L, j}=x_j^L, \quad 1\le j\le \ell; \\
		\displaystyle\ud V_t^{R, j}=\frac{2\ud t}{V_t^{R,j}-W_t}, \qquad V_0^{R, j}=x_j^R, \quad 1\le j\le r. 
	\end{cases}
\end{align}
For a $(\ell+r+2)$-polygon $(\Omega; \bs{x}^L, x, \bs{x}^R, y)$ with $\bs{x}^L=(x^L_{\ell}, \ldots, x^L_1), \bs{x}^R=(x^R_1, \ldots, x^R_r)$, let $\varphi:\Omega\to\HH$ be a conformal map with $\varphi(y)=\infty$. 
$\SLE_{\kappa}(\bs{\rho}^L; \bs{\rho}^R)$ in $(\Omega; \bs{x}^L, x, \bs{x}^R, y)$ is defined as the image of $\SLE_{\kappa}(\bs{\rho}^L; \bs{\rho}^R)$ in $\HH$ from $\varphi(x)$ to $\infty$ with force points $(\varphi(\bs{x}^L); \varphi(\bs{x}^R))$ under $\varphi^{-1}$.

\paragraph{Curve Spaces. } For $2$-polygon $(\Omega; x, y)$, we denote by $\chamber(\Omega; x, y)$ the space of continuous simple curves $\eta\subset\overline{\Omega}$ from $x$ to $y$ such that $\eta\cap\partial\Omega=\{x,y\}$. 
We will consider the following two general curve spaces. 
\begin{itemize}
\item Fix $N\ge 1$, and for $2N$-polygon $(\Omega; \bs{x})$ with $\bs{x}=(x_1, \ldots, x_{2N})$, we consider $N$ disjoint curves in $\Omega$ connecting $\{x_1, \ldots, x_{2N}\}$ pairwise. 
Topologically, these $N$ curves form a planar pair partition, which we call a link pattern. More precisely, a link pattern $\alpha$ of size $N$ is a pair partition $\alpha=\{\{a_1,b_1\},\ldots,\{a_N,b_N\}\}$ of the set $\{1,\ldots,2N\}$, where there do not exist two links $\{a_i,b_i\},\{a_j,b_j\}\in\alpha$ such that $a_i<a_j<b_i<b_j$. We denote the set of such link patterns by $\LP_N$; in particular,
\[\LP_0=\{\emptyset\}, \qquad 
\#\LP_N=\Cat_N=\frac{1}{N+1}\binom{2N}{N}.
\]
Fix a $2N$-polygon $(\Omega;\bs{x})$ with $\bs{x}=(x_1,\ldots,x_{2N})$ and $\alpha\in\LP_N$. We denote by $\chamber_{\alpha}(\Omega;\bs{x})$ the space of continuous simple curves $\bs{\eta}=(\eta^1,\ldots,\eta^N)$ associated with $\alpha$ such that $\eta^j\in\chamber(\Omega;x_{a_j},x_{b_j})$ for all $j$ and $\eta^i\cap\eta^j=\emptyset$ for $i\neq j$.
\item Fix $p\ge 2$, and for $(p+1)$-polygon $(\Omega; \bs{x}, y)$ with $\bs{x}=(x_1, \ldots, x_p)$, we denote by $\chamberfusion{p}(\Omega; \bs{x}, y)$ the space of $p$-tuples of continuous simple curves $\bs{\gamma}=(\gamma^1, \ldots, \gamma^p)$ such that $\gamma^j\in\chamber(\Omega; x_j, y)$ for all $1\le j\le p$ and that $\gamma^i\cap\gamma^j=\{y\}$ for all $i\neq j$. 
\end{itemize}

\subsection{Multi-chordal SLE}
\label{subsec::multichordal_SLE}

\paragraph*{Pure partition functions.}
Fix $\kappa\in (0,8)$. 
Pure partition functions of multi-chordal $\SLE_{\kappa}$ are the recursive collection $\{\LZ_{\alpha} \colon \alpha \in \bigsqcup_{N\geq 0} \LP_N\}$ 
of functions $\LZ_{\alpha}(\HH; \cdot) \colon \LX_{2N}\to\R$
uniquely determined by the following four properties:
\begin{itemize}
	\item Chordal BPZ system~\eqref{eqn::BPZ_chordal} with $n=2N$ and $\lambda=0$. 

	\item M\"{o}bius covariance: for all M\"obius maps $\varphi$ of the upper half-plane $\HH$ such that $\varphi(x_{1}) < \cdots < \varphi(x_{2N})$, we have
	\begin{align*}
		\LZ_{\alpha}(\HH; x_{1},\ldots,x_{2N}) = 
		\prod_{j=1}^{2N} \varphi'(x_{j})^{\mathfrak{b}} 
		\times \LZ_{\alpha}(\HH; \varphi(x_{1}),\ldots,\varphi(x_{2N})).
	\end{align*}
	\item Asymptotics: with $\LZ_{\emptyset} \equiv 1$ for the empty link pattern $\emptyset \in \LP_0$, the collection $\{\LZ_{\alpha} \colon \alpha\in\LP_N\}$ satisfies the following recursive asymptotics property. Fix $j \in \{1,2, \ldots, 2N-1\}$ and $y \in (x_{j-1}, x_{j+2})$ (with the convention that $x_0 = -\infty$ and  $x_{2N+1} = +\infty$). We have
	\begin{align*}
		\lim_{x_j,x_{j+1}\to y} \frac{\LZ_{\alpha}(\HH; x_1,\ldots, x_{2N})}{ (x_{j+1}-x_j)^{-2\mathfrak{b}} }
		= 
		\begin{cases}
			\LZ_{\alpha/\{j,j+1\}}(\HH; x_1, \ldots, x_{j-1}, x_{j+2}, \ldots, x_{2N}), 
			& \quad \text{if }\{j, j+1\}\in\alpha , \\
			0 ,
			& \quad \text{if }\{j, j+1\} \not\in \alpha ,
		\end{cases}
	\end{align*}
	where $\alpha/\{k,l\}$ denotes the link pattern in $\LP_{N-1}$ obtained by removing $\{k,l\}$ from $\alpha$ and then relabeling the remaining indices so that they are the first $2(N-1)$ positive integers. 
\end{itemize}
The existence of pure partition functions is proved in~\cite{ KytolaPeltolaPurePartitionFunctions, PeltolaWuGlobalMultipleSLEs, WuHyperSLE, AngHoldenSunYu2023, FengLiuPeltolaWu2024} for different range of $\kappa$ and they together gives the existence for $\kappa\in (0,8)$. 
The uniqueness of pure partition functions when $\kappa\in (0,8)$ was originally proved in~\cite{FloresKlebanPDE2} under an extra constraint of power-law bound. Now the uniqueness without such constraint follows from Corollary~\ref{cor::BPZ_Catalan}.

We extend the definition of $\LZ_{\alpha}$ to general nice $2N$-polygon $(\Omega; \bs{x})=(\Omega; x_1, \ldots, x_{2N})$ as 
\begin{align}\label{eqn::PPF_COV}
	\LZ_{\alpha}(\Omega; x_1, \ldots, x_{2N}):=\prod_{j=1}^{2N}|\varphi'(x_j)|^{\mathfrak{b}}\times \LZ_{\alpha}(\HH; \varphi(x_1), \ldots, \varphi(x_{2N})),
\end{align}
where $\varphi$ is any conformal map from $\Omega$ onto $\HH$ with $\varphi(x_1)<\cdots<\varphi(x_{2N})$. 
The asymptotics become
\begin{align}\label{eqn::PPF_ASY}
\lim_{x_j, x_{j+1}\to \xi}\frac{\LZ_{\alpha}(\Omega; \bs{x})}{\Poisson(\Omega; x_j, x_{j+1})^{\mathfrak{b}}}=\begin{cases}
\LZ_{\alpha/\{j,j+1\}}(\Omega; \ddot{\bs{x}}_j), &\text{if }\{j,j+1\}\in\alpha,\\
0,&\text{if }\{j,j+1\}\not\in\alpha,
\end{cases}
\end{align}
where $\ddot{\bs{x}}_j=(x_1, \ldots, x_{j-1}, x_{j+2}, \ldots, x_{2N})$. 
When $\kappa\in (0,6]$, pure partition functions also satisfy the following power-law bound (proved in~\cite{WuHyperSLE}): 
\begin{equation}\label{eqn::PPF_PLB}
0<\LZ_{\alpha}(\Omega; \bs{x})\le \prod_{\{a,b\}\in\alpha}\Poisson(\Omega; x_a, x_b)^{\mathfrak{b}}. 
\end{equation}

\paragraph*{Multi-chordal SLE.}
Fix $\kappa\in (0,8)$ and $N\ge 1$. Fix $2N$-polygon $(\Omega; \bs{x})$ with $\bs{x}=(x_1, \ldots, x_{2N})$. 
The $N$-chordal $\SLE_{\kappa}$ in polygon $(\Omega; \bs{x})$ associated with link pattern $\alpha\in\LP_N$ is the unique probability measure on the configuration space $\chamber_{\alpha}(\Omega; \bs{x})$ with resampling property: for each $j\in\{1, \ldots, N\}$, the conditional law of $\eta^j$ given $\{\eta^i: i\neq j\}$ is chordal $\SLE_{\kappa}$ connecting $x_{a_j}$ and $x_{b_j}$ in the connected component $\Omega_j$ of the domain $\Omega\setminus\cup_{i\neq j}\eta^i$ having the end points $x_{a_j}$ and $x_{b_j}$ on its boundary. We denote its law by $\Pchord_{\alpha}=\Pchord_{\alpha}(\Omega; \bs{x})$. 
There are several literature about the existence and uniqueness of $N$-chordal $\SLE_{\kappa}$ for different range of $\kappa$: \cite{MillerSheffieldIG2, PeltolaWuGlobalMultipleSLEs, WuHyperSLE, BeffaraPeltolaWuUniqueness, AngHoldenSunYu2023, ZhanExistenceUniquenessMultipleSLE, FengLiuPeltolaWu2024}. 
They guarantee the existence and uniqueness of $N$-chordal $\SLE_{\kappa}$ for $\kappa\in (0,8)$. 

Suppose that $\bs{\eta}=(\eta^1,\ldots,\eta^N)$ is $N$-chordal $\SLE_{\kappa}$ in $(\HH;x_1,\ldots,x_{2N})$ associated with $\alpha$, with $\eta^1$ starting from $x_1$. Then the law of $\eta^1$ is the same as chordal $\SLE_{\kappa}$ in $(\HH;x_1,\infty)$ weighted by the local martingale (see~\cite[Lemma~4.8]{PeltolaWuGlobalMultipleSLEs})
\begin{equation}\label{eqn::PPF_mart}
	M_t(\LZ_{\alpha}) =\LZ_{\alpha}(\HH;W_t,g_t(x_2),\ldots,g_t(x_{2N})) \prod_{j=2}^{2N}g'_t(x_j)^{\mathfrak{b}}. 
\end{equation}


\begin{lemma}[Boundary perturbation]
\label{lem::multichordalSLE_bp} 
Fix $\kappa\in (0,4]$ and $N\ge 1$. 
Fix $\alpha=\{\{a_1, b_1\}, \ldots, \{a_N, b_N\}\}\in\LP_N$ and $2N$-polygon $(\Omega;\bs{x})$ with $\bs{x}=( x_1, \ldots, x_{2N})$. 
Suppose $K$ is a compact subset of $\overline{\Omega}$ such that $K$ has a positive distance from $\{x_1, \ldots, x_{2N}\}$ and $\{\bs{\eta}\cap K=\emptyset\}$ is not empty. Then, the $N$-chordal $\SLE_{\kappa}$ in the smaller domain $(\Omega\setminus K; \bs{x})$ is absolutely continuous with respect to that in $(\Omega; \bs{x})$, with Radon-Nikodym derivative
\begin{align}
	\label{eqn::multichordalSLESLEbp}
\frac{\ud \Pchord_{\alpha}(\Omega\setminus K; \bs{x})}{\ud\Pchord_{\alpha}(\Omega;\bs{x})}=	\frac{\LZ_{\alpha}(\Omega; \bs{x})}{\LZ_{\alpha}(\Omega\setminus K; \bs{x})}\one\left\{ \bs{\eta}\cap K=\emptyset \right\} \exp\left(\frac{\mathfrak{c}}{2} \blm\left(\Omega; \bs{\eta},K\right)\right), 
	\end{align}
where $\LZ_{\alpha}(\Omega\setminus K; \bs{x})$ is the product of pure partition functions of connected components of $\Omega\setminus K$. 
\end{lemma}
\begin{proof}
See~\cite[Proposition~3.4]{PeltolaWuGlobalMultipleSLEs} for the case when $\Omega\setminus K$ is simply connected. 
See~\cite[Lemma~2.4]{HuangWuCommutationRadial} for the general case. \end{proof}

\subsection{Half-watermelon SLE}
\label{subsec::halfwatermelon_SLE}

\paragraph*{Partition functions.}
Fix $\kappa\in (0,8)$ and $p\ge 1$. 
We first give the partition functions for half-$p$-watermelon $\SLE_{\kappa}$ in $\HH$. Define, for $\bs{x}=(x_1, \ldots, x_p)\in\LX_p$, 
\begin{align}
&\LZ_{\shuffle_p}(\HH;\bs{x})=\prod_{1\le i<\ell\le p}(x_{\ell}-x_i)^{\frac{2}{\kappa}}, \qquad\text{for } y=\infty, \\
&\LZfusion{p}(\HH;\bs{x},y)=\prod_{1\le i<\ell\le p}(x_{\ell}-x_i)^{\frac{2}{\kappa}}\times \prod_{j=1}^p(y-x_j)^{1-\frac{2}{\kappa}(p+2)}, \qquad\text{for }y\in (x_p, \infty). 
\label{eqn::LZfusion_def}
\end{align}
For $p=0$, we set its partition function equal to $1$. It is clear that $\LZ_{\shuffle_p}(\bs{x})$ is the special case of $\LZ_{\shuffle_p}^{(\lambda)}(\bs{x})$ defined in~\eqref{eqn::halfwatermelon_pf_drift} and it satisfies the BPZ system~\eqref{eqn::BPZ_chordal} with $n=p$ and $\lambda=0$. For general $(p+1)$-polygon $(\Omega; \bs{x}, y)$ with $\bs{x}=(x_1, \ldots, x_p)$, define
\begin{align} \label{eqn::halfwatermelon_COV}
\LZfusion{p}(\Omega;\bs{x},y)=\prod_{j=1}^p|\varphi'(x_j)|^{\mathfrak{b}}\times|\varphi'(y)|^{\frac{p(2p+4-\kappa)}{2\kappa}}\times\LZfusion{p}(\HH;\varphi(\bs{x}), \varphi(y)),
\end{align}
where $\varphi$ is any conformal map from $\Omega$ onto $\HH$ with $\varphi(x_1)<\cdots<\varphi(x_p)<\varphi(y)$. 

\paragraph*{Half-watermelon SLE.}
Fix $\kappa\in (0,8)$ and $p\ge 1$. Fix $(p+1)$-polygon $(\Omega; \bs{x}, y)$ with $\bs{x}=(x_1, \ldots, x_p)$. The half-$p$-watermelon $\SLE_{\kappa}$ in $(\Omega; \bs{x}, y)$ is the probability measure on $\bs{\gamma}=(\gamma^1, \ldots, \gamma^p)\in\chamberfusion{p}(\Omega; \bs{x}, y)$ which is characterized by the following recursion relation:
\begin{itemize}
\item $\gamma^1$ is chordal $\SLE_{\kappa}(2, \ldots, 2)$ in $\Omega$ from $x_1$ to $y$ with force points $(x_2, \ldots, x_p)$; 
\item given $\gamma^1$, the conditional law of $(\gamma^2, \ldots, \gamma^p)$ is half-$(p-1)$-watermelon $\SLE_{\kappa}$ in $(\Omega\setminus\gamma^1; x_2, \ldots, x_p; y)$. 
\end{itemize}
Such measure exists and is unique, see~\cite{MillerSheffieldIG2} and~\cite{HuangPeltolaWuMultiradialSLE}.
We denote the law of half-$p$-watermelon $\SLE_{\kappa}$ in $(\Omega;\bs{x},y)$ by $\Pfusion{p}(\Omega;\bs{x},y)$. For $p=1$, this means ordinary chordal $\SLE_{\kappa}$ from $x_1$ to $y$. 

Suppose that $\bs{\gamma}=(\gamma^1,\ldots,\gamma^p)$ is half-$p$-watermelon $\SLE_{\kappa}$ in $(\HH;x_1,\ldots,x_p;\infty)$ with $\gamma^j$ starting from $x_j$. Then the law of $\gamma^j$ is the same as chordal $\SLE_{\kappa}$ in $(\HH;x_j,\infty)$ weighted by the local martingale (see~\cite[Lemma~2.5]{HuangWuYangMultipleSLEsDysonBM})
\begin{equation}\label{eqn::halfwatermelon_mart}
	M_t(\LZ_{\shuffle_{p}})=\LZ_{\shuffle_{p}}(\HH;g_t(x_1),\ldots,g_t(x_{j-1}),W_t,g_t(x_{j+1}),\ldots,g_t(x_p)) \prod_{i\neq j} g'_t(x_i)^{\mathfrak{b}}. 
\end{equation}

\begin{lemma}[Boundary perturbation {\cite[Proposition~3.4]{HuangPeltolaWuMultiradialSLE}}]
\label{lem::halfwatermelonSLE_bp}
Fix $\kappa\in (0,4]$ and $p\ge 1$. Fix $(p+1)$-polygon $(\Omega;\bs{x},y)$ with $\bs{x}=(x_1, \ldots, x_p)$. Suppose $K$ is a compact subset of $\overline{\Omega}$ such that $K$ has a positive distance from $\{x_1, \ldots, x_p, y\}$ and $\Omega\setminus K$ is simply connected. Then, the half-$p$-watermelon $\SLE_{\kappa}$ in the smaller domain $(\Omega\setminus K; \bs{x}, y)$ is absolutely continuous with respect to that in $(\Omega; \bs{x},y)$, with Radon-Nikodym derivative 
\begin{align}
\frac{\ud \Pfusion{p}(\Omega\setminus K;\bs{x},y)}{\ud \Pfusion{p}(\Omega;\bs{x},y)}
=\frac{\LZfusion{p}(\Omega;\bs{x},y)}{\LZfusion{p}(\Omega\setminus K;\bs{x},y)}\one\{\bs{\gamma}\cap K=\emptyset\} \exp\left(\frac{\mathfrak{c}}{2}\blm(\Omega; \bs{\gamma}, K)\right).
\end{align}
\end{lemma}

\subsection{Mixture of half-watermelon and multi-chordal SLE}\label{subsec::LZmix}
For $n\ge 2$, recall that $\mathsf{P}_n=\{p\in\{0,1,\ldots, n\}: n-p\text{ is even}\}$ is defined in~\eqref{eqn::parityn_def}. 
For $p\in\mathsf{P}_n$, set $2N=n-p$. Fix $\alpha\in\LP_N$. Suppose $\bs{s}=(s_1, \ldots, s_p)$ is a subset of $\{1, \ldots, n\}$ such that $s_1<\cdots<s_p$. We write $\{v_1, \ldots, v_{2N}\}=\{1, \ldots, n\}\setminus\{s_1, \ldots, s_p\}$ such that $v_1<\cdots<v_{2N}$. Suppose $(\Omega; \bs{x}, y)$ with $\bs{x}=(x_1, \ldots, x_n)$ is nice $(n+1)$-polygon. 
\begin{itemize}
\item Suppose $\bs{\gamma}=(\gamma^1, \ldots, \gamma^p)$ is half-$p$-watermelon $\SLE_{\kappa}$ in $(\Omega; x_{s_1}, \ldots, x_{s_p}; y)$. We also write $\bs{\gamma}=\cup_{j=1}^p\gamma^j$ and denote its law by $\Pfusion{p}=\Pfusion{p}(\Omega; x_{s_1}, \ldots, x_{s_p}; y)$. Its partition function is given by 
$\LZfusion{p}(\Omega; x_{s_1}, \ldots, x_{s_p}; y)$. 
\item Suppose $\bs{\eta}=(\eta^1, \ldots, \eta^N)$ is $N$-chordal $\SLE_{\kappa}$ in $(\Omega; x_{v_1}, \ldots, x_{v_{2N}})$ associated with $\alpha$. We also write $\bs{\eta}=\cup_{j=1}^N\eta^j$ and denote its law by $\Pchord_{\alpha}=\Pchord_{\alpha}(\Omega; x_{v_1}, \ldots, x_{v_{2N}})$. Its partition function is given by 
$\LZ_{\alpha}(\Omega; x_{v_1}, \ldots, x_{v_{2N}})$.
\item Let $\Pfusion{p}\otimes\Pchord_{\alpha}$ be the measure under which $\bs{\gamma}\sim\Pfusion{p}$ and $\bs{\eta}\sim\Pchord_{\alpha}$ are independent. We say that the pair $(\bs{s}, \alpha)$ is allowable if $\{\bs{\gamma}\cap\bs{\eta}=\emptyset\}$ is not empty. When $(\bs{s},\alpha)$ is allowable, define 
\begin{align}\label{eqn::LZmix_def}
\begin{split}
\LZmix(\Omega; x_1, \ldots, x_n; y):=&\LZfusion{p}(\Omega; x_{s_1}, \ldots, x_{s_p}; y)\times \LZ_{\alpha}(\Omega; x_{v_1}, \ldots, x_{v_{2N}})\\
&\times\Efusion{p}\otimes\Echord_{\alpha}\left[\one\{\bs{\gamma}\cap\bs{\eta}=\emptyset\}\exp\left(\frac{\mathfrak{c}}{2}\blm(\Omega; \bs{\gamma}, \bs{\eta})\right)\right]. 
\end{split}
\end{align}
\end{itemize}

\begin{proposition}\label{prop::LZmix_BPZ}
Fix $\kappa\in (0,4]$ and $n\ge 2$. 
\begin{itemize}
\item The expectation in~\eqref{eqn::LZmix_def} is finite. 
We write $\LZmix(\bs{x})=\LZmix(\HH; \bs{x},\infty)$ for $\bs{x}\in\LX_n$.
\item Fix $p\in\mathsf{P}_n$ and recall that $\xi_p$ is defined in~\eqref{eqn::xi_p_def}. The collection 
\begin{align}\label{eqn::solutions_zerolambda}
\{\LZmix(\bs{x}): (\bs{s}, \alpha)\text{ is allowable}\}
\end{align}
is a basis for the space $\LSchord{n}^{(0;0;\xi_p)}$ and they are all positive. 
\end{itemize}
\end{proposition}
The proof of Proposition~\ref{prop::LZmix_BPZ} is split into five lemmas. 
\begin{lemma}\label{lem::LZmix_bound}
Assume the same setup as in~\eqref{eqn::LZmix_def}. 
We have
\begin{align}\label{eqn::LZmix_cascade1}
\LZmix(\Omega; x_1, \ldots, x_n; y)
	=&\LZfusion{p}(\Omega; x_{s_1}, \ldots, x_{s_p}; y)
	\Efusion{p}\left[\LZ_{\alpha}(\Omega\setminus\bs{\gamma}; x_{v_1}, \ldots, x_{v_{2N}})\right];\\
\label{eqn::LZmix_cascade2}
	\LZmix(\Omega; x_1, \ldots, x_n; y)
	=&\LZ_{\alpha}(\Omega; x_{v_1}, \ldots, x_{v_{2N}})
	\Echord_{\alpha}\left[\LZfusion{p}(\Omega\setminus\bs{\eta}; x_{s_1}, \ldots, x_{s_p}; y)\right]. 
\end{align}
In particular, we have
\begin{align}\label{eqn::LZmix_bound}
0<\LZmix(\Omega; \bs{x}; y)\le \LZfusion{p}(\Omega; x_{s_1}, \ldots, x_{s_p}; y)\times\prod_{\{a,b\}\in\alpha}\Poisson(\Omega; x_{v_a}, x_{v_b})^{\mathfrak{b}}.
\end{align}
\end{lemma}
\begin{proof}
Let us consider the pair $(\bs{\gamma}, \bs{\eta})$ in the construction above Proposition~\ref{prop::LZmix_BPZ}. The cases $p=0$ or $N=0$ follow directly from the empty-family conventions, so we assume $p,N\ge1$ below. By conformal covariance, it suffices to work in a bounded domain $\Omega$.
\begin{itemize}
\item From the boundary perturbation in Lemma~\ref{lem::multichordalSLE_bp}, the conditional law of $\bs{\eta}$ given $\bs{\gamma}$ weighted by 
\[\one\{\bs{\gamma}\cap\bs{\eta}=\emptyset\}\exp\left(\frac{\mathfrak{c}}{2}\blm(\Omega; \bs{\gamma}, \bs{\eta})\right)\]
is the same as $N$-chordal $\SLE_{\kappa}$ in the connected components of $\Omega\setminus\bs{\gamma}$. Therefore, 
\begin{equation}\label{eqn::etagivengamma}
\LZ_{\alpha}(\Omega; x_{v_1}, \ldots, x_{v_{2N}})\Echord_{\alpha}\left[\one\{\bs{\gamma}\cap\bs{\eta}=\emptyset\}\exp\left(\frac{\mathfrak{c}}{2}\blm(\Omega; \bs{\gamma}, \bs{\eta})\right)\cond \bs{\gamma}\right]=\LZ_{\alpha}(\Omega\setminus\bs{\gamma}; x_{v_1}, \ldots, x_{v_{2N}}). 
\end{equation}
This gives~\eqref{eqn::LZmix_cascade1} as desired. 
\item From the boundary perturbation in Lemma~\ref{lem::halfwatermelonSLE_bp}, the conditional law of $\bs{\gamma}$ given $\bs{\eta}$ weighted by 
\[\one\{\bs{\gamma}\cap\bs{\eta}=\emptyset\}\exp\left(\frac{\mathfrak{c}}{2}\blm(\Omega; \bs{\gamma}, \bs{\eta})\right)\]
is the same as half-$p$-watermelon $\SLE_{\kappa}$ in the connected component of $\Omega\setminus\bs{\eta}$ having $y$ on the boundary. Therefore, 
\begin{align*}
\LZfusion{p}(\Omega; x_{s_1}, \ldots, x_{s_p}; y)
	\Efusion{p}\left[\one\{\bs{\gamma}\cap\bs{\eta}=\emptyset\}
	\exp\left(\frac{\mathfrak{c}}{2}\blm(\Omega;\bs{\gamma},\bs{\eta})\right) \cond\bs{\eta}\right] 
	= \LZfusion{p}(\Omega\setminus\bs{\eta}; x_{s_1}, \ldots, x_{s_p}; y).
\end{align*}
This gives~\eqref{eqn::LZmix_cascade2} as desired. 
\end{itemize}
For the RHS of~\eqref{eqn::etagivengamma}, using~\eqref{eqn::PPF_PLB} and~\eqref{eqn::bPoisson_mono}, we have
\begin{align}\label{eqn::PLB_aux}
0<\LZ_{\alpha}(\Omega\setminus\bs{\gamma}; x_{v_1}, \ldots, x_{v_{2N}})
\le& \prod_{\{a,b\}\in\alpha}\Poisson(\Omega\setminus\bs{\gamma}; x_{v_a}, x_{v_b})^{\mathfrak{b}}
\le \prod_{\{a,b\}\in\alpha}\Poisson(\Omega; x_{v_a}, x_{v_b})^{\mathfrak{b}}.
\end{align}
Plugging into~\eqref{eqn::LZmix_cascade1}, we obtain~\eqref{eqn::LZmix_bound} as desired. 
\end{proof}

\begin{lemma}\label{lem::LZmix_ASY}
Assume the same setup as in~\eqref{eqn::LZmix_def}.
The function $\LZmix(\Omega; \bs{x}; y)$ has the following asymptotics: fix $j\in \{1, \ldots, n-1\}$ and $u\in (x_{j-1},x_{j+2})$, then
\begin{align}\label{eqn::LZmix_ASY}
\lim_{x_j, x_{j+1}\to u}\frac{\LZmix(\Omega; \bs{x}; y)}{\Poisson(\Omega; x_j, x_{j+1})^{\mathfrak{b}}}=
\begin{cases}
\LZmixhat(\Omega; \ddot{\bs{x}}_j; y), &\text{if \{j,j+1\} are paired in $\alpha$}, \\
0, &\text{if else},
\end{cases}
\end{align}
where by $\{j,j+1\}$ paired in $\alpha$, we mean that there is $i$ such that $v_i=j, v_{i+1}=j$ and $\{i,i+1\}\in\alpha$, and $\hat{\alpha}=\alpha/\{i,i+1\}$ and $\ddot{\bs{x}}_j=(x_1, \ldots, x_{j-1}, x_{j+2}, \ldots, x_n)$.
\end{lemma}
\begin{proof}
If $\{j,j+1\}$ are not paired in $\alpha$, the limit follows from~\eqref{eqn::LZmix_bound}. In the following, we assume that $\{j,j+1\}$ are paired in $\alpha$. In this case, there is $i$ such that $v_i=j, v_{i+1}=j+1$ and $\{i,i+1\}\in\alpha$. We denote $\hat{\alpha}=\alpha/\{i,i+1\}$. From~\eqref{eqn::LZmix_cascade1}, we have 
\begin{align}\label{eqn::LZmix_ASY_paired_aux1}
\LZmix(\Omega;\bs{x};y)=\LZfusion{p}(\Omega; x_{s_1}, \ldots, x_{s_p};y)\Efusion{p}\left[\LZ_{\alpha}(\Omega\setminus\bs{\gamma}; x_{v_1}, \ldots, x_{v_{2N}})\right].
\end{align}
Let us consider the asymptotics of $\LZ_{\alpha}(\Omega\setminus\bs{\gamma}; x_{v_1}, \ldots, x_{v_{2N}})$. 
Given $\bs{\gamma}$, from~\eqref{eqn::PPF_ASY}, we have almost surely, 
\begin{align*}
\lim_{x_j, x_{j+1}\to u}\frac{\LZ_{\alpha}(\Omega\setminus\bs{\gamma}; x_{v_1}, \ldots, x_{v_{2N}})}{\Poisson(\Omega; x_j, x_{j+1})^{\mathfrak{b}}}
=&\lim_{x_j, x_{j+1}\to u}\frac{\LZ_{\alpha}(\Omega\setminus\bs{\gamma}; x_{v_1}, \ldots, x_{v_{2N}})}{\Poisson(\Omega\setminus\bs{\gamma}; x_j, x_{j+1})^{\mathfrak{b}}}\frac{\Poisson(\Omega\setminus\bs{\gamma}; x_j, x_{j+1})^{\mathfrak{b}}}{\Poisson(\Omega; x_j, x_{j+1})^{\mathfrak{b}}}\\
=&\LZ_{\hat{\alpha}}(\Omega\setminus\bs{\gamma}; x_{v_1}, \ldots, x_{v_{i-1}}, x_{v_{i+2}}, \ldots, x_{v_{2N}}). 
\end{align*}
Using~\eqref{eqn::PLB_aux}, we have 
\begin{align*}
\frac{\LZ_{\alpha}(\Omega\setminus\bs{\gamma}; x_{v_1}, \ldots, x_{v_{2N}})}{\Poisson(\Omega; x_j, x_{j+1})^{\mathfrak{b}}}\le \prod_{\substack{\{a,b\}\in\alpha\\\{a,b\}\neq \{i,i+1\}}}\Poisson(\Omega; x_{v_a}, x_{v_b})^{\mathfrak{b}}. 
\end{align*}
Plugging the two observations into~\eqref{eqn::LZmix_ASY_paired_aux1}, bounded convergence theorem gives
\begin{align*}
\lim_{x_j, x_{j+1}\to u}\frac{\LZmix(\Omega; \bs{x}; y)}{\Poisson(\Omega; x_j, x_{j+1})^{\mathfrak{b}}}=&\LZfusion{p}(\Omega; x_{s_1}, \ldots, x_{s_p}; y)\Efusion{p}\left[\LZ_{\hat{\alpha}}(\Omega\setminus\bs{\gamma}; x_{v_1}, \ldots, x_{v_{i-1}}, x_{v_{i+2}}, \ldots, x_{v_{2N}})\right]\\
=&\LZmixhat(\Omega; \ddot{\bs{x}}_j; y),
\end{align*}
as desired. 
\end{proof}

\begin{lemma}\label{lem::LZmix_BPZ_s}
Assume the same setup as in Proposition~\ref{prop::LZmix_BPZ} and $p\ge 1$. The function $\LZmix(\bs{x})$ satisfies BPZ system~\eqref{eqn::BPZ_chordal} with $\lambda=0$ for $j\in\{s_1, \ldots, s_p\}$.  
\end{lemma}
\begin{proof}
Fix $j\in\{s_1,\ldots,s_p\}$ and write $j=s_q$. Let $\bs{\gamma}=(\gamma^1,\ldots,\gamma^p)$ be half-$p$-watermelon $\SLE_{\kappa}$ in $(\HH;x_{s_1},\ldots,x_{s_p};\infty)$. We parameterize the $\gamma^q$ starting from $x_j$ by half-plane capacity and let $g_t$ be the mapping-out function and $W_t$ its driving function. We have
\begin{align*}
	N_t	:=&\Efusion{p}\left[\LZ_{\alpha}(\HH\setminus\bs{\gamma};x_{v_1},\ldots,x_{v_{2N}})\cond\gamma^q[0,t]\right]\\
	=&\Efusion{p}\left[\LZ_{\alpha}(\HH\setminus g_t(\bs{\gamma});g_t(x_{v_1}),\ldots,g_t(x_{v_{2N}}))\cond\gamma^q[0,t]\right]
	\prod_{i=1}^{2N}g'_t(x_{v_i})^{\mathfrak{b}}
	\tag{due to~\eqref{eqn::PPF_COV}}\\
	=&\frac{\LZmix(g_t(x_1),\ldots,g_t(x_{j-1}),W_t,g_t(x_{j+1}),\ldots,g_t(x_n))}{\LZ_{\shuffle_{p}}(\HH;g_t(x_{s_1}),\ldots,g_t(x_{s_{q-1}}),W_t,g_t(x_{s_{q+1}}),\ldots,g_t(x_{s_p}))}
	\prod_{i=1}^{2N}g'_t(x_{v_i})^{\mathfrak{b}}
	\tag{due to~\eqref{eqn::LZmix_cascade1}}
\end{align*}
is a local martingale with respect to $\gamma^q$. Combining with~\eqref{eqn::halfwatermelon_mart}, we find
\begin{equation*}
	M_t(\LZmix)
	=\LZmix(g_t(x_1), \ldots, g_t(x_{j-1}), W_t, g_t(x_{j+1}), \ldots, g_t(x_{n}))
	\prod_{\ell\ne j} g'_t(x_\ell)^{\mathfrak{b}}
\end{equation*}
is a local martingale with respect to chordal $\SLE_{\kappa}$ in $(\HH;x_j,\infty)$. Corollary~\ref{cor::SLE_mart} implies that $\LZmix(\bs{x})$ satisfies the BPZ equation~\eqref{eqn::BPZ_chordal_j} for $j\in\{s_1,\ldots,s_p\}$ as desired.
\end{proof}

\begin{lemma}\label{lem::LZmix_BPZ_v}
Assume the same setup as in Proposition~\ref{prop::LZmix_BPZ} and $N\ge 1$. The function $\LZmix(\bs{x})$ satisfies BPZ system~\eqref{eqn::BPZ_chordal} with $\lambda=0$ for $j\in\{v_1, \ldots, v_{2N}\}$.  
\end{lemma}
\begin{proof}
Without loss of generality, we assume $j=v_1$. Let $\bs{\eta}=(\eta^1,\ldots,\eta^N)$ be $N$-chordal $\SLE_{\kappa}$ in $(\HH;x_{v_1},\ldots,x_{v_{2N}})$ associated with $\alpha$, with $\eta^1$ starting from $x_j$. We parameterize $\eta^1$ by half-plane capacity and let $g_t$ be the mapping-out function and $W_t$ its driving function. We have
\begin{align*}
	N_t
	:=&\Echord_{\alpha}\left[\LZ_{\shuffle_{p}}(\HH\setminus\bs{\eta};x_{s_1},\ldots,x_{s_p})\cond\eta^1[0,t]\right]\\
	=&\Echord_{\alpha}\left[\LZ_{\shuffle_{p}}(\HH\setminus g_t(\bs{\eta});g_t(x_{s_1}),\ldots,g_t(x_{s_p}))\cond\eta^1[0,t]\right]
	\prod_{i=1}^p g'_t(x_{s_i})^{\mathfrak{b}}
	\tag{due to~\eqref{eqn::halfwatermelon_COV}}\\
	=&\frac{\LZmix(g_t(x_1),\ldots,g_t(x_{j-1}),W_t,g_t(x_{j+1}),\ldots,g_t(x_n))}{\LZ_{\alpha}(\HH;W_t,g_t(x_{v_2}),\ldots,g_t(x_{v_{2N}}))}
	\prod_{i=1}^p g'_t(x_{s_i})^{\mathfrak{b}}
	\tag{due to~\eqref{eqn::LZmix_cascade2}}
\end{align*}
The process $N_t$ is a local martingale with respect to $\eta^1$. Combining with~\eqref{eqn::PPF_mart}, we find
\begin{equation*}
	M_t(\LZmix)
	=\LZmix(g_t(x_1), \ldots, g_t(x_{j-1}), W_t, g_t(x_{j+1}), \ldots, g_t(x_{n}))
	\prod_{\ell\ne j} g'_t(x_\ell)^{\mathfrak{b}}
\end{equation*}
is a local martingale with respect to chordal $\SLE_{\kappa}$ in $(\HH;x_j,\infty)$. Corollary~\ref{cor::SLE_mart} implies that $\LZmix(\bs{x})$ satisfies the BPZ equation~\eqref{eqn::BPZ_chordal_j} for $j=v_1$. By symmetry, we conclude that $\LZmix(\bs{x})$ satisfies the BPZ equation~\eqref{eqn::BPZ_chordal_j} for $j\in\{v_1,\ldots,v_{2N}\}$ as desired.
\end{proof}

\begin{lemma}\label{lem::allowable_pair}
Fix $n\ge 2$ and $p\in\mathsf{P}_n$. The number of allowlable pairs $(\bs{s},\alpha)$ in Proposition~\ref{prop::LZmix_BPZ} is given by $\mathsf{B}_{n,p}$ defined in~\eqref{eqn::bnp_def}. 
\end{lemma}
\begin{proof}
Denote by $\Cat_m=\frac{1}{m+1}\binom{2m}{m}$ the $m$th Catalan number.  
For $p\in\mathsf{P}_n$, we write $2N=n-p$.
Assume the same notation as in Proposition~\ref{prop::LZmix_BPZ}. 
The domain $\Omega\setminus\bs{\gamma}$ has $(p+1)$ connected components. In each connected component, the boundary points are connected pairwise by some of $\{\eta^1, \ldots, \eta^N\}$. Thus the number of different link patterns is given by 
\[\sum_{\substack{m_0, m_1, \ldots, m_p\ge 0\\ m_0+m_1+\cdots+m_p=N}}\Cat_{m_0}\Cat_{m_1}\cdots\Cat_{m_p}.\]
This is $p$-fold convolution of Catalan numbers, which is known to be $\binom{n}{N}-\binom{n}{N-1}$ which is the same as $\mathsf{B}_{n,p}$. 
\end{proof}

\begin{proof}[Proof of Proposition~\ref{prop::LZmix_BPZ}]
\begin{itemize}
\item The expectation in~\eqref{eqn::LZmix_def} is finite due to~\eqref{eqn::LZmix_bound} proved in Lemma~\ref{lem::LZmix_bound}. The same lemma gives strict positivity.
\item The partition function $\LZmix(\bs{x})$ satisfies BPZ system~\eqref{eqn::BPZ_chordal} with $\lambda=0$ due to Lemmas~\ref{lem::LZmix_BPZ_s} and~\ref{lem::LZmix_BPZ_v}.
\item It is clear from the construction that $\LZmix(\bs{x})$ is translation-invariant. Thus it satisfies the first Ward identity~\eqref{eqn::Ward} with $\nu=0$. 
\item Fix $p\in\mathsf{P}_n$ and $(\bs{s}, \alpha)$ allowable. The function $\LZmix(\bs{x})$ has the following scaling rule: for $R>0$ and $\bs{x}\in\LX_n$, 
\begin{align}\label{eqn::LZmix_scaling}
\LZmix(R\bs{x})=R^{\xi_{p}/\kappa}\LZmix(\bs{x}),
\end{align}
due to~\eqref{eqn::LZshuffle_def} and~\eqref{eqn::PPF_COV}. Note that
\[\frac{\xi_{p}}{\kappa}=\frac{p(p-1)}{\kappa}-(n-p)\mathfrak{b}.\]
Therefore, $\LZmix$ satisfies the second Ward identity~\eqref{eqn::Ward_2nd} with $\xi=\xi_p$. 
\item Let us check that solutions in~\eqref{eqn::solutions_zerolambda} are linearly independent. Suppose $\{\Cmix\}$ are constants such that 
\begin{align*}
\sum_{(\bs{s}, \alpha)\text{ allowable}}\Cmix\LZmix(\bs{x})=0, \qquad \text{for }\bs{x}\in\LX_n. 
\end{align*}
We apply~\eqref{eqn::LZmix_ASY} recursively and obtain $\Cmix=0$. This gives the desired linear independence. 
\end{itemize}
Collecting the above observations, the functions in~\eqref{eqn::solutions_zerolambda} are linearly independent solutions for $\LSchord{n}^{(0;0;\xi_p)}$. As the number of functions in~\eqref{eqn::solutions_zerolambda} is $\mathsf{B}_{n,p}$ due to Lemma~\ref{lem::allowable_pair} and the dimension for $\LSchord{n}^{(0;0;\xi_p)}$ is also $\mathsf{B}_{n,p}$ due to Theorem~\ref{thm::xi_p_decomposition}, the functions in~\eqref{eqn::solutions_zerolambda} form a basis for $\LSchord{n}^{(0;0;\xi_p)}$. 
\end{proof}

\begin{remark}\label{rem::E_polynomial_alternativeproof}
We present an alternative proof for Proposition~\ref{prop::widehatE} in this remark.
Fix $n\ge2$ and $\bs{x}\in\LX_n$, and write $\widehat E_\kappa:=\widehat E(\bs{x})$ to display its dependence on $\kappa$.

First, suppose that $\kappa\in(0,4)$.
For each $p\in\mathsf P_n$, the functions in~\eqref{eqn::solutions_zerolambda} are linearly independent elements of $\LSchord{n}^{(0;0)}$ and satisfy the scaling rule~\eqref{eqn::LZmix_scaling}.
Differentiating this rule gives
$\kappa\LE\LZmix=\xi_p\LZmix$.
By Lemma~\ref{lem::Ematrix_tilde}, their images under $M(\bs{x})^{-1}\zeta_{\bs{x}}$ are linearly independent eigenvectors of $\widehat E_\kappa$ with eigenvalue $\xi_p$.
The exponents $\xi_p$ are distinct in this range.
By Lemma~\ref{lem::allowable_pair}, there are $\mathsf B_{n,p}$ such vectors for each $p$, and Lemma~\ref{lem::chordal_Ward_Jordan_normal_form} gives
\[
    \sum_{p\in\mathsf P_n}\mathsf B_{n,p}
    =\binom{n}{\lfloor n/2\rfloor}
    =\dim\ker\widetilde T^{(0)}.
\]
Thus these vectors form an eigenbasis of $\ker\widetilde T^{(0)}$, proving both~\eqref{eqn::chordal_dilation_characteristic} and~\eqref{eqn::chordal_dilation_annihilating} for $\kappa\in(0,4)$.
\medbreak
Next, we extend both identities to every $\kappa>0$.
By Lemma~\ref{lem::Ematrix_tilde}, the space $\ker\widetilde T^{(0)}$ is independent of $\kappa$.
Choose a fixed basis of this space.
The explicit formula~\eqref{eqn::chordal_dilation_similarity} shows that the entries of $\widehat E_\kappa$ in this basis are affine in $\kappa$.
Since each $\xi_p$ is also affine in $\kappa$, the coefficients in $t$ of
\[
    \det(tI-\widehat E_\kappa)
    -\prod_{p\in\mathsf P_n}(t-\xi_p)^{\mathsf B_{n,p}}
\]
and the entries of $Q_\kappa(\widehat E_\kappa)$ are polynomials in $\kappa$.
They vanish for every $\kappa\in(0,4)$ and hence vanish identically.
This proves~\eqref{eqn::chordal_dilation_characteristic} and~\eqref{eqn::chordal_dilation_annihilating} for all $\kappa>0$.
\end{remark}

\begin{corollary}
Fix $n\ge 2$ and $\kappa\in (0,4]$. 
All solutions in $\LSchord{n}^{(0;0)}$ satisfy power-law bound. 
\end{corollary}
\begin{proof}
From~\eqref{eqn::LZmix_bound}, we have
\[\LZmix(\bs{x})\le \LZ_{\shuffle_p}(x_{s_1}, \ldots, x_{s_p})\times\prod_{\{a,b\}\in\alpha}|x_{v_b}-x_{v_a}|^{-2\mathfrak{b}}.\]
Thus, all functions in the basis $\cup_{p\in\mathsf{P}_n}\{\LZmix(\bs{x}): (\bs{s}, \alpha)\text{ is allowable}\}$ for $\LSchord{n}^{(0;0)}$ satisfy power-law bound. This gives the conclusion.
\end{proof}


\begin{thebibliography}{CDCH{\etalchar{+}}14}

\bibitem[AHSY26]{AngHoldenSunYu2023}
Morris Ang, Nina Holden, Xin Sun, and Pu~Yu.
\newblock Conformal welding of quantum disks and multiple {SLE}: the non-simple
  case.
\newblock {\em Probab. Theory Related Fields}, 194(1-2):717--777, 2026.

\bibitem[ALS20]{AruLupuSepulvedaFPSGFF}
Juhan Aru, Titus Lupu, and Avelio Sep\'ulveda.
\newblock First passage sets of the 2{D} continuum {G}aussian free field.
\newblock {\em Probab. Theory Related Fields}, 176(3-4):1303--1355, 2020.

\bibitem[ASW19]{AruSepulvedaWernerBTLStwoDGFF}
Juhan Aru, Avelio Sep\'ulveda, and Wendelin Werner.
\newblock On bounded-type thin local sets of the two-dimensional {G}aussian
  free field.
\newblock {\em J. Inst. Math. Jussieu}, 18(3):591--618, 2019.

\bibitem[BB03]{BauerBernardCFTSLE}
Michel Bauer and Denis Bernard.
\newblock Conformal field theories of stochastic {L}oewner evolutions.
\newblock {\em Comm. Math. Phys.}, 239(3):493--521, 2003.

\bibitem[BBK05]{BauerBernardKytolaMultipleSLE}
Michel Bauer, Denis Bernard, and Kalle Kyt{\"o}l{\"a}.
\newblock Multiple {S}chramm-{L}oewner evolutions and statistical mechanics
  martingales.
\newblock {\em J. Stat. Phys.}, 120(5-6):1125--1163, 2005.

\bibitem[BP26]{BerestyckiPowellGFFandLQG}
Nathana\"el Berestycki and Ellen Powell.
\newblock Gaussian free field and {L}iouville quantum gravity.
\newblock  {\em Cambridge Studies in Advanced Mathematics, Cambridge University Press, Cambridge,} volume 220, 2026.

\bibitem[BPW21]{BeffaraPeltolaWuUniqueness}
Vincent Beffara, Eveliina Peltola, and Hao Wu.
\newblock On the uniqueness of global multiple {SLE}s.
\newblock {\em Ann. Probab.}, 49(1):400--434, 2021.

\bibitem[BPZ84a]{BPZInfiniteConformalSymmetry2DQuantum}
A.~A. Belavin, A.~M. Polyakov, and A.~B. Zamolodchikov.
\newblock Infinite conformal symmetry in two-dimensional quantum field theory.
\newblock {\em Nuclear Phys. B}, 241(2):333--380, 1984.

\bibitem[BPZ84b]{BPZInfiniteConformalSymmetryCritical2D}
A.~A. Belavin, A.~M. Polyakov, and A.~B. Zamolodchikov.
\newblock Infinite conformal symmetry of critical fluctuations in two
  dimensions.
\newblock {\em J. Statist. Phys.}, 34(5-6):763--774, 1984.




\bibitem[CW26]{ChenWuSemiclassicalBPZ}
Mo Chen and Hao Wu.
\newblock Quasi-exponentials and semi-classical limits for the chordal {BPZ} system.
\newblock Preprint in arXiv:2609.38192, 2026.


\bibitem[DF84]{DotsenkoFateevCoulombGas}
Vl.~S. Dotsenko and V.~A. Fateev.
\newblock Conformal algebra and multipoint correlation functions in {2D} statistical models.
\newblock {\em Nuclear Phys. B}, 240(3):312--348, 1984.

\bibitem[DLMF]{NISTDLMF}
NIST Digital Library of Mathematical Functions.
\newblock Section~5.9(i): Gamma function, Hankel's loop integral.
\newblock \href{https://dlmf.nist.gov/5.9.E2}{https://dlmf.nist.gov/5.9.E2}.

\bibitem[Dub06]{DubedatEulerIntegralsCommutingSLEs}
Julien Dub\'edat.
\newblock Euler integrals for commuting {SLE}s.
\newblock {\em J. Stat. Phys.}, 123(6):1183--1218, 2006.

\bibitem[Dub07]{DubedatCommutationSLE}
Julien Dub{\'e}dat.
\newblock Commutation relations for {S}chramm-{L}oewner evolutions.
\newblock {\em Comm. Pure Appl. Math.}, 60(12):1792--1847, 2007.

\bibitem[Dub09]{DubedatSLEFreeField}
Julien Dub{\'e}dat.
\newblock S{LE} and the free field: partition functions and couplings.
\newblock {\em J. Amer. Math. Soc.}, 22(4):995--1054, 2009.


\bibitem[Dub15]{DubedatSLEVirasoroFusion}
Julien Dub\'edat.
\newblock {SLE} and {V}irasoro representations: fusion.
\newblock {\em Comm. Math. Phys.}, 336(2):761--809, 2015.

\bibitem[FK15a]{FloresKlebanPDE1}
Steven~M. Flores and Peter Kleban.
\newblock A solution space for a system of null-state partial differential
  equations: {P}art 1.
\newblock {\em Comm. Math. Phys.}, 333(1):389--434, 2015.

\bibitem[FK15b]{FloresKlebanPDE2}
Steven~M. Flores and Peter Kleban.
\newblock A solution space for a system of null-state partial differential
  equations: {P}art 2.
\newblock {\em Comm. Math. Phys.}, 333(1):435--481, 2015.

\bibitem[FK15c]{FloresKlebanPDE3}
Steven~M. Flores and Peter Kleban.
\newblock A solution space for a system of null-state partial differential
  equations: {P}art 3.
\newblock {\em Comm. Math. Phys.}, 333(2):597--667, 2015.

\bibitem[FLPW24]{FengLiuPeltolaWu2024}
Yu~Feng, Mingchang Liu, Eveliina Peltola, and Hao Wu.
\newblock Multiple {SLEs} for $\kappa\in (0,8)$: Coulomb gas integrals and pure
  partition functions.
\newblock Preprint in arXiv: 2406.06522, 2024.

\bibitem[FMTV00]{FelderMarkovTarasovVarchenkoKZ}
G.~Felder, Y.~Markov, V.~Tarasov, and A.~Varchenko.
\newblock Differential equations compatible with {KZ} equations.
\newblock {\em Math. Phys. Anal. Geom.}, 3(2):139--177, 2000.

\bibitem[FP20]{FloresPeltolaSchurWeyl}
Steven~M. Flores and Eveliina Peltola.
\newblock Higher-spin quantum and classical {Schur--Weyl} duality for
  {$\mathfrak{sl}_2$}.
\newblock Preprint in arXiv:2008.06038, 2020.

\bibitem[FPW24]{FengPeltolaWuConnectionProbaFKIsing}
Yu~Feng, Eveliina Peltola, and Hao Wu.
\newblock Connection probabilities of multiple {FK-Ising} interfaces.
\newblock {\em Probability Theory and Related Fields}, 189(1-2):281--367, March
  2024.

\bibitem[FSKZ17]{FloresSimmonsKlebanZiffCrossingProba}
S.~M. Flores, J.~J.~H. Simmons, P.~Kleban, and R.~M. Ziff.
\newblock A formula for crossing probabilities of critical systems inside
  polygons.
\newblock {\em J. Phys. A}, 50(6):064005, 91, 2017.

\bibitem[HLR25]{HaghighatLiuReshetikhinIrregularBlocks}
Babak Haghighat, Yihua Liu, and Nicolai Reshetikhin.
\newblock Flat connections from irregular conformal blocks.
\newblock {\em Comm. Math. Phys.}, 406(6):Article~138, 2025.

\bibitem[HPW26]{HuangPeltolaWuMultiradialSLE}
Chongzhi Huang, Eveliina Peltola, and Hao Wu.
\newblock Multiradial {SLE} with spiral: resampling property and boundary
  perturbation.
\newblock {\em Int. Math. Res. Not. IMRN}, 2026(15):rnag160, 2026.

\bibitem[HW26]{HuangWuCommutationRadial}
Chongzhi Huang and Hao Wu.
\newblock Classification of commutation relation for multi-radial {SLE}.
\newblock Preprint in arXiv:2609.19739, 2026.

\bibitem[HWY26]{HuangWuYangMultipleSLEsDysonBM}
Chongzhi Huang, Hao Wu, and Lu~Yang.
\newblock Multiple {SLE}s and {Dyson} {Brownian} motion: transition density and
  {Green's} function.
\newblock {\em J. Theor. Probab. to appear.} Preprint in arXiv:2311.06789, 2023.

\bibitem[IKV16]{IvanovKangVasilevSlitHoloStochasticGFF}
Georgy Ivanov, Nam-Gyu Kang, and Alexander Vasil'ev.
\newblock Slit holomorphic stochastic flows and {G}aussian free field.
\newblock {\em Complex Anal. Oper. Theory}, 10(7):1591--1617, 2016.

\bibitem[KL07]{KozdronLawlerMultipleSLEs}
Michael~J. Kozdron and Gregory~F. Lawler.
\newblock The configurational measure on mutually avoiding {SLE} paths.
\newblock In {\em Universality and renormalization}, volume~50 of {\em Fields
  Inst. Commun.}, pages 199--224. Amer. Math. Soc., Providence, RI, 2007.


\bibitem[KLPR25]{KarrilaLafayPeltolaRoussillonUSTBranches}
Alex Karrila, Augustin Lafay, Eveliina Peltola, and Julien Roussillon.
\newblock Planar {UST} branches and $c=-2$ degenerate boundary correlations.
\newblock {\em Probab. Math. Phys.}, 6:1443--1506, 2025.

\bibitem[KP16]{KytolaPeltolaPurePartitionFunctions}
Kalle Kyt\"ol\"a and Eveliina Peltola.
\newblock Pure partition functions of multiple {SLE}s.
\newblock {\em Comm. Math. Phys.}, 346(1):237--292, 2016.

\bibitem[KP20]{KytolaPeltolaConformalCovBoundaryCorrelation}
Kalle Kyt\"{o}l\"{a} and Eveliina Peltola.
\newblock Conformally covariant boundary correlation functions with a quantum
  group.
\newblock {\em J. Eur. Math. Soc. (JEMS)}, 22(1):55--118, 2020.

\bibitem[KPS26]{KarrilaPeltolaSchougGFFdegeneratelevelline}
Alex Karrila, Eveliina Peltola, and Lukas Schoug.
\newblock Level lines of the {G}aussian free field and $c=1$ degenerate
  conformal blocks. Preprint in arXiv:2606.22228, 2026.

\bibitem[KWW24]{KrusellWangWuCommutationRelation}
Ellen Krusell, Yilin Wang, and Hao Wu.
\newblock Commutation relations for two-sided radial {SLE}.
\newblock Preprint in arXiv:2405.07082, 2024.

\bibitem[KZ84]{KnizhnikZamolodchikov1984}
V.~G. Knizhnik and A.~B. Zamolodchikov.
\newblock Current algebra and {Wess--Zumino} model in two dimensions.
\newblock {\em Nucl. Phys. B}, 247(1):83--103, 1984.

\bibitem[Law09]{LawlerPartitionFunctionsSLE}
Gregory~F. Lawler.
\newblock Partition functions, loop measure, and versions of {SLE}.
\newblock {\em J. Stat. Phys.}, 134(5-6):813--837, 2009.

\bibitem[LPW25]{LiuPeltolaWuUST}
Mingchang Liu, Eveliina Peltola, and Hao Wu.
\newblock Uniform spanning tree in topological polygons, partition functions
  for {SLE}(8), and correlations in $c=-2$ logarithmic {CFT}.
\newblock {\em Ann. Probab.}, 53(1):23--78, 2025.


\bibitem[LW04]{LawlerWernerBrownianLoopsoup}
Gregory~F. Lawler and Wendelin Werner.
\newblock The {B}rownian loop soup.
\newblock {\em Probab. Theory Related Fields}, 128(4):565--588, 2004.


\bibitem[LW21]{LiuWuCrossingProbaMGFF}
Mingchang Liu and Hao Wu.
\newblock Scaling limits of crossing probabilities in metric graph {GFF}.
\newblock {\em Electron. J. Probab.}, 26:Paper No.~37, 46 pp., 2021.

\bibitem[MS16]{MillerSheffieldIG2}
Jason Miller and Scott Sheffield.
\newblock Imaginary geometry {II}: {R}eversibility of {$\operatorname{SLE}\sb
  \kappa(\rho\sb 1;\rho\sb 2)$} for {$\kappa\in(0,4)$}.
\newblock {\em Ann. Probab.}, 44(3):1647--1722, 2016.

\bibitem[Pel19]{PeltolaCFTSLE}
Eveliina Peltola.
\newblock Toward a conformal field theory for {S}chramm-{L}oewner evolutions.
\newblock {\em J. Math. Phys.}, 60(10):103305, 39, 2019.

\bibitem[Pel20]{PeltolaBasisPDE}
Eveliina Peltola.
\newblock Basis for solutions of the {B}enoit \& {S}aint-{A}ubin {PDE}s with
  particular asymptotics properties.
\newblock {\em Ann. Inst. Henri Poincar\'{e} D}, 7(1):1--73, 2020.

\bibitem[PW19]{PeltolaWuGlobalMultipleSLEs}
Eveliina Peltola and Hao Wu.
\newblock Global and local multiple {SLE}s for $\kappa \leq 4$ and connection
  probabilities for level lines of {GFF}.
\newblock {\em Comm. Math. Phys.}, 366(2):469--536, 2019.

\bibitem[PW23]{PeltolaWuCrossingProbaIsing}
Eveliina Peltola and Hao Wu.
\newblock Crossing probabilities of multiple {I}sing interfaces.
\newblock {\em Ann. Appl. Probab.}, 33(4):3169--3206, 2023.

\bibitem[Sch00]{SchrammFirstSLE}
Oded Schramm.
\newblock Scaling limits of loop-erased random walks and uniform spanning
  trees.
\newblock {\em Israel J. Math.}, 118:221--288, 2000.

\bibitem[She07]{SheffieldGFFMath}
Scott Sheffield.
\newblock Gaussian free fields for mathematicians.
\newblock {\em Probab. Theory Related Fields}, 139(3-4):521--541, 2007.



\bibitem[SS13]{SchrammSheffieldContinuumGFF}
Oded Schramm and Scott Sheffield.
\newblock A contour line of the continuum {G}aussian free field.
\newblock {\em Probab. Theory Related Fields}, 157(1-2):47--80, 2013.

\bibitem[Wu20]{WuHyperSLE}
Hao Wu.
\newblock Hypergeometric {SLE}: conformal {M}arkov characterization and
  applications.
\newblock {\em Comm. Math. Phys.}, 374(2):433--484, 2020.

\bibitem[WW17]{WangWuLevellinesGFFI}
Menglu Wang and Hao Wu.
\newblock Level lines of {G}aussian free field {I}: zero-boundary {GFF}.
\newblock {\em Stochastic Process. Appl.}, 127(4):1045--1124, 2017.

\bibitem[Zha24]{ZhanExistenceUniquenessMultipleSLE}
Dapeng Zhan.
\newblock Existence and uniqueness of nonsimple multiple {SLE}.
\newblock {\em J. Stat. Phys.}, 191(8):Paper No. 101, 15, 2024.

\bibitem[Zha25]{zhang2025multiplechordalslekappaquantum}
Jiaxin Zhang.
\newblock Multiple chordal {SLE}($\kappa$) and quantum Calogero-Moser system.
\newblock Preprint in arXiv:2505.16093, 2025.

\bibitem[ZZ17]{ZabrodinZotov2017}
A.~Zabrodin and A.~Zotov.
\newblock {KZ}--{Calogero} correspondence revisited.
\newblock {\em J. Phys. A: Math. Theor.}, 50(20):205202, 2017.

\end{thebibliography}

{\small
\newcommand{\etalchar}[1]{$^{#1}$}

}

\end{document}